\documentclass[letterpaper,11pt,reqno]{amsart}

\usepackage[margin=1in]{geometry}
\usepackage{amsmath}
\usepackage{amsthm}
\usepackage{amssymb}
\usepackage[foot]{amsaddr}
\usepackage[numeric,initials,nobysame]{amsrefs}
\usepackage{enumerate}
\usepackage{bbm}
\usepackage{xcolor}
\usepackage{comment}
\usepackage{hyperref}
\usepackage{soul}
\usepackage{graphicx}
\usepackage{fancyhdr}
\usepackage{mathtools}  
\hypersetup{colorlinks = true, linkcolor = blue, citecolor = blue}

 \numberwithin{equation}{section}
\newtheorem{theorem}{Theorem}[section]
\newtheorem{corollary}[theorem]{Corollary}
\newtheorem{proposition}[theorem]{Proposition}
\newtheorem{lemma}[theorem]{Lemma}

\theoremstyle{definition}
\newtheorem{defn}{Definition}[section]

\theoremstyle{remark}
\newtheorem{remark}{Remark}[section]

\newcommand{\RR}{\mathbb{R}}

\newcommand{\NN}{\mathbb{N}}

\newcommand{\EE}{\mathbf{E}}
\newcommand{\PP}{\mathbf{P}}

\newcommand{\clf}{\mathcal{F}}
\newcommand{\clp}{\mathcal{P}}
\newcommand{\cls}{\mathcal{S}}

\newcommand{\clg}{\mathcal{G}}
\newcommand{\clh}{\mathcal{H}}

\newcommand{\Om}{\Omega}
\newcommand{\om}{\omega}
\newcommand{\tsigma}{\tilde \sigma}

\newcommand{\btheta}{\boldsymbol \theta}
\newcommand{\bmu}{\boldsymbol \mu}
\newcommand{\brho}{\boldsymbol \rho}
\newcommand{\beps}{\boldsymbol \epsilon}
\newcommand{\br}{\boldsymbol r}
\newcommand{\be}{\boldsymbol e}
\newcommand{\bX}{\boldsymbol X}
\newcommand{\bF}{\boldsymbol \clf}
\newcommand{\bG}{\boldsymbol \clg}
\newcommand{\bH}{\boldsymbol \clh}

\newcommand{\po}{p}

\definecolor{amet}{rgb}{0.8, 0.2, 0.8}

\definecolor{ques}{rgb}{0.8, 0.2, 0.2}

\definecolor{nico}{rgb}{0.8, 0, 0.7}

\title{Approximating Quasi-Stationary Distributions with Interacting Reinforced Random Walks}
\date{\today}
\author{Amarjit Budhiraja}
\author{Nicolas Fraiman}
\author{Adam Waterbury}
\address{Department of Statistics and Operations Research\\ University of North Carolina at Chapel Hill\\ Hanes Hall, Chapel Hill, NC 27599, USA}
\email{amarjit@email.unc.edu, fraiman@email.unc.edu, atw02@live.unc.edu}



\begin{document}

\begin{abstract}
We propose two numerical schemes for approximating quasi-stationary distributions (QSD) of finite state Markov chains with absorbing states. Both schemes are described in terms of certain interacting chains in which the interaction is given
in terms of the total time occupation measure of all particles in the system and has the impact of reinforcing transitions, in an appropriate fashion, to states where the collection of particles has spent more time. The schemes can be viewed as combining the key features of the two basic simulation-based methods for approximating QSD originating from the works of Fleming and Viot (1979) and 
Aldous, Flannery and Palacios (1998), respectively. The key difference between the two schemes studied here is that in the first 
method one starts with $a(n)$ particles at time $0$ and number of particles stays constant over time whereas in the second method we start with one particle and at most one particle is added at each time instant in such a manner that there are $a(n)$
particles at time $n$. We prove almost sure convergence to the unique QSD and establish Central Limit Theorems for the two schemes under the key assumption that $a(n)=o(n)$. When $a(n)\sim n$, the fluctuation behavior is expected to be non-standard. Some exploratory numerical results are presented to illustrate the performance of the two approximation schemes.
\end{abstract}

\maketitle


\section{Introduction}

Markov processes with absorbing states  occur frequently in epidemiology \cite{bart60}, statistical physics \cite{van92}, and population biology  \cite{melvil}. Quasi-stationary distributions (QSD) are the basic mathematical object used to describe the long time behavior of such Markov processes on non-absorption events. Just as stationary distributions of ergodic Markov processes make the law of the Markov process, initialized at that distribution, invariant at all times, quasi-stationary distributions are probability measures that leave the conditional law of the Markov process, on the event of non-absorption, invariant. QSD have been widely studied since the pioneering work of Kolmogorov \cite{Kol38}, Yaglom \cite{Yag47} and Sevastyanov \cite{Sev51}, cf. \cites{melvil, pol, colpiemarsersan}.
Numerical computation of QSD is an important problem and the goal of this work is to investigate two related approximation schemes for QSD of finite state Markov chains. Specifically, we consider the following setting. 

Let $\Delta$ denote a finite set and consider a nonempty subset $\partial \Delta \subset \Delta$. Let $\Delta^o \doteq \Delta \setminus \partial \Delta$ and assume that $\Delta^o$ is nonempty. Let $\{Y_n\}$ be a Markov chain taking values in $\Delta$ with transition probability kernel $\{P_{x,y}\}_{x,y\in \Delta}$.
We denote by $\PP_{\nu}$ the probability measure under which $\{Y_n\}$ has initial distribution $\nu$, namely $\PP_{\nu}(Y_0 \in A) = \nu(A)$. If $\nu = \delta_x$ for some $x \in \Delta$, we write $\PP_x$ instead of $\PP_{\nu}$. We assume that $\{Y_n\}$ is absorbed upon entering $\partial \Delta$. In particular, for each $x \in \partial \Delta$,
\[
\PP_x(Y_1 \in \Delta^o) = 0.
\]
Without loss of generality, we assume that $\partial \Delta$ consists of a single point which we denote by $0$. Note that
\[
\PP_x(Y_1 = y) = P_{x,y} \quad\mbox{ for } x,y \in \Delta.
\]
A probability measure $\mu$ on $\Delta^o$ is a \emph{quasi-stationary distribution (QSD)} for the chain $\{Y_n\}$ if 
$$
\PP_\mu(Y_n =x \,|\, Y_n \in \Delta^o) = \mu(x)\text{, for all }x \in \Delta^o\text{ and }n \in \NN.
$$
We assume that $\Delta^o$ is an irreducible class of the Markov chain and that $P_{x,0}>0$ for some $x \in \Delta^o$. Under this irreducibility assumption on the chain it follows from  Perron-Frobenius theory that there is a unique QSD for $\{Y_n\}$ which we denote by $\theta_*$; see \cite{colpiemarsersan}*{Chapter 3}.
This probability measure  on $\Delta^o$ can be characterized as the normalized left eigenvector of the substochastic matrix $\{P_{x,y}\}_{x,y\in \Delta^o}$ associated with some eigenvalue $\lambda \in (0,1)$.
In particular, 
unlike invariant distributions for Markov processes, the QSD is characterized as a   solution of a {\em nonlinear} equation and thus presents harder numerical challenges. In general, numerical linear algebra methods become difficult when the underlying transition probability matrix is large or ill-conditioned. Thus, it is natural to explore simulation-based approaches. 

There have been two main simulation-based approaches for approximating QSD. These approaches originate from the works of 
Fleming and Viot \cite{FV79} and Aldous, Flannery and Palacios \cite{AFP88}, respectively.
In numerical schemes based on the ideas of Fleming and Viot (see \cites{burholmar, delmic2}),
one considers a collection  of  particles evolving independently according to the Markov chain with transition probability kernel $\{P_{x,y}\}$, and whenever a particle is absorbed it jumps instantly to the position of another particle selected at random. It is known that as both time and the number of particles tend to infinity, the empirical measure of the current positions of the particles converges almost surely to the unique QSD $\theta_*$ \cites{delmic2, ville1, benclo}. The method of Aldous et al.\ (see \cites{benclo, BGZ16}) approximates the QSD with the time occupation measure of a single particle that evolves according to the transition kernel $\{P_{x,y}\}$ between visits to $0$, and when it hits $0$ it jumps to a previously visited position with probability proportional to the time the chain spent at that position.

There has been substantial recent progress in analyzing the convergence rates of these algorithms. C\'erou, Delyon, Guyader and Rousset \cite{CDGR20} proved a Central Limit Theorem (CLT) for the law of Fleming-Viot particle systems at a given fixed time under very general assumptions. Lelievre, Pillaud-Vivien and Reygner \cite{LPR18} obtained an infinite-time version in the setting of finite space Markov chains, extending the ideas of Del Moral and Miclo \cite{MM03}. For the Aldous, Flannery and Palacios scheme, Bena\"{i}m and Cloez \cite{benclo} and, independently, Blanchet, Glynn and Zheng \cite{BGZ16} proved a Central Limit Theorem, see also \cites{delmic3, delmic4}. 

Each of the approximation methods discussed above has  benefits and shortcomings. Approximating with several particles helps the approximation better explore the space, particularly when the Markov process has metastable states where a scheme using a single particle can get stuck in place for long periods of time.
On the other hand, as the number of particles approach infinity, a Fleming-Viot approximation approaches the conditional law of the Markov chain (conditioned on non-extinction) at some finite-time instant rather than the QSD, and thus in order to obtain a good approximation for the QSD one needs to run the algorithm over long time periods. This can be computationally expensive and numerical experiments (see Section \ref{sec:numeric}) suggest that, with equivalent number of particle moves, a single particle reinforced random walk scheme of Aldous et al.\ performs better than a Fleming-Viot type scheme. This trade-off between the exploration of state space through multiple particles and the reinforcement of particle transition probabilities based on the time occupation measure motivates the present work, which studies two algorithms that
combine desirable features of both approximation schemes.

The two schemes that we study consider a collection of particles that, unlike Fleming-Viot approximations in which interactions occur through the current particle states, are governed by interactions with the time occupancy measures of {\em all} particles. Specifically, when a particle is absorbed, it instantly jumps to a state with probability proportional to the total time  spent at that position by all the particles in the collection.
The main difference between the two schemes considered in this work is that in the first scheme we start with $a(n)$ particles at time $0$
and the number of particles stays constant over time, whereas in the second scheme we add one particle at a time at some fixed rate so that there are $a(n)$ particles at time instant $n$. The approximation to the QSD is given by the combined (and suitably normalized) time occupation measure of all particles in the system. Our main results, Theorems  \ref{thm:mainrate}, \ref {thm:clt1}, \ref {thm:branchlln1}, and \ref{thm:branchclt1}
provide a.s.\ convergence to the QSD (i.e.\ strong law of large numbers) and 
central limit theorems for the two schemes. 
In Section \ref{sec:numeric} we present some exploratory numerical results on the performance of the two schemes and its comparison  with the Fleming-Viot and Aldous et al.\ methods.
The approach to the mathematical analysis of the two schemes is
inspired by the  methods used in \cites{benclo}, for the study of
the Aldous et al.\ scheme based on the path of a single particle, and draws from  techniques for establishing central limit results for general stochastic approximation schemes developed by  Delyon \cite{dey} and  Fort \cite{for}.

The theory of stochastic approximations (SA) has a long history, starting from the works of Robbins and Monro \cite{RM51}, and Kiefer and Wolfowitz \cite{KW52}. Since then, it has found many applications and has developed into a thriving area of research \cites{kus03,bor09,BMP12}.
In a typical stochastic approximation scheme one constructs a discrete time stochastic process whose continuous time interpolation over suitably slow decreasing time steps approaches the fixed point of a deterministic ordinary differential equation (ODE) as the continuous time parameter
$t$ approaches infinity.
One of the key differences, from this standard setting,  in the analysis of our first scheme
presented in Sections \ref{sec:rate}--\ref{sec:clt}, is that instead of a single stochastic approximation sequence, one needs to study an array, indexed by $n$, of sequences such that for each $n$
the sequence can be viewed as a SA algorithm targeting the QSD as the number of steps increase.
Our first result, Theorem \ref{thm:mainrate}, provides a strong law of large numbers for this
array as $n$ and the number of time steps become large. This result also provides an
almost sure upper bound on the rate of convergence which plays a crucial role later in the proof of the central limit theorem in Theorem \ref {thm:clt1}. In order to establish a suitable
rate of convergence, we introduce the notion of {\em pseudo-trajectory sequences} (see Definition \ref{defn:lambdapt}), which is inspired by the ideas of asymptotic pseudo-trajectories considered in \cites{benclo, ben3}, and is well-suited for array-type schemes such as those considered
here.

In Theorem \ref{thm:clt1}  we establish a central limit theorem for the array by
considering the $n$-th sequence run for $n$ time steps. The proof uses several ideas from \cite{for}*{Section 4}. In that work, the author considers a general SA algorithm 
which covers settings such as that of a 
 controlled Markov chain that evolves, conditional on the past history of the system, according to a stochastic kernel depending on the current approximation. The proofs of \cite{for} do not 
 easily extend to array settings of the form considered in the current work and it turns out that the rate of convergence in Theorem \ref{thm:mainrate} is key to suitably controlling the error arrays in the martingale decomposition of the SA sequences.
 One of the key requirements in the proofs is that $a(n)= o(n)$. Indeed, when $a(n)\sim n$, the errors due to the finite-time behavior of the collection of particles can accumulate and the fluctuation properties under the natural central limit scaling can be somewhat non-standard, see Remark \ref{rem:rem2} for a discussion of this point.


While in this work our focus is on approximating the QSD of a finite state Markov chain, the approach used to prove Theorems \ref{thm:mainrate} and \ref{thm:clt1} is more generally applicable. In particular, the notion of a pseudo-trajectory sequence introduced in Definition \ref{defn:lambdapt} should be useful for obtaining  bounds on the rate of convergence 
and establishing central limit theorems for other types of SA arrays.

The second numerical scheme is studied in Section \ref{sec:branch}.
In this method the approximation is initialized with a single particle and as time progresses particles are added to the system. At each step at most one particle is added and the number of particles at time $n$ is denoted by $a(n)$. Once more, the combined time occupation measure of 
all particles is used to approximate the QSD and 
 to replace particles that get absorbed. Since the number of particles changes over time, the analysis of error terms and the covariance structure gets more involved. In order to keep the presentation simple, here we restrict attention  to the case where $a(n)\sim n^{\zeta}$ for some
 $\zeta \in (0,1)$. In Theorem \ref{thm:branchlln1} we prove a.s. convergence of the approximation to the QSD and in Theorem \ref{thm:branchclt1} we provide a 
  central limit theorem for this approximation scheme.
  
  One of the challenges in constructing stochastic approximation schemes, with provable central limit fluctuations, for approximating QSD using a large number of particles is to carefully analyze the contribution to the variance and bias due to the finite-time behavior of the dynamics and to suitably calibrate the weights given to particle states as time increases.
  Specifically, for the two algorithms studied in the current work, we find that in comparison 
  to the single particle SA schemes studied in \cites{benclo, BGZ16}, one needs to place higher weights on  particle states at later time instants in order to suitably counterbalance the variability due to the finite-time behavior of the chains. This point is discussed further in Remark 
  \ref{rem:rem1}, however a precise understanding of relationships between size of SA arrays
  and time step sizes, for central limit results to hold, remains to be fully developed.
  Finally, we remark that in this work we consider SA arrays and sequences  with time steps
  of order $1/n$. Convergence and fluctuation results for interacting particle schemes with more general time steps satisfying appropriate decay conditions will be a topic for future study.

  We now describe the two schemes in some detail.

\subsection{Description of the algorithms}\label{sec:algdescription}

We denote by $\mathcal{P}(\Delta^o)$ the space of probability measures on  $\Delta^o$. Letting $d\doteq |\Delta^o|$, $\mathcal{P}(\Delta^o)$ can be identified with the $(d-1)$-dimensional simplex
$$\cls \doteq \Big\{x \in \RR_+^d: \sum_{i=1}^d x_i=1\Big\}.$$
For notational convenience, elements of $\Delta^o$ will be labeled as $\{1, 2, \ldots d\}$.
 For each $\nu \in \mathcal{P}(\Delta^o)$, we consider a transition probability kernel $K[\nu]$ on $\Delta^o$ given by 
\begin{equation}\label{eq:kernel}
K[\nu]_{x,y} \doteq P_{x,y} + P_{x,0}\,\nu(y) \quad\mbox{ for } x,y \in \Delta^o.
\end{equation}
For each $\nu \in \mathcal{P}(\Delta^o)$, the Markov chain associated with the transition probability kernel $K[\nu]$ is irreducible, and we denote the corresponding unique invariant distribution  by $\pi(\nu)$. Define
$$h(\nu)  \doteq \pi(\nu) - \nu \quad\mbox{ for } \nu \in \mathcal{P}(\Delta^o).$$
It is well known that $h: \cls \to T \cls \doteq \{x\in \RR^d: \sum_{i=1}^d x_i =0\}$ is a smooth function 
and the Jacobian matrix $\nabla h(\theta_*)$ is a Hurwitz matrix, in particular there is some $L > 0$ such that the eigenvalues of $\nabla h(\theta_*)$ 
have their real parts bounded above by $-L$; see \cite{benclo}*{Corollary 2.3}.

The  approximation algorithms described below are given in terms of a certain step size sequence  denoted by $\{\gamma_n\}_{n=1}^{\infty}$, and we assume that for some $\gamma_* > 0$, 
 we have 
\begin{equation}\label{eq:stepsize}
\gamma_{k+1} \doteq \frac{\gamma_*}{k+N_*},\; k \in \NN_0,
\end{equation}
where $N_* = \lfloor \gamma_*\rfloor +1$.
Let $\{a(n)\}_{n\in\NN}$ be a  sequence of positive integers increasing to $\infty$.

\noindent {\bf Algorithm I.} For fixed $x_0 \in \Delta^o$, we consider a collection  $\{X^{i,n}_{k}\}_{1\leq i\leq a(n), n \in \NN, k \in \NN_0}$ of $\Delta^o$--valued random variables,  an array $\{\theta_{k}^n\}_{n \in \NN, k \in \NN_0}$ of $\clp(\Delta^o)$--valued random measures, and a collection $\{\clf^n_k\}_{n \in \NN, k \in \NN_0}$ of $\sigma$-fields  given on some probability space $(\Om, \clf, \PP)$,  defined recursively as follows. For $n \in \NN$ and $k=0$, let
$$
X^{i,n}_0 \doteq x_0, \;\; 1 \leq i \leq a(n),
\quad
\mathcal{F}_0^n \doteq \{\emptyset, \Omega\}
\quad\text{and}\quad
\theta_{0}^n\doteq \delta_{x_0}.
$$
Having defined the above random variables and $\sigma$-fields for   some $k \in \NN_0$ and all $n\in \NN$, define, for each $n \in \NN$ and $1 \leq i \leq a(n)$
\begin{equation}\label{eq:eq242}
\PP\left(X^{1,n}_{k+1} = y_1, \dots, X^{a(n),n}_{k+1} = y_{a(n)} \;\middle|\;  \mathcal{F}_k^n \right) = \prod_{i=1}^{a(n)} K[\theta_k^n]_{x_i,y_i}.
\end{equation}
on the set  $\{X_{k}^{1,n} = x_1,\dots,X_{k}^{a(n),n} = x_{a(n)}\}$. The filtration is extended as
 $$\mathcal{F}_{k+1}^n \doteq \mathcal{F}_{k}^n \vee \sigma\Big(X_{k+1}^{1,n},\dots,X_{k+1}^{a(n),n}\Big)$$
and the new estimate of the QSD is given by
\begin{equation}\label{eq:stochalg}
\theta^n_{k+1} \doteq (1 - \gamma_{k+1})\theta^n_k  + \gamma_{k+1} \frac{1}{a(n)} \sum_{i=1}^{a(n)} \delta_{X^{i,n}_{k+1}}.
\end{equation}
We are interested in the asymptotic behavior of $\theta_n \doteq \theta_n^n$. In order to write $\{\theta^n_k\}_{n=1}^{\infty}$ as a stochastic approximation (SA) algorithm, for $1\le i \le a(n)$ and $k \in \NN_0$, let
\begin{equation}\label{eq:eq554S}
\epsilon^n_{k+1} \doteq \frac{1}{a(n)} \sum_{i=1}^{a(n)} \epsilon^{i,n}_{k+1}
\quad\text{where}\quad
\epsilon^{i,n}_{k+1}\doteq \delta_{X^{i,n}_{k+1}} - \pi(\theta^n_{k}).
\end{equation}
Then the evolution for the QSD approximation  $\theta^n_k$  from (\ref{eq:stochalg}) can be rewritten as 
\begin{equation}\label{eq:stochalg1}
\theta^n_{k+1} = \theta^n_k + \gamma_{k+1} \big(h(\theta^n_k) + \epsilon^n_{k+1}\big).
\end{equation}

\noindent {\bf Algorithm II.} In order to distinguish from the notation used for the 
first scheme, we will use  bold symbols to denote some key quantities with slightly different definitions than those in the definition of the first algorithm. In this method, rather than starting with $a(n)$ particles, we will start with $1$
particle at time $0$ and add particles over time. This algorithm is therefore described by a
single sequence of random variables rather than by an array. In particular $a(n)$ will denote the number of particles at the $n$-th time step rather than the number of particles in the $n$-th sequence in the array. 
Here $\{a(n)\}$ is a non-decreasing sequence of integers satisfying the following:
\begin{enumerate}
\item $a(0) = 1$.
\item For each $n \in \NN$, $a(n+1) - a(n) \leq 1$.
\item The number of particles at instant $n$ is $a(n)$ and there is some $\zeta \in (0,1)$ such that the $n$-th particle is added at time step $b(n)= \lfloor n^{1/\zeta}\rfloor$.
\end{enumerate}
The above properties in particular say that $a(n) \sim n^{\zeta}$ and the sequence $\{b(n)\}$ satisfies
 $b(1) \doteq 0$, and 
\begin{equation}\label{eq:kappadef}
b(n) \doteq \inf\{ m > b(n-1) : a(m) = a(n-1) + 1 \}.
\end{equation}
We will also need a $\{1, \ldots , a(n)\}$ valued random variable $\iota_n$ which
will tell us where to add the new particle at time instant $n+1$ if $a(n+1)= a(n)+1$. The precise
manner in which this particle is added is not important and one can use an arbitrary non-anticipative rule for doing so. More precisely, the scheme is given as follows.

Consider a collection $\{\bX^i_n\}_{1\leq i\leq a(n+1), n \in\NN_0}$ of $\Delta^o$--valued random variables, a sequence $\{\iota_n\}_{n \in \NN_0}$ of random variables with $\iota_n$ taking values in $\{1, \ldots, a(n)\}$, a sequence $\{\btheta_n\}_{n \in \NN_0}$ of $\mathcal{P}(\Delta^o)$--valued random measures, and a sequence $\{\bF_n\}_{n \in \NN_0}$ of $\sigma$-fields given on some probability space $(\Omega, \mathcal{F}, \PP)$, recursively defined as follows. We let 
$$
\bX_0^1 \doteq x_0,\quad 
\iota_0\doteq1, \quad 
\bF_0 \doteq \{\emptyset, \Omega\}
\quad 
\text{and}\quad \btheta_0 \doteq \delta_{x_0}. 
$$
Note $a(0)=a(1)=1$. We  let $\iota_1 \doteq 1$. Having defined $\{\bX^i_n\}_{1\leq i\leq a(n+1)}$,  $\btheta_n$, $\iota_{n+1}$ and $\bF_n$, define the elements for the next step as follows:
{\setlength{\leftmargini}{1em}
\begin{itemize}
\item Conditioned on $\bF_n$, particles evolve according to the kernel $K[\btheta_n]$ independently. In particular, if no branching occurs, namely $a(n+2)=a(n+1)$, then
$$
\PP\left( \bX^1_{n+1}=y_1,\dots, \bX^{a(n+2)}_{n+1}=y_{a(n+2)} \;\middle|\; \bF_n\right) 
= \prod_{i=1}^{a(n+1)}K[\btheta_n]_{x_i,y_i}
$$
on the set $\{\bX^1_n = x_1,\dots,\bX^{a(n+1)}_n = x_{a(n+1)}\}$.
On the other hand, if a branching event occurs, i.e.  $a(n+2)=a(n+1)+1$,  on the set $\{\bX^1_n = x_1,\dots,
\bX^{a(n+1)}_n = x_{a(n+1)} \mbox{ and } \iota_{n+1}=\ell\}$,  the particle with index $\ell$
will replicate, the new particle be given the index $a(n+2)$, and
$$
\PP\left( \bX^1_{n+1}=y_1,\dots, \bX^{a(n+2)}_{n+1}=y_{a(n+2)} \;\middle|\; \bF_n \right) 
= \left(\prod_{i=1}^{a(n+1)}K[\btheta_n]_{x_i,y_i}\right) K[\btheta_n]_{x_{\ell},y_{a(n+2)}}.
$$
\item With $\bG_{n+1}\doteq \bF_n \vee \sigma\{\bX^1_{n+1}, \ldots , \bX^{a(n+2)}_{n+1}\}$
and $\bH_{n+1}$ an arbitrary $\sigma$-field independent of $\bG_{n+1}$, let 
 $\iota_{n+2}$ be an arbitrary $\bG_{n+1}\vee \bH_{n+1}$ measurable random variable with values in $\{1,\dots,a(n+2)\}$.
\item Let $\bF_{n+1} = \bF_n \vee \sigma\big(\bX_{n+1}^1\dots,\bX_{n+1}^{a(n+2)}\big) \vee \sigma(\iota_{n+2})$.
\item Finally, let the new QSD estimate be
 \begin{equation}\label{eq:branchalg1}
\btheta_{n+1} \doteq (1 - \gamma_{n+1})\btheta_n + \gamma_{n+1} \frac{1}{a(n+1)} \sum_{i=1}^{a(n+1)}\delta_{\bX^i_{n+1}}.
\end{equation}
\end{itemize}}
Note that by construction, $\btheta_n$, $\{\bX^i_n\}_{1\leq i\leq a(n+1)}$, and $\iota_{n+1}$ are $\bF_n$ measurable for all $n \in \NN_0$. Also note that $\iota_{n+1}$ plays a role in the definition of the measure $\btheta_{n+2}$ only when $a(n+2)=a(n+1)+1$.

In order to write $\btheta_n$ as a SA algorithm, we define, for $1 \leq i \leq a(n+1)$ and $n \in \NN_0$,
\begin{equation}\label{eq:eq143r}
\beps_{n+1} \doteq \frac{1}{a(n+1)} \sum_{i=1}^{a(n+1)} \beps^i_{n+1}
\quad\text{where}\quad
\beps^i_{n+1} \doteq \delta_{\bX^i_{n+1}} - \pi(\btheta_n).
\end{equation}
Then the evolution equation in (\ref{eq:branchalg1}) can be rewritten as
\begin{equation}\label{eq:branchalg2}
\btheta_{n+1} = \btheta_n + \gamma_{n+1}( h(\btheta_n) + \beps_{n+1}).
\end{equation}

\subsection{Statement of results}

We first describe the results for Algorithm I, namely the algorithm given by \eqref{eq:stochalg1}. The following theorem proves that the approximation scheme converges a.s. to the unique QSD $\theta_*$ and provides an a.s. upper bound on the rate at which $\{{\theta}^{n}_k\}$ converges to $\theta_*$. 

\begin{theorem}\label{thm:mainrate}
As $n \to \infty$, ${\theta}^{n}_n\to \theta_*$ almost surely. Furthermore,
for each $\po \in (0,1)$, there is a $\beta > 0$, such that for $\PP$-a.e. $\om$, there is a  $n_0\equiv n_0(\om) \in \NN$ such that for all  $n \geq n_0$ and  $n^{\po} \le k \leq n$,
$$
\| \theta^{n}_k - \theta_*\| \leq k^{-\beta}.
$$
\end{theorem}

Theorem \ref{thm:mainrate} is proved in Section \ref{sec:rate}. The next theorem provides
a central limit theorem for the sequence $\{\theta_n^n\}$.  Define the sequence $\{\sigma_n\}_{n=1}^{\infty}$ by 
\begin{equation}\label{eq:signdef}
\sigma_n \doteq \sqrt{a(n)/\gamma_n}.
\end{equation}
This sequence will give the scaling factor in the CLT.
The covariance matrix for the limiting Gaussian distribution is given in terms of a  nonnegative definite matrix $U_*$ which is introduced later in \eqref{eq:ustar}.
For the CLT we will need additional conditions on the step sizes and the number of particles
in the system.
\begin{theorem}\label{thm:clt1}
Suppose that $a(n)/n\to 0$ as $n\to \infty$ and $\gamma_* > L^{-1}$.
Then, as $n \to \infty$, 
$$
\sigma_n ( \theta^n_n - \theta_*) \stackrel{\mathcal{L}}{\to} \mathcal{N}(0, V),
$$
where $V$ is the solution to the Lyapunov equation 
\begin{equation}\label{eq:eq247}
U_* + \nabla h(\theta_*)V + V \nabla h(\theta_*)^T + \gamma_*^{-1}V = 0,
\end{equation}
 $U_*$ is the nonnegative definite matrix given by \eqref{eq:ustar}, and $\stackrel{\mathcal{L}}{\to}$ denotes convergence in distribution.
\end{theorem}

Theorem \ref{thm:clt1}  is proved in Section \ref{sec:clt} by combining results from Sections 
\ref{sec:rate} and \ref{sec:errors}. 

The following are our main results for Algorithm II given by \eqref{eq:branchalg2}.
The first result proves the a.s. convergence of the scheme. This time we don't provide convergence rates as it turns out that unlike the proof of Theorem \ref{thm:clt1}, the proof of Theorem \ref{thm:branchclt1} does not require the use of convergence rates.
\begin{theorem}\label{thm:branchlln1}
As $n \to \infty$, $\btheta_n \to \theta_*$ almost surely.
\end{theorem}
Theorem \ref{thm:branchlln1} is proved in Section \ref{sec:branch}.

Our final result gives a CLT for Algorithm II. Proof is given in Section \ref{sec:cltalg2}.
\begin{theorem}\label{thm:branchclt1}
Suppose that $a(n)/n\to 0$ as $n\to \infty$ and $\gamma_* > L^{-1}$.
Then, as $n \to \infty$, 
$$
\sigma_n ( \btheta_n - \theta_*) \stackrel{\mathcal{L}}{\to} \mathcal{N}(0, V),
$$
where $V$ is the solution to the Lyapunov equation
$$
U_* + \nabla h(\theta_*) V + V\nabla h(\theta_*)^T + (1+ \zeta)\gamma_*^{-1} V = 0,
$$
and $U_*$ is the nonnegative definite matrix given by \eqref{eq:ustar}.
\end{theorem}

\begin{remark}\label{rem:rem1}
The condition $\gamma_* > L^{-1}$ is used in an important way in the proofs of CLT in Theorems 
\ref{thm:clt1} and 
\ref{thm:branchclt1}. We note that the CLT for a single particle scheme given in \cite{benclo}
allows for any $\gamma_* > (2L)^{-1}$. Thus, we find that for CLT results here we need larger step sizes
than those allowed for the single particle scheme. Larger step sizes correspond to placing higher weights
on particle states at later time instants. This need for suitably emphasizing later time points more
arises in order to counterbalance the variability due to the large number of particles at any fixed time instant.
\end{remark}

\begin{remark}\label{rem:rem2}
Recall that for the CLT results we require that $a(n)=o(n)$. This condition is crucial in obtaining the estimates on the discrepancy array (resp. sequence) given in Lemma \ref{lem:rholem1} (resp. Proposition \ref{prop:rhordiff}). 
As noted in the Introduction, when $a(n)\sim n$ one expects nonstandard fluctuation behavior under the natural CLT scaling. To see this, consider the elementary setting of a collection of
i.i.d.\ Markov chains. Specifically, let
  $\{X^n_m, \; m\in\NN_0\}_{n\in \NN}$ be a collection of i.i.d.\ irreducible Markov chains on $\Delta^o$ with transition probability kernel $K_0$ and stationary distribution $\theta_*$. For simplicity suppose that  $X^n_0 = x_0$ for all $n\in \NN$, for some $x_0 \in \Delta^o$.  Define
\begin{equation}\label{eq:indepthetan}
\theta^{n}_m =  \frac{1}{m a(n)}\sum_{k=1}^{m}  \sum_{i=1}^{a(n)}   ( \delta_{X^i_k} - \theta_*), \;\; m, n \in \NN.
\end{equation}
It is straightforward to show that if 
 $a(n) = o(n)$, then, as $n \to \infty$,
$$
\sqrt{a(n) n}(\theta^n_n - \theta_*) \stackrel{\mathcal{L}}{\to} \mathcal{N}(0,U_*).
$$
where $U_*$ is defined in a similar manner as in \eqref{eq:ustar}.
However when 
$a(n) \sim a_*n$ for some $a_* \in (0,\infty)$,  a different behavior emerges, and in fact
the asymptotic mean of the scaled differences $\sqrt{a(n) n}(\theta^n_n - \theta_*)$ is nonzero
as $n \to \infty$.  In particular, one can easily see that
$$
\sqrt{a(n)n }(\theta^n_n - \theta_*) \stackrel{\mathcal{L}}{\to} \mathcal{N}(\alpha_*,U_*),
$$
where 
$$
\alpha_* \doteq a_*[ (K_0Q_0)_{x_0,\cdot} - \theta_*(K_0Q_0)],\;\; \theta_*(K_0Q_0) \doteq \sum_{x\in \Delta^o} (K_0Q_0)_{x,\cdot} \theta_*(x),
$$
where $Q_0$ is defined as in \eqref{eq:qnudef} on replacing on its right side $K[\nu]$ with $K_0$
and $\Pi(\nu)$ with the $d\times d$ matrix $[\theta_*, \theta _*, \cdots]^T$.
For the stochastic approximation algorithms considered in this work, in order to study the limit behavior when 
 $a(n) \sim  n$ one will  need to carefully analyze the limiting behavior of state dependency in the (appropriately scaled) discrepancy array/sequence, which describes the deviations of the linearized evolution from the underlying stochastic approximation algorithm (see discussion in Section \ref{sec:decandlin} below) in order to identify
 the  asymptotic `drift' in the Gaussian limit. This study will be taken up elsewhere.
\end{remark}
 
\begin{remark}\label{rem:rem3}
Since in Algorithm II one particle is added at a time and at time $k$ there are $a(k)$ particles, a more natural
 choice of the central limit scaling than   $\sigma_n$  is given by the sequence
$$
\beta_n \doteq \left( \gamma_n \frac{1}{n} \sum\limits_{k=1}^{n} \frac{1}{a(k)}\right)^{-1/2}.
$$
From Theorem \ref{thm:branchclt1} it follows immediately that
$$
{\beta}_n ( \btheta_n - \theta_*) \stackrel{\mathcal{L}}{\to} \mathcal{N}(0,\tilde{V}),
$$
where $\tilde{V}$ is the unique solution to the Lyapunov equation
\begin{equation}\label{eq:noteasy10}
(1 - \zeta)U_* + (1+\zeta)\gamma_*^{-1} \tilde{V} + \nabla h(\theta_*) \tilde{V}+ \tilde{V} \nabla h(\theta_*)^T = 0.
\end{equation}
On the other hand, recall that for Algorithm I the central limit theorem takes the form
$$
\sigma_n(\theta_n^n - \theta_*) \stackrel{\mathcal{L}}{\to} \mathcal{N}(0, V),
$$
where $V$ is the solution to \eqref{eq:eq247}.
The quantities $V$ and $\tilde V$ can be viewed as the `per-particle' asymptotic covariance matrices for the two numerical schemes.
\end{remark}

\subsection{Decomposition and linearization}
\label{sec:decandlin}
One of the key ingredients in the proofs is the following explicit representation of the solution of Poisson's equation 
associated with the transition probability kernel $K[\cdot]$.
For a proof, see \cite{ben3}*{Lemma 5.1}.
\begin{lemma}\label{lem:qsmooth}
For each $\nu \in \mathcal{P}(\Delta^o$), let $\Pi(\nu)$ be the $d \times d$ matrix with entries $\Pi(\nu)_{x,y} = \pi(\nu)_{y}$.  Then for each $\nu \in \mathcal{P}(\Delta^o)$, the matrix 
\begin{equation}\label{eq:qnudef}
Q[\nu] \doteq - \int_0^{\infty} \left[\exp\left( t(K[\nu] - I)\right) - \Pi(\nu)\right] \,dt,
\end{equation}
is well-defined and the map $\nu \mapsto Q[\nu]$ is continuously differentiable. Furthermore, 
\begin{equation}\label{eq:pois}
(I - K[\nu])Q[\nu] = Q[\nu](I - K[\nu]) = I - \Pi(\nu).
\end{equation}
\end{lemma}
Using the above result, and following \cites{ben3,benclo}, we decompose the noise in Algorithm I given in 
\eqref{eq:eq554S} in the following manner: for each $n \in \NN$ and $1 \leq i \leq a(n)$, write
\begin{equation*}
\begin{aligned}
\epsilon_{k+1}^{i,n} &= \delta_{X^{i,n}_{k+1}} - \pi(\theta^n_{k}) 
= (X^{i,n}_{k+1})^T (I-\Pi(\theta^n_k)) \\
&= (X^{i,n}_{k+1})^T (I - K[\theta^n_k])Q[\theta^n_k] 
= Q[\theta_k^{n}]_{X_{k+1}^{i,n},\cdot} - (K[\theta_k^{n}]Q[\theta_k^{n}])_{X_{k+1}^{i,n},\cdot}.
\end{aligned}
\end{equation*}
Then we can write $\epsilon_{k+1}^{i,n} = e^{i,n}_{k+1} + r^{i,n}_{k+1}$ where
\begin{equation}\label{eq:er-def}
\begin{aligned}
e^{i,n}_{k+1} &\doteq Q[\theta_k^{n}]_{X_{k+1}^{i,n},\cdot} - (K[\theta^{n}_{k}]Q[\theta^{n}_k])_{X_k^{i,n}, \cdot}, \;\;
r^{i,n}_{k+1} &\doteq (K[\theta^{n}_{k}]Q[\theta^{n}_k])_{X_k^{i,n}, \cdot} - (K[\theta^{n}_k]Q[\theta^{n}_k])_{X_{k+1}^{i,n},\cdot}
\end{aligned}
\end{equation}
For each $n,k \in \NN$, define the \emph{error} and \emph{remainder} arrays $\{e^n_k\}$ and $\{r^n_k\}$ by
\begin{equation}\label{eq:errors}
e^{n}_{k+1} \doteq \frac{1}{a(n)} \sum_{i=1}^{a(n)} e^{i,n}_{k+1}, 
\qquad  
r^{n}_{k+1} \doteq  \frac{1}{a(n)} \sum_{i=1}^{a(n)} r^{i,n}_{k+1}.
\end{equation}
Then the algorithm defined in (\ref{eq:stochalg}) can be written as
\begin{equation}\label{eq:rep1}
\theta^{n}_{k+1} \doteq \theta^{n}_{k} + \gamma_{k+1} h(\theta^{n}_k) + \gamma_{k+1} e^{n}_{k+1} + \gamma_{k+1} r^{n}_{k+1}.
\end{equation}

Along with the above evolution equation, it will be helpful to consider the \emph{linearized evolution} array $\{\mu^n_k\}$ given by
\begin{equation}\label{eq:mu}
\mu^{n}_0 \doteq 0 
\quad\text{and}\quad 
\mu^{n}_{k+1} \doteq \big(I + \gamma_{k+1} \nabla h(\theta_*) \big)\mu_k^{n} + \gamma_{k+1} {e}_{k+1}^{n},
\end{equation}
and to study the \emph{discrepancy} array $\{\rho^n_k\}$ given by
\begin{equation}\label{eq:rho}
\rho^{n}_0 \doteq \theta^n_0 - \theta_* 
\quad\text{and}\quad 
\rho_{k}^{n} \doteq (\theta_{k}^{n} - \theta_*) - \mu_{k}^{n}.
\end{equation}
Note that $\theta_{k}^{n} - \theta_* = \mu_{k}^{n} + \rho_{k}^{n}$ for all $k\in \NN_0$. As in \cite{for}, the proof of Theorem \ref{thm:clt1} relies on two steps: the first is to prove a central limit theorem for the sequence $\{\mu_n^n\}$ (with suitable scaling), and the second is to show that under the central limit scaling, the sequence $\{\rho^n_n\}$ tends to 0 in probability. 

We follow a similar approach for Algorithm II introduced in \eqref{eq:branchalg1}. This time we define the \emph{error} and \emph{remainder} sequences $\{\be_{n+1}\}$ and $\{\br_{n+1}\}$ by 
\begin{equation}\label{eq:benrn}
\be_{n+1} \doteq \frac{1}{a(n+1)} \sum_{i=1}^{a(n+1)} e^{i}_{n+1},
\qquad
\br_{n+1} \doteq \frac{1}{a(n+1)} \sum_{i=1}^{a(n+1)} r^{i}_{n+1},
\end{equation}
where the terms for each particle are given by
\begin{equation}\label{eq:er-def-br}
\begin{aligned}
e^{i}_{n+1} &\doteq Q[\btheta_n]_{\bX_{n+1}^{i},\cdot} - (K[\btheta_{n}]Q[\btheta_n])_{\bX_n^{i}, \cdot} \\
r^{i}_{n+1} &\doteq (K[\btheta_{n}]Q[\btheta_n])_{\bX_n^{i}, \cdot} - (K[\btheta_n]Q[\btheta_k])_{\bX_{n+1}^{i},\cdot}
\end{aligned}
\end{equation}
Then the sequence $\{\btheta_n\}$ defined in (\ref{eq:branchalg1}) can be rewritten as
\begin{equation}\label{eq:branchTheta1}
\btheta_{n+1} = \btheta_n + \gamma_{n+1} h(\btheta_n) + \gamma_{n+1} \be_{n+1} + \gamma_{n+1} \br_{n+1}.
\end{equation}
We also introduce the \emph{linearized evolution} sequence $\{\bmu_k\}$ given by
\begin{equation}\label{eq:mudef1}
\bmu_0 \doteq 0
\quad\text{and}\quad 
\bmu_{n+1} \doteq (I + \gamma_{n+1} \nabla h(\theta_*))\bmu_n + \gamma_{n+1} \be_{n+1}
\end{equation}
and we define the \emph{discrepancy} sequence $\{\brho_k\}$ by
\begin{equation}\label{eq:rhobranch1}
\brho_0 \doteq \btheta_0 - \theta_*
\quad\text{and}\quad 
\brho_{n+1} \doteq \btheta_{n+1} - \theta_* - \bmu_{n+1}.
\end{equation}
The proof once more  proceeds by first establishing  a central limit theorem for the linearized evolution and then showing that the discrepancy is asymptotically negligible.

\subsection{Notation}
\label{sec:notat}
The following notation will be used.
Convergence in distribution of random variables $Z_n$ to $Z$ will be denoted as
$Z_n \stackrel{\mathcal{L}}{\to} Z$. 
Constants in the proofs of various estimates will be denoted as $\kappa, \kappa_1, \kappa_2, \cdots$; their values may change from one proof to next. For a space $S$, $m \in \NN$ and a  bounded $h: S \to \RR^m$, $\|h\|_{\infty}\doteq \sup_{s\in S} \|h(s)\|$. For nonnegative sequences $\{a_n\}$, $\{b_n\}$ we write $a_n \sim b_n$, if $a_n/b_n \to 1$ as $n\to \infty$.
For a vector $v\in \RR^d$, the $j$-th coordinate will be denoted as $v(j)$ or $v_j$.
We denote by $ \mathcal{C}^0 \doteq C^0(\RR_+, \mathcal{P}(\Delta^o))$  the space of continuous $\mathcal{P}(\Delta^o)$-valued functions on $[0, \infty)$ endowed with the topology of uniform convergence on compact intervals. Recall that a sequence $\{x_n\}$ from $ \RR_+$ to $\mathcal{P}(\Delta^o)$ converges to a limit $x_*$ in $\mathcal{C}^0$ if and only if for each $T > 0$, 
$$
\lim_{n\rightarrow\infty} \sup_{0\leq s \leq T}\| x_n(t) - x_*(t)\| = 0.
$$
For $x \in \mathcal{C}^0$ we let $\| x \|_{T,*} \doteq \sup_{0\leq s \leq T}\| x(s)\|$.  Recall that the topology on $\mathcal{C}^0$ is induced by the metric 
$$
d(x,y) \doteq \sum_{T=1}^{\infty} 2^{-T}\min\{1 , \| x - y\|_{T,*}\}, \;\; x,y \in \mathcal{C}^0.
$$

\subsection{Organization}
The paper is organized as follows. In Section \ref{sec:rate} we prove a.s. convergence of Algorithm I and provide some associated rate of convergence bounds (Theorem \ref{thm:mainrate}). In Section \ref{sec:errors} we analyze the noise terms of Algorithm I. Combining results of Sections \ref{sec:rate} and \ref{sec:errors}, in Section \ref{sec:clt} we prove the  central limit theorem for this algorithm stated in Theorem \ref{thm:clt1}. In Section \ref{sec:branch} we prove a.s. convergence for Algorithm II
 and in Section \ref{sec:cltalg2} we establish the corresponding CLT.
 Finally, in Section \ref{sec:numeric} we present some exploratory numerical experiments.

\section{Convergence of Algorithm I}\label{sec:rate}

This section is dedicated to the proof of Theorem \ref{thm:mainrate}.
In Section \ref{sec:asymp} we introduce a notion of  {\em pseudo-trajectory sequences} for the flow induced by $h$ that is motivated by ideas of asymptotic pseudo trajectories  considered in \cites{ben,benclo} and which is more well-suited for the array-type stochastic approximations  studied here.  In Section \ref{sec:llnpf} we show that the sequence
 $\{\hat{\theta}^{n}\}$ of continuous time processes obtained from a suitable interpolation of our stochastic approximation array $\{\theta^n_k\}$ satisfies the  pseudo-trajectory sequence property introduced in Section \ref{sec:asymp} and finally, in Section \ref{sec:mainrate} we use this fact to complete the proof of Theorem \ref{thm:mainrate}.

\subsection{Pseudo-trajectory sequences}\label{sec:asymp}
Consider the sequence of algorithm update time instants $\{\tau_k\}$ associated with the SA, defined as
\begin{equation}\label{eq:updatetime}
\tau_0 = 0, \; \tau_k \doteq \sum_{j=1}^k \gamma_j, \; k \in \NN.
\end{equation}
For $r \in \RR_+$, we let $\tau_r \doteq \tau_{\lfloor r \rfloor}$.
  For $\nu \in \mathcal{P}(\Delta^o)$, consider the ODE
associated with the flow induced by $h$,
\begin{equation}\label{eq:phiflow}
\dot{\Phi}(t) = h(\Phi(t)), \;\; \Phi(0) = \nu.
\end{equation}
We denote the solution to (\ref{eq:phiflow}) with initial condition $\Phi(0) = \nu$ by $\{\Phi_t(\nu)\}$. 

We now introduce a notion of a pseudo-trajectory sequence that will be convenient for our purposes. 
\begin{defn}\label{defn:lambdapt}
For  $\lambda < 0$ and $\po \in (0,1]$, we say that a sequence $\{X_n\} \subset \mathcal{C}^0$ is a $(\lambda, \tau_n, \po)$-pseudo-trajectory sequence (PTS) for  $\Phi$ if for all $T > 0$ and $\epsilon>0$, there is an $n_0\in \NN$ such that for all $n \ge n_0$ and $0 \leq j \leq L_n \doteq L_n(\po, T) \doteq \left\lfloor \frac{1}{T}\left(\tau_n - \frac{\tau_{n^{\po}}}{2}\right)\right\rfloor + 1$,
and $t_{n,j} \doteq \frac{\tau_{n^{\po}}}{2}+jT$,
$$
  \sup_{0\leq u \leq 2T} \left\|  X_n(t_{n,j} + u ) - \Phi_{u}(X_n(t_{n,j}))\right\|  \le \exp\{(\lambda+\epsilon) t_{n,j}\}.
$$
\end{defn}

%

The following lemma provides an upper bound for the rate at which a $(\lambda,\tau_n,\po)$-PTS converges to $\theta_*$. Recall that the largest eigenvalue of $\nabla h(\theta_*)$ is bounded above by $- L < 0$.

\begin{lemma}\label{lem:rate1}
Suppose that for some $\lambda < 0$ and $\po \in (0,1)$,  $\{X_n\}$ is a $(\lambda, \tau_n, \po)$-PTS for $\Phi$. Then there is some $\beta > 0$ and $n_0 \in \NN$ such that for all $n \geq n_0$, if $n^{\po} \leq m \leq n$, then
$$
\| X_n(\tau_m) - \theta_*\| \leq \exp(-\beta \tau_m).
$$
\end{lemma}

\begin{proof}
Fix $\alpha_1 \in (0, L)$. Then we can find (cf.
 \cite{benclo}*{Lemma 2.1}) some $T\in (0,\infty)$  so that for all $\nu \in \mathcal{P}(\Delta^o)$,
$$
\sup_{T\le u \le 2T} \| \Phi_u(\nu) - \theta_*\| \leq \exp( - \alpha_1 T)\|\nu - \theta_*\|.
$$
For $n^{\po} \le m \le n$, let $0\le j(m)\le L_n$ be such that $t_{n, j(m)+1} \le \tau_m \le t_{n, j(m)+1} +T$ and let $u(m) \doteq \tau_m - t_{n,j(m)}$, so that
$$\tau_m = (\tau_m - t_{n, j(m)}) + t_{n, j(m)} = u(m) + t_{n, j(m)}.$$
Note that $u(m)\in [T,2T]$.
Now, fix $\epsilon \in (0, -\lambda)$, and let $\alpha_2 \doteq -(\lambda + \epsilon) > 0$. Define
$\alpha \doteq \alpha_1 \wedge \alpha_2$. Since $\{X_n\}$ is a $(\lambda,\tau_n,\po)$-PTS for $\Phi$, we can find some $n_0$ such that for all $n \geq n_0$ and   for each $n^{\po} \le m \leq n$, 
\begin{equation*}
\begin{split}
\| X_n(\tau_m) - \theta_*\| &\leq \| X_n(u(m) + t_{n, j(m)}) - \Phi_{u(m)}(X_n(t_{n, j(m)}))\|
 + \| \Phi_{u(m)}(X_n(t_{n, j(m)})) - \theta_*\|\\
&\leq \exp( - \alpha t_{n, j(m)}) + \exp(- \alpha T)\| X_n(t_{n, j(m)}) - \theta_*\|\\
&\leq \exp \left( - \alpha \frac{\tau_{n^{\po}}}{2}\right) + \exp(-\alpha T)\|X_n(t_{n, j(m)}) - \theta_*\|.
\end{split}
\end{equation*}
Iterating this for an additional $j(m)$ times, 
 we  see that there are $\kappa_i \equiv \kappa_i(\alpha, T) \in (0,\infty)$ such that
 if $n \geq n_0$ and $m \geq n^{\po}$, then 
\begin{equation}\label{eq:eq709}
\begin{split}
\| X_n(\tau_m) - \theta_*\| &\leq \exp\left( - \alpha \frac{\tau_{n^{\po}}}{2}\right)\left( \sum_{k=0}^{j(m)} \exp(- \alpha k T)\right) + \kappa_1\exp\left(-\alpha (j(m)+1)T\right)\\
&\le \kappa_2\left (\exp\left( - \alpha \frac{\tau_{n^{\po}}}{2}\right) + \exp\left(-\alpha (j(m)+1)T\right)\right)\\
&= \kappa_2\left (\exp\left( - \alpha \frac{\tau_{n^{\po}}}{2}\right) + \exp\left(-\alpha (j(m)+2)T\right) \exp( \alpha T) \right)\\
&\le \kappa_3 \left (\exp\left( - \alpha \frac{\tau_{n^{\po}}}{2}\right) + \exp\left(-\alpha (j(m)+2)T\right)  \right)
\end{split}
\end{equation}
Note that, by our choice of $j(m)$,
$$T(j(m)+2) \ge \tau_m - \frac{\tau_{n^{\po}}}{2} \ge \frac{\tau_{n^{\po}}}{2}.$$
Also note that for $k \in \NN$
$$
\gamma_* (\log(k+N_*)- \log(N_*)) \le \tau_k \le 1+ \gamma_*\log(k+N_*-1)	
$$
from which it follows that, there is a $\kappa_4 \in (0, \infty)$ and $n_1\ge n_0$ such that and all $n\ge n_1$,
$\frac{\tau_{n^{\po}}}{2} \ge \kappa_4 \tau_n.$
Combining the above two observations with \eqref{eq:eq709}, we have for all $n \ge n_1$ and $n^{\po} \le m \leq n$
$$
\| X_n(\tau_m) - \theta_*\| \leq 2\kappa_3 \exp\left( - \alpha \frac{\tau_{n^{\po}}}{2}\right)
\le 2\kappa_3 \exp\left( - \alpha \kappa_4 \tau_{m}\right).$$

The result follows.
\end{proof}

\subsection{The algorithm as a pseudo-trajectory sequence}\label{sec:llnpf}
In this section we show that a suitable continuous time interpolation of the array $\{\theta^n_k\}$ is a PTS for $\Phi$ in the sense of Definition \ref{defn:lambdapt}. 
For $n \in \NN$,   let $\hat{\theta}^n$ be the continuous-time process defined as
$$
\hat{\theta}^n(\tau_k + t) \doteq \theta_k^n + t\frac{\theta_{k+1}^n -\theta_k^n}{\tau_{k+1}-\tau_{k}},
\; t \in [0, \gamma_{k+1}) \mbox{ and } k \in \NN_0.
$$
We write $\bar{\theta}^{n}(\cdot)$ to denote the analogous continuous-time process obtained by piecewise constant interpolations of $\{\theta^{n}_{k}\}$. We will prove in this section that, with $\lambda = -(2\gamma_*)^{-1}$
and arbitrary $\po \in (0,1)$,
$\{\hat{\theta}^n\}$   is a $(\lambda, \tau_n, \po)$-PTS for $\Phi$.
%
%
Towards that end, let
\begin{equation}\label{eq:epsilondef}
\bar{\epsilon}^{n}(\tau_k + t) \doteq \epsilon^{n}_{k+1}, \; \; t \in [0, \gamma_{k+1}), \; k \in \NN_0
\end{equation}
and define
\begin{equation}\label{eq:deltadef}
\Delta(n,t_{n,j}, T) \doteq  \sup_{0\leq u \leq 2T} \left\| \int_{ t_{n,j} }^{ t_{n,j}+ u}\bar{\epsilon}^{n}(s)ds\right\|, \;\; 0 \leq j \leq L_n,
\end{equation}
 where $L_n = L_n(\po,T)$ is as in Definition \ref{defn:lambdapt}. In Lemma \ref{lem:deltaasy} we provide an estimate relating $\hat \theta^n$ with $\Delta(\cdot , \cdot, \cdot)$ that is used to prove asymptotic properties of $\{\hat{\theta}^{n}\}$. The proof is a consequence of the Lipschitz property of
 $h$ and Gr{\"o}nwall's lemma.
Define $m : \RR_+ \rightarrow \NN_0$ by 
\begin{equation}\label{eq:mtdef}
m(t)  \doteq \sup\{k \geq 0: t \geq \tau_k\}, \; t\ge 0.
\end{equation}
\begin{lemma}\label{lem:deltaasy}
	For each $T \in  (0,\infty)$ there is a  $C \doteq C(T)\in (0,\infty)$
such that for all $n\in \NN$ and
$k\leq L_n$,
$$
\sup_{0\leq u \leq 2T}\left\| \hat{\theta}^{n}(t_{n,k} + u) - \Phi_u(\hat{\theta}^{n}(t_{n,k}))\right\| \leq C(T)[ \Delta(n, t_{n,k}, T) + \gamma_{m(t_{n,k})}]
$$
\end{lemma}

\begin{proof}
Fix $T \in (0,\infty)$. Note that, for $n \in \NN$ and $t\ge 0$,
$$
\hat{\theta}^{n}(t) = \hat{\theta}^{n}(0) + \int_0^t [ h(\bar{\theta}^{n}(s)) + \bar{\epsilon}^{n}(s)]ds.
$$
Define
$$
A_{n,k}(s) \doteq \int_{ t_{n,k}}^{ t_{n,k} + s }[ h(\bar{\theta}^{n}(u)) - h(\hat{\theta}^{n}(u))]du,\;\; B_{n,k}(s) \doteq \int_{t_{n,k} }^{ t_{n,k} + s} \bar{\epsilon}^{n}(u)du.
$$
Then, for $0\le s \le 2T$,
\begin{equation*}
\begin{split}
\hat{\theta}^{n}(t_{n,k}+ s) = \hat{\theta}^{n}(t_{n,k}) + \int_0^s h(\hat{\theta}^{n}(t_{n,k}+ u))du + A_{n,k}(s) + B_{n,k}(s).
\end{split}
\end{equation*}
Also,
$$
\Phi_s(\hat{\theta}^{n}(t_{n,k})) = \hat{\theta}^{n}(t_{n,k}) + \int_0^s h (\Phi_u(\hat{\theta}^{n}(t_{n,k})))du.
$$
Letting $K$ denote the Lipschitz constant of $h$ we see that
\begin{equation}\label{eq:gron}
\begin{split}
\| \hat{\theta}^{n}(t_{n,k} + s) - \Phi_s(\hat{\theta}^{n}(t_{n,k}))\| &= \left\| \int_0^s [h(\hat{\theta}^{n}(t_{n,k} + u)) - h(\Phi_u(\hat{\theta}^{n}(t_{n,k})))]du + A_{n,k}(s) + B_{n,k}(s)\right\|\\
&\leq K \int_0^s \| \hat{\theta}^{n}( t_{n,k} + u) - \Phi_u(\hat{\theta}^{n}( t_{n,k}))\|du + \|A_{n,k}(s)\| + \|B_{n,k}(s)\|,
\end{split}
\end{equation}
for all $s \geq 0$. Next 
for each $u \in [ t_{n,k}, t_{n,k} + 2T]$,
\begin{equation}\label{eq:thetatau2}
\hat{\theta}^{n}(u) - \bar{\theta}^{n}(u)= \hat{\theta}^{n}(u) - \hat{\theta}^{n}(\tau_{m(u)}) = \int_{\tau_{m(u)}}^{u}[ h(\bar{\theta}^{n}(s)) + \bar{\epsilon}^{n}(s)]ds.
\end{equation}
Note that, with $\kappa_1 \doteq \| h\|_{\infty}+2$, for $u \in [ t_{n,k}, t_{n,k} + 2T]$,
$$
\left\| \int_{\tau_{m(u)}}^u \left(h(\bar{\theta}^{n}(s))
+ \bar{\epsilon}^{n}(s)\right) ds \right\| \le  \kappa_1(u - \tau_{m(u)}) \le  \kappa_1 \gamma_{m(u)} \leq \kappa_1 \gamma_{m( t_{n,k})},
$$
  %
Combining the above estimate with \eqref{eq:thetatau2}, it follows that 
for $0 \leq s \leq 2T$,
\begin{equation}\label{eq:boundA}
\begin{split}
\| A_{n,k}(s) \| &\leq K \int_{ t_{n,k}}^{ t_{n,k} + s} \| \bar{\theta}^{n}(u) - \hat{\theta}^{n}(u)\|du \leq 2KT  \kappa_1 \gamma_{m(t_{n,k})}.
\end{split}
\end{equation}
The result now follows on using the estimate 
\eqref{eq:boundA} in \eqref{eq:gron},
recalling the definition of $B_{n,k}$ and
$\Delta(n, t_{n,k},T)$,
and applying Gr{\"o}nwall's lemma.
\end{proof}

Lemma  \ref{lem:deltalemma} provides the key estimate in the proof that $\{\hat{\theta}^{n}\}$ is a PTS for $\Phi$. The main ingredients in its proof are Lemmas  \ref{lem:boundu}, \ref{lem:delta2}, \ref{lem:delta3}, and \ref{lem:delta4} given below. Consider the following decomposition of the algorithm's noise given in terms of $\delta^{\ell,i, n}_{k+1}$ defined as, for $1\le i \le a(n)$,
\begin{equation}\label{eq:delta-defs}
\delta^{\ell,i,n}_{k+1}(x) \doteq \left\lbrace
\begin{aligned}
&\gamma_{k+1}\,Q[\theta_k^{n}]_{X_{k+1}^i,x}\;
-\gamma_{k+1}\,\left(K[\theta_k^{n}]Q[\theta_k^{n}]\right)_{X_k^i,x}
&\qquad\ell = 1&\;\\
&\gamma_{k+1}\,\left(K[\theta_k^{n}]Q[\theta_k^{n}]\right)_{X_k^i,x}
\;-\gamma_{k}\,\left(K[\theta_k^{n}]Q[\theta_k^{n}]\right)_{X_k^i,x}
&\qquad\ell= 2&\;\\
&\gamma_{k}\,\left(K[\theta_k^{n}]Q[\theta_k^{n}]\right)_{X_k^i,x}
\;-\gamma_{k+1}\left(K[\theta_{k+1}^{n}]Q[\theta_{k+1}^{n}]\right)_{X_{k+1}^i,x}
&\qquad\ell = 3&\;\\
&\gamma_{k+1}\,\left(K[\theta_{k+1}^{n}]Q[\theta_{k+1}^{n}]\right)_{X_{k+1}^i,x}
\;-\gamma_{k+1}\,\left(K[\theta_k^{n}]Q[\theta_k^{n}]\right)_{X_{k+1}^i,x}
&\qquad\ell  = 4&\;
\end{aligned}
\right.
\end{equation}
For each $1 \leq \ell \leq 4$, let
\begin{equation}\label{eq:noise-decomp}
\delta^{\ell,{n}}_{k+1} \doteq \dfrac{1}{a(n)} \sum_{i=1}^{a(n)}\delta^{\ell,i,n}_{k+1}
\quad\text{and observe that}\quad
\gamma_{k+1} \epsilon_{k+1}^{n} = \sum_{\ell=1}^{4} \delta^{\ell,n}_{k+1}
\end{equation}
since
\begin{equation}\label{eq:eq417}
\gamma_{k+1} \epsilon_{k+1}^{i,n} = \gamma_{k+1} \left(Q[\theta_k^{n}]_{X_{k+1}^i,x}- \left(K[\theta_k^{n}]Q[\theta_k^{n}]\right)_{X_{k+1}^i,x}\right)
 =\sum_{\ell=1}^4 \delta^{\ell,i,n}_{k+1}.
\end{equation}

 The following lemma estimates the error term corresponding to $\ell=1$.
%
Henceforth in this section we assume that $\po \in (0,1)$ and $T\in (0,\infty)$ are fixed, and $\lambda \doteq -(2 \gamma_*)^{-1}$. Recall the quantities $L_n$ and $t_{n,j}$ from Definition \ref{defn:lambdapt}.

\begin{lemma}\label{lem:boundu}
Let $q \ge 2$.
There is a  $n_0 \in \NN$ and $C(q,T)\in (0,\infty)$ such that for all $n \ge n_0$ and all $0\le k \leq L_n$,
$$
\EE\left(\sup_{0\leq u \leq 2T}\left\| \sum_{j = m(t_{n,k})}^{m(t_{n,k}+u)}\delta^{1,n}_{j+1}\right\|^q\right) \leq C(q,T) \exp\left( q \lambda t_{n,k} \right).
$$
\end{lemma}

\begin{proof}
Note that for each $n \in \NN$, $\{\delta^{1,n}_{j}\}_{j=1}^{\infty}$ is adapted to $\{\mathcal{F}^n_{j}\}_{j=1}^{\infty}$. Additionally, from \eqref{eq:eq242}, for all $n\in \NN, j \in \NN_0$,
\begin{equation}\label{eq:martinc1}
\begin{split}
\EE[\delta^{1,n}_{j+1} | \mathcal{F}_j] &= \gamma_{j+1} \frac{1}{a(n)} \sum_{i=1}^{a(n)} \EE\left[ \left(Q[\theta_j^{n}]_{X_{j+1}^{i,n},\cdot} - K[\theta_j^{n}]Q[\theta_j^{n}]_{X_{j}^{i,n},\cdot}\right) \bigg| \mathcal{F}_j^n \right] = 0.
\end{split}
\end{equation}
Also, for all  $n\in \NN$ and  $j \in \NN_0$, 
$$
\big\| \delta^{1,n}_{j+1} \big\| \le \frac{1}{a(n)} \sum_{i=1}^{a(n)} \big\| \delta^{1,i,n}_{j+1} \big\| \leq
  \gamma_{j+1}\frac{1}{a(n)} \sum_{i=1}^{a(n)} \left\| Q[\theta^n_j]_{X_{j+1}^{i,n},\cdot} - K[\theta^n_j]Q[\theta^n_j]_{X_{j}^{i,n},\cdot}\right\| \le \kappa_1\gamma_{j+1}
$$
for some  $\kappa_1 \in (0,\infty)$. 
Thus, for each $n \in \NN$, $\{\delta^{1,n}_{j}\}_{j=1}^{\infty}$ is a martingale difference sequence, and so from  Burkholder's inequality we can find a  $\kappa_2(q) \in (0,\infty)$ such that for all $n \in \NN$ and $0\le k \leq L_n$,
\begin{equation*}
\begin{split}
f_{n,k}(q) &\doteq \EE\left( \sup_{0\leq  u \leq 2T}\left\|  \sum_{i = m(t_{n,k})}^{m(t_{n,k}+u)} \delta^{1,n}_{i+1}\right\|^q\right)
\leq  \EE\left( \sup_{m(t_{n,k})\leq j \leq m(t_{n,k}+2T)}\left\| \sum_{i=m(t_{n,k})}^j \delta^{1,n}_{i+1}\right\|^q\right)\\
&\leq \kappa_2(q) \EE\left(\left[  \sum_{i=m(t_{n,k})}^{m(t_{n,k} + 2T)} \big\|\delta^{1,n}_{i+1}\big\|^2\right]^{q/2}\right)
\leq \kappa_2(q) \kappa_1^q  \left[ \sum_{i=m(t_{n,k})}^{m(t_{n,k}+2T)} \gamma_{i+1}^2\right]^{q/2}\\
&\leq \kappa_2(q) \kappa_1^q \big[ m(t_{n,k}+2T) - m(t_{n,k})+1 \big]^{q/2} \gamma^q_{m(t_{n,k})}.
\end{split}
\end{equation*}
Next note that, for some $\kappa_3(T) \in (0,\infty)$ and $n_0 \in \NN$,
\begin{equation}\label{eq:eq256}
\big[m(t_{n,k}+2T) - m(t_{n,k})+1 \big] \le  \kappa_3(T) \exp\left( \frac{t_{n,k}}{\gamma_*} \right)
\end{equation}
and for all $n \ge n_0$
\begin{equation}\label{eq:eq257}
\gamma_{m(t_{n,k})} \le  \frac{\gamma_*e}{N_*}\exp\left( - \frac{t_{n,k}}{\gamma_*} \right).
\end{equation}

%
%

Thus, there is some $\kappa_4(q,T) \in (0,\infty)$ such that for all  $n\ge n_0$, 
$$
f_{n,k}(q) \leq \kappa_4(q,T) \exp\left( - \frac{q}{2\gamma_*}t_{n,k}\right) = \kappa_4(q,T) \exp\left(\lambda qt_{n,k}\right).
$$
The result follows.
\end{proof}

The next three lemmas, namely Lemmas \ref{lem:delta2}, \ref{lem:delta3}, and \ref{lem:delta4} will be used to bound the remaining error terms.

\begin{lemma}\label{lem:delta2}
 There is a  $C \in (0, \infty)$ such that for all $T>0$, $n \in \NN$ and  $k \leq L_n$
$$
\sup_{0\leq u \leq 2T}\left\|\sum_{j=m(t_{n,k})}^{m(t_{n,k}+u)} \delta^{2,n}_{j+1}\right\| \leq C \gamma_{m( t_{n,k})}.
$$
\end{lemma}

\begin{proof}
Fix $n \in \NN$ and let 
\begin{equation}\label{eq:eq319}
\kappa \doteq \sup_{\theta \in \mathcal{P}(\Delta^o), x \in \Delta^o}\left\|
\left(K[\theta]Q[\theta]\right)_{x,\cdot}\right\| < \infty.
\end{equation}
Then for each $1 \leq i \leq a(n)$, $\| \delta^{2,i,n}_{j+1}\| \leq \kappa( \gamma_{j+1} - \gamma_{j})$ for all $j \in \NN_0$, and so for each $k \leq L_n$,
$$
\sup_{0 \leq u\leq 2T}\left\|\sum_{j=m(t_{n,k})}^{m(t_{n,k}+u)} \delta^{2,n}_{j+1} \right\| \leq \sup_{0 \leq u \leq 2T} \sum_{j=m(t_{n,k})}^{m(t_{n,k}+u)}  \frac{1}{a(n)}\sum_{i=1}^{a(n)}\big\| \delta^{2,i,n}_{j+1} \big\| \leq  \kappa \gamma_{m( t_{n,k})}.
$$
The result follows.
\end{proof}

\begin{lemma}\label{lem:delta3}
 There is a  $C \in (0,\infty)$  such that for all $T>0$, $n \in \NN$ and $0\le k \leq L_n$.
$$
\sup_{0\leq u \leq 2T}\left\|\sum_{j=m(t_{n,k})}^{m(t_{n,k}+u) } \delta^{3,n}_{j+1}\right\| \leq C  \gamma_{m(t_{n,k})}.
$$
\end{lemma}

\begin{proof}
Let $b_j^{i,n}(\cdot) \doteq \gamma_{j}K[\theta_j^{n}]Q[\theta_j^{n}]_{X_j^i, \cdot}$, so that $\delta^{3,i,n}_{j+1} = b_j^{i,n} - b_{j+1}^{i,n}$. Then, with $\kappa$ as in \eqref{eq:eq319},
we have that
$
\|b^{i,n}_j\|  \leq \kappa \gamma_{j}.
$
Thus
\begin{equation*}
\begin{split}
\left\|\sum_{j=m(t_{n,k})}^{m(t_{n,k}+u) } \delta^{3,n}_{j+1}\right\| 
& = \left\|\sum_{j=m(t_{n,k})}^{m(t_{n,k}+u) } \frac{1}{a(n)} \sum_{i=1}^{a(n)} \left(b^{i,n}_j - b^{i,n}_{j+1}\right)\right\| \\
&\leq \frac{1}{a(n)} \sum_{i=1}^{a(n)}\left\| b_{m(t_{n,k})}^{i,n} - b_{m(t_{n,k}+u) + 1}^{i,n}\right\| \leq 2\kappa \gamma_{m(t_{n,k})}.
\end{split}
\end{equation*}
The result follows.
\end{proof}

\begin{lemma}\label{lem:delta4}
Fix $T\in (0,\infty)$.
 There is a  $C \in (0,\infty)$ such that for all   $n\in \NN$ and $0\le k\le L_n$ 
$$
\sup_{0\leq u \leq 2T}\left\|\sum_{j=m(t_{n,k})}^{m(t_{n,k}+u)} \delta^{4,n}_{j+1}\right\| \leq C\gamma_{m(t_{n,k})}.
$$
\end{lemma}

\begin{proof}
Using the boundedness and Lipschitz property of $K(\cdot)$ and $Q(\cdot)$, we see that for some $\kappa_1\in(0,\infty)$,
$\|\delta^{4,i,n}_{j+1}\| \leq \kappa_1 \gamma_{j+1} \|\theta_{j+1}^{n} - \theta_j^{n}\|$. Also, from \eqref{eq:stochalg1},
$$
\|\theta_{j+1}^{n} - \theta_j^{n}\| = \gamma_{j+1} \| h(\theta_j^{n}) + \epsilon_{j+1}^{n}\| \leq 3\gamma_{j+1}.
$$
Thus, for $0\le u \le 2T$,
\begin{align*}
\left\|\sum_{j=m(t_{n,k})}^{m(t_{n,k}+u)} \delta^{4,n}_{j+1}\right\| 
= \left\|\sum_{j=m(t_{n,k})}^{m(t_{n,k}+u)} \frac{1}{a(n)} \sum_{i=1}^{a(n)}\delta^{4,i,n}_{j+1}\right\| 
&\leq 3\kappa_1 \sum_{j=m(t_{n,k})}^{m(t_{n,k}+u)} \gamma_{j+1}^2 \\
&\leq  3\kappa_1 \gamma_{m(t_{n,k})+1}^2 \big[ m(t_{n,k}+2T) - m(t_{n,k})+1 \big].
\end{align*}
The result now follows on noting that from \eqref{eq:eq256} and \eqref{eq:eq257}, for some $\kappa_2\in (0,\infty)$
\[
\gamma^2_{m(t_{n,k})+1} \big[m(t_{n,k}+2T) - m(t_{n,k})+1 \big] \le  \kappa_2\gamma_{m(t_{n,k})}\exp\left(\frac{t_{n,k}}{\gamma_*}\right)\exp\left(-\frac{2t_{n,k}}{\gamma_*}\right)
\le \kappa_2\gamma_{m(t_{n,k})}.
\qedhere
\]
\end{proof}

The following corollary is used in the proof of Lemma \ref{lem:deltalemma}. Recall the definition of $r^n_{k+1}$ from \eqref{eq:errors}.
For a collection of events $\{A_{n,k}: 0 \le k \le L_n, n \in \NN\}$ we denote
$$\{A_{n,k} \mbox{ i.o. }\} \doteq \{\om: \om \in A_{n,k} \mbox{ for infinitely many  } (n,k), \mbox{ s.t. }
0\le k \le L_n, n \in \NN\}.$$


\begin{corollary}\label{cor:etazero}
Fix  $T \in (0,\infty)$. Then, for each $C \in (0,\infty)$
%
$$
  \PP\left( \sup_{0\leq u \leq 2T}\left\| \sum_{j = m(t_{n,k})}^{m(t_{n,k}+u)}\gamma_{j+1}r^{n}_{j+1} \right\|  + C \gamma_{m(t_{n,k})} > \frac{1}{2}\exp(\lambda t_{n,k}) \text{ i.o.}\right) = 0.
$$
\end{corollary}

\begin{proof}
	Fix $T, C \in (0,\infty)$.
From Lemmas \ref{lem:delta2}, \ref{lem:delta3}, and \ref{lem:delta4}, and \eqref{eq:eq257}, for some $\kappa_1\in (0,\infty)$, $n_0 \in \NN$, and all $n\ge n_0$,
\begin{align*}
	\sup_{0\leq u \leq 2T}\left\| \sum_{j = m(t_{n,k})}^{m(t_{n,k}+u)}\gamma_{j+1}r^{n}_{j+1} \right\|  + C \gamma_{m(t_{n,k})}
	&\le \sum_{\ell=2}^4 \sup_{0\leq u \leq 2T}\left\|\sum_{j=m(t_{n,k})}^{m(t_{n,k}+u)} \delta^{\ell,n}_{j+1}\right\| + C \gamma_{m(t_{n,k})} \\
	&\le \kappa_1 \gamma_{m(t_{n,k})}  
	\le \frac{\kappa_1\gamma_*e}{N_*} \exp\left( - \frac{t_{n,k}}{\gamma_*} \right).
\end{align*}
The last expression can be bounded by $\frac{1}{2}\exp(\lambda t_{n,k})$ for $n$ sufficiently large. The result follows.
%
%
%
\end{proof}

We now present the key estimate that will be used to prove Theorem \ref{thm:mainrate}.

\begin{lemma}\label{lem:deltalemma}
For each $T > 0$, 
$$
\limsup_{n\rightarrow\infty}\sup_{k \leq L_n} \frac{1}{t_{n,k}} \log \Delta(n,t_{n,k}, T) \le \lambda \quad\text{ a.s.}
$$

\end{lemma}

\begin{proof}
Fix $T\in (0,\infty)$ and  $\epsilon \in (0, - \lambda)$. Write $\sigma = \lambda+\epsilon$.
From the boundedness of $\epsilon^n_{k+1}$, we can find $\kappa_1 \in (0,\infty)$ such that
\begin{equation*}
\begin{split}
&\Delta(n,t_{n,k},T) = \sup_{0\leq u \leq 2T}\left\| \int_{t_{n,k}}^{t_{n,k}+u} \bar{\epsilon}^{n}(s)ds\right\|\\
&\leq\sup_{0\leq u \leq 2T}\left\| \int_{\tau_{m(t_{n,k})}}^{\tau_{m(t_{n,k}+u)}}\bar{\epsilon}^{n}(s)ds\right\|  + \kappa_1\left[ \sup_{0\leq u \leq T}| t_{n,k} - \tau_{m(t_{n,k})}| + |t_{n,k} + u - \tau_{m(t_{n,k} +u)}|\right]\\
&\leq \sup_{0\leq u \leq 2T}\left\| \sum_{j = m(t_{n,k})}^{m(t_{n,k}+u)} \gamma_{j+1}\epsilon_{j+1}^{n}\right\| + 2\kappa_1 \gamma_{m(t_{n,k})}.
\end{split}
\end{equation*}
From \eqref{eq:eq417} it follows that
\begin{equation*}
\begin{split}
\Delta(n,t_{n,k},T)  &\leq   \sup_{0\leq u \leq 2T}\left\| \sum_{j = m(t_{n,k})}^{m(t_{n,k}+u)}\delta^{1,n}_{j+1}\right\| +  
\sup_{0\leq u \leq 2T}\left\| \sum_{j = m(t_{n,k})}^{m(t_{n,k}+u)}\gamma_{j+1}r^{n}_{j+1} \right\|  + 2\kappa_1 \gamma_{m(t_{n,k})}.
\end{split}
\end{equation*}
Thus
\begin{equation}\label{eq:infoft1}
\begin{split}
\PP\Big(\Delta(n,t_{n,k},T) \geq \exp(t_{n,k}\sigma) \quad\text{i.o.}\Big)
\leq \PP\left(  \sup_{0\leq u \leq 2T}\left\| \sum_{j = m(t_{n,k})}^{m(t_{n,k}+u)}\delta^{1,n}_{j+1}\right\|  > \frac{1}{2}\exp(t_{n,k}\sigma) \quad\text{i.o.}\right) \\
\qquad +  \PP\left( \sup_{0\leq u \leq 2T}\left\| \sum_{j = m(t_{n,k})}^{m(t_{n,k}+u)}\gamma_{j+1}r^{n}_{j+1} \right\|  + 2\kappa_1 \gamma_{m(t_{n,k})} > \frac{1}{2}\exp(t_{n,k}\sigma) \quad\text{i.o.}\right).
\end{split}
\end{equation}
From Corollary \ref{cor:etazero}, since $\sigma>\lambda$,
the second term in \eqref{eq:infoft1} equals $0$.
Since $\epsilon \in (0, -\lambda)$ is arbitrary, in order 
to complete the proof of the lemma it now suffices to
show that
\begin{equation}\label{eq:borelcant1}
\PP\left(  \sup_{0\leq u \leq 2T}\left\| \sum_{j = m(t_{n,k})}^{m(t_{n,k}+u)} \delta^{1,n}_{j+1} \right\|  
> \frac{1}{2}\exp(t_{n,k}\sigma) \quad\text{i.o.}\right) = 0.
\end{equation}
Note that we can find some $\alpha_0 > 0$ and $n_0\in \NN$ such that for all $n \ge n_0$ and all $k \leq L_n$,
$$
t_{n,k}  \geq \frac{\tau_{n^{\po}}}{2} \geq \alpha_0 \log(n).
$$
Fix 
$q >  \frac{1}{\epsilon \alpha_0}\vee 2$. Applying Lemma \ref{lem:boundu} and Markov's inequality, we can find $\kappa_2, \kappa_3 \in (0,\infty)$ 
such that  
\begin{multline}
\sum_{n=1}^{\infty}\sum_{k=1}^{L_n}  \PP\left(  \sup_{0\leq u \leq 2T}\left\| \sum_{j = m(t_{n,k})}^{m(t_{n,k}+u)} \delta^{1,n}_{j+1} \right\|  > \frac{1}{2}\exp(t_{n,k}\sigma)\right)\\
\leq  2^q\kappa_2C(q,T) \sum_{n=1}^{\infty}\sum_{k=1}^{L_n}  \exp\left( -qt_{n,k} \sigma\right) \exp\left( qt_{n,k}\lambda\right)
=  2^q\kappa_2C(q,T) \sum_{n=1}^{\infty}\sum_{k=1}^{L_n}  \exp\left( -qt_{n,k} \epsilon\right)\\
\leq 2^q\kappa_3C(q,T) \sum_{n=1}^{\infty} \log(n) \exp\left( -q \epsilon \alpha_0 \log(n) \right)
\leq  2^q\kappa_3C(q,T) \sum_{n=1}^{\infty}\log(n) \frac{1}{n^{q\epsilon \alpha_0}} 
< \infty,
\end{multline}
where $C(q,T)$ is as in Lemma \ref{lem:boundu}. 
The equation in \eqref{eq:borelcant1} now follows from the Borel-Cantelli lemma and the result follows. 
\end{proof}

\subsection{Proof of Theorem \ref{thm:mainrate}}\label{sec:mainrate}

We now complete the proof of Theorem \ref{thm:mainrate}. 
Fix $\po \in (0,1)$ . From Lemma \ref{lem:deltaasy}, for every $T<\infty$,  there is a  $C(T) \in (0, \infty)$ 
such that for all $n \in \NN$ and all $0 \leq k \leq L_n$,
\begin{equation}\label{eq:eq637}
\begin{split}
\sup_{0\leq u \leq 2T} \big\| \hat{\theta}^{n} (t_{n,k}+u) - \Phi_u(\hat{\theta}^{n}(t_{n,k})) \big\| &\leq C(T) \big[ \Delta(n, t_{n,k}, T) + \gamma_{m(t_{n,k})} \big].
\end{split}
\end{equation}
Additionally, Lemma \ref{lem:deltalemma} ensures that for a.e. $\om$, for every  $\epsilon \in (0,-\lambda)$, $T<\infty$, there is some $n_1 \equiv n_1(\om,\epsilon, T)  \in \NN$ such that for all $n \geq n_1$ and all $0\le k \leq L_n$,
$$
 \Delta(n,t_{n,k}, T) \leq \exp\left( t_{n,k} \left(\lambda + \epsilon/2\right)\right).
$$
 Combining this with \eqref{eq:eq257} and \eqref{eq:eq637}, we have that for some $n_2 \ge n_1$ and all $n\ge n_2$,
$0\le k \leq L_n$, we have
$$
\sup_{0\leq u \leq 2T} \big\|\hat{\theta}^{n} (t_{n,k}+u) - \Phi_u(\hat{\theta}^{n}(t_{n,k})) \big\| \leq \exp\left(\left( \lambda + \epsilon\right)t_{n,k}\right).
$$
We have thus shown that $\{\hat{\theta}^{n}\}$ is a.s. a $(\lambda, \tau_n, \po)$-pseudo-trajectory, so Lemma \ref{lem:rate1} ensures that there is some $\beta > 0$ and $n_0=n_0(\om) \in \NN$ such that for all $n \geq n_0$ and  $n^{\po}\le k \leq n$, 
$$
\big\| \hat{\theta}^{n}(\tau_k) -  \theta_* \big\| \leq \exp(-\beta \tau_k).
$$
The result follows.\qed

\section{Analysis of the noise terms in Algorithm I}\label{sec:errors}
The goal of this section is to provide  estimates on the error terms defined in \eqref{eq:errors} that will be useful for the study of the CLT. In Section \ref{sec:covariance} we characterize the covariance structure of the error terms $\{e_{k+1}^n\}$. In Section \ref{sec:moments} we  provide some  bounds on the moments of  $\{e_{k+1}^n\}$. Finally, in Section \ref{sec:r errors} we estimate  the remainder terms $\{r_{k+1}^n\}$.

\subsection{Covariance structure of the error terms}\label{sec:covariance} 
We first study the covariance structure of the error terms  $\{e_{k+1}^n\}$.
Consider the collection of $d\times d$ matrices $\{F_{\theta}(z) : \theta \in \mathcal{P}(\Delta^o), z\in \Delta^o\}$ defined by, for $x,y \in \Delta^o, \; (\theta, z) \in \mathcal{P}(\Delta^o)\times \Delta^o$,
\begin{equation}\label{eq:eq444}
F_{\theta}(z)_{x,y} \doteq \sum_{u\in\Delta^o}( K[\theta]_{z,u}Q[\theta]_{u,x}Q[\theta]_{u,y}) - \left(K[\theta]Q[\theta]\right)_{z,y}\left(K[\theta]Q[\theta]\right)_{z,x},
\end{equation}
and let $U_*$ be the $d\times d$ matrix defined as
\begin{equation}\label{eq:ustar}
U_* \doteq \sum_{w\in \Delta^o}F_{\theta_*}(w)(\theta_*)_w.
\end{equation}
It is easily verified that $U_*$ is a nonnegative definite matrix.
The following result gives an expression for the conditional covariance matrix of $e_{k+1}^n$.
\begin{proposition}\label{prop:clt2 A} 
For each $n \in \NN$, $0\le k \leq n-1$, and $x, y \in \Delta^o$,
			\begin{equation}\label{eq:eproduct5}
			\EE\left[ e^{n}_{k+1} (x) e^{n}_{k+1} (y) \,\middle\vert\, \mathcal{F}^n_k\Big. \right] = \frac{1}{a(n)} \left( (U_*)_{x,y} + (D^{(1),n}_{k})_{x,y} + (D^{(2),n}_{k})_{x,y}\right),
			\end{equation}
			where the following hold:
			\begin{enumerate}[(i)]
			\item There is a $C_1\in (0,\infty)$ such that for all $n \in \NN$ and $0\le k \leq n-1$, $\| D^{(1),n}_k \| \leq C_1 \| \theta^{n}_k - \theta_*\|$.
			\item There are some $C_2, \beta \in (0,\infty)$ such that for  all $n \in \NN$ and $1\le k \leq n$, 
			$$
			\gamma_{k} \EE \left\| \sum_{m=1}^{k} D^{(2),n}_{m-1}\right\| \leq C_2 k^{-\beta}
			$$
			\end{enumerate}
\end{proposition}

\begin{proof}
	Fix $n \in \NN$ and $0\le k \leq n-1$.
Then, from \eqref{eq:eq242}, for each $x,y  \in \Delta^o$ and $1 \leq i \leq a(n)$,
 $$
\EE\left[ Q[\theta_k^n]_{X^{i,n}_{k+1},x} | \mathcal{F}_k^n\right] =\left(K[\theta_k^n]Q[\theta_k^n]\right)_{X_k^{i,n},x},
$$
 and
\begin{equation}\label{eq:eproduct3}
\begin{split}
\EE\left[ Q[\theta_k^n]_{X_{k+1}^{i,n},x}Q[\theta_k^n]_{X_{k+1}^{i,n},y} | \mathcal{F}_k^n\right] 
&= \sum_{w\in\Delta^o} K[\theta_k^n]_{X^{i,n}_k,w}Q[\theta_k^n]_{w,x}Q[\theta_k^n]_{w,y}.
\end{split}
\end{equation}
Similarly, for each $x,y \in \Delta^o$, if $1 \leq i \neq j \leq a(n)$, then
\begin{equation}\label{eq:eproduct2}
\begin{split}
\EE\left[ Q[\theta_k^n]_{X_{k+1}^{i,n},x}Q[\theta_k^n]_{X_{k+1}^{j,n},y} | \mathcal{F}_k^n \right] &= \left(K[\theta_k^n]Q[\theta_k^n]\right)_{X_k^{i,n},x}\left(K[\theta_k^n]Q[\theta_k^n]\right)_{X_k^{j,n},y}.
\end{split}
\end{equation}
Therefore, for each $x,y \in \Delta^o$,
\begin{equation}\label{eq:eproduct1}
\begin{aligned}
&\EE\left[  e_{k+1}^{n}(x)e_{k+1}^{n}(y) | \mathcal{F}_k^n \right]\\
 &= \frac{1}{a(n)^2} \sum_{i,j=1}^{a(n)}\EE\bigg[ Q[\theta_k^n]_{X_{k+1}^{i,n},x}Q[\theta_k^n]_{X_{k+1}^{j,n},y} - Q[\theta_{k}^{n}]_{X_{k+1}^{i,n},x}\left(K[\theta_k^n]Q[\theta_k^n]\right)_{X_{k}^{j,n},y} \\
&\qquad\qquad - Q[\theta_k^n]_{X_{k+1}^{j,n},y}\left(K[\theta_k^n]Q[\theta_k^n]\right)_{X_{k}^{i,n},x} + \left(K[\theta_k^n]Q[\theta_k^n]\right)_{X_{k}^{i,n},x}\left(K[\theta_k^n]Q[\theta_k^n]\right)_{X_{k}^{j,n},y} \bigg| \mathcal{F}_k^n \bigg]  \\
&= \frac{1}{a(n)^2} \sum_{i,j=1}^{a(n)}\EE\bigg[ Q[\theta_k^n]_{X_{k+1}^{i,n},x}Q[\theta_k^n]_{X_{k+1}^{j,n},y} - \left(K[\theta_k^n]Q[\theta_k^n]\right)_{X_{k}^{i,n},x}\left(K[\theta_k^n]Q[\theta_k^n]\right)_{X_{k}^{j,n},y} \bigg| \mathcal{F}_k^n \bigg]  \\
&= \frac{1}{a(n)^2} \sum_{i=1}^{a(n)}  \left(  \sum_{w\in\Delta^o} K[\theta_k^n]_{X_k^{i,n},w}Q[\theta_k^n]_{w,x}Q[\theta_k^n]_{w,y} - \left(K[\theta_k^n]Q[\theta_k^n]\right)_{X_k^{i,n},y}\left(K[\theta_k^n]Q[\theta_k^n]\right)_{X_k^{i,n},x}\right),
\end{aligned}
\end{equation}
where in the last line we have used \eqref{eq:eproduct3} and \eqref{eq:eproduct2}. Recalling the definition of $F_{\theta}(\cdot)$ we now see that
\begin{equation}\label{eq:eproduct4}
\begin{split}
&\EE[e_{k+1}^{n}(x)e_{k+1}^{n}(y) |\mathcal{F}_k^n] = \frac{1}{a(n)^2} \sum_{i=1}^{a(n)} F_{\theta_k^n}(X_k^{i,n})_{x,y}.
\end{split}
\end{equation}
We can write
$$
F_{\theta_k^n}(X_k^{i, n})_{x,y} = (U_*)_{x,y} + (D^{(1),i,n}_{k})_{x,y} + (D^{(2),i,n}_k)_{x,y},
$$
where
$$
D_k^{(1),i,n} \doteq \sum_{w\in \Delta^o} (F_{\theta_k^n}(w)\pi(\theta_k^n)_w - F_{\theta_*}(w)(\theta_*)_w), \;\; D_k^{(2),i,n} \doteq F_{\theta_k^n}(X_k^{i,n}) - \sum_{w\in \Delta^o}F_{\theta_k^n}(w)\pi(\theta_k^n)_w.
$$

Letting
$$
D_k^{(1),n} \doteq \frac{1}{a(n)} \sum_{i=1}^{a(n)} D^{(1),i,n}_k = \sum_{w\in\Delta^o} \left(F_{\theta_k^n}(w)\pi(\theta_k^n)_w - F_{\theta_*}(w)(\theta_*)_w \right),
$$
$$
D_k^{(2),n} \doteq  \frac{1}{a(n)} \sum_{i=1}^{a(n)} D^{(2),i,n}_k,
$$
and using (\ref{eq:eproduct4}), we obtain the identity in \eqref{eq:eproduct5}.\\

\noindent\emph{Proof of  Claim (i):}
Since $\pi(\theta_*) = \theta_*$, we have that
\begin{equation}\label{eq:eq939}
D^{(1),n}_k =  \sum_{w\in \Delta^o} \left(F_{\theta_k^n}(w)\pi(\theta_k^n)_w - F_{\theta_*}(w)\pi(\theta_*)_w\right).
\end{equation}
Since the maps $\theta \mapsto K[\theta]$, $\theta \mapsto Q[\theta]$ are bounded and Lipschitz, it follows that $\theta \mapsto F_{\theta}(w)$
is a bounded and Lipschitz map as well for every $w \in \Delta^o$. Also, $\theta \mapsto \pi(\theta)$ is a bounded and Lipschitz map 
(see e.g. \cite{benclo}*{Corollary 2.3}).
Combining these facts we see that $\theta \mapsto \sum_{w\in \Delta^o} F_{\theta}(w) \pi(\theta)_w$ is a bounded and Lipschitz map as well.
The claim in (i) is now immediate from the representation of $D_k^{(1),n}$ in \eqref{eq:eq939}.

\noindent\emph{Proof of Claim (ii):}
It suffices to show that for each $(u,v) \in \Delta^o \times \Delta^o$, there are some $C_2, \beta >0$ such that for all $n \in \NN$ and $1 \leq k \leq n$,
$$
\gamma_{k}\EE\left\| \sum_{m=1}^{k} (D^{(2),n}_{m-1})_{u,v}\right\| \leq C_2 k^{-\beta}.
$$

Now fix $(u,v) \in \Delta^o \times \Delta^o$ and, abusing notation, denote $(D^{(2),n}_m)_{u,v}$ once more as $D^{(2),n}_m$. By another abuse of notation, denote the $(u,v)$-th coordinate
of $F_{\theta}(x)$, for  $\theta \in \mathcal{P}(\Delta^o)$ and $x \in \Delta^o$, by $F_{\theta}(x)$ as well. For $\theta \in\mathcal{P}(\Delta^o)$ let $U_{\theta} \in \RR^d$ be the vector whose $x$-th coordinate is given by  $U_{\theta}(x) \doteq (Q[\theta]F_{\theta})(x)$. By the Poisson equation \eqref{eq:pois} we have that
\begin{equation*}
    \begin{split}
[(I - K[\theta])U_{\theta}](x) &=  F_{\theta}(x) - \sum_{w \in \Delta^o} F_{\theta}(w) \pi(\theta)_w.
\end{split}
\end{equation*}
Therefore, if we let
$$
D^{(2,a),i,n}_k \doteq U_{\theta^n_k}(X_{k+1}^{i,n}) - (K[\theta^n_k]U_{\theta^n_k})(X_{k}^{i,n}),
\qquad  
D^{(2,b),i,n}_k \doteq U_{\theta^n_k}(X_k^{i,n}) - U_{\theta^n_k}(X_{k+1}^{i,n}),
$$
and
$$
D^{(2,a),n}_k \doteq \frac{1}{a(n)}\sum_{i=1}^{a(n)}D_k^{(2,a),i,n},
\qquad  
D^{(2,b),n}_k \doteq \frac{1}{a(n)}\sum_{i=1}^{a(n)}D_k^{(2,b),i,n},
$$
then $D^{(2),n}_k = D^{(2,a),n}_k + D^{(2,b),n}_k$. Note that with $\mathcal{G}_k^n \doteq \mathcal{F}_{k+1}^n$, we have that, for each fixed $n \in \NN$, $\{D^{(2,a),n}_k\}_{k=1}^{\infty}$ is a $\mathcal{G}_k^n$-martingale increment sequence.
 Applying  Burkholder's inequality, we see that, for some $\kappa_1 \in (0,\infty)$,
and for all $n \in \NN$,
\begin{equation}\label{eq:d2a1}
\begin{split}
\EE \left[\left |  \sum_{m=1}^k D^{(2,a),n}_{m-1} \right|^2 \right] &\leq  \kappa_1 \sum_{m=1}^{k} \EE\left| D^{(2,a),n}_{m-1}\right|^2
\leq \kappa_1  \left(\frac{1}{a(n)}\right)^2 \sum_{m=1}^k  \sum_{i,j=1}^{a(n)} \EE\left[ D^{(2,a),i,n}_{m-1} D^{(2,a),j,n}_{m-1}\right].\\
\end{split}
\end{equation}
For $i \neq j$, we have  by a conditioning argument, and using \eqref{eq:eq242},  that
\begin{equation}\label{eq:d2a2}
\begin{split}
\EE\left[ D^{(2,a),i,n}_m D^{(2,a),j,n}_m \right] = 0.
\end{split}
\end{equation}
Let
$$
\kappa_2 \doteq \sup_{\theta \in \mathcal{P}(\Delta^o), x,y \in \Delta^o} | U_{\theta}(x) - (K[\theta]U_{\theta})(y) |^2 < \infty.
$$
Then
\begin{equation}\label{eq:d2a3}
\begin{split}
\EE\left[ \left( D_{m}^{(2,a),i,n}\right)^2 \right] &= \EE\left[ \left( U_{\theta_m^{n}}(X^{i,n}_{m+1}) - (K[\theta_m^{n}]U_{\theta_m^{n}})(X^{i,n}_m)\right)^2 \right] 
\leq \kappa_2.
\end{split}
\end{equation}
Combining (\ref{eq:d2a1}), (\ref{eq:d2a2}), and (\ref{eq:d2a3}) we see that with $\kappa_3 = \kappa_1\kappa_2$, 
$$
\EE \left[\left|  \sum_{m=1}^k D^{(2,a),n}_{m-1} \right|^2 \right] 
\leq \kappa_1 \sum_{m=1}^k  \frac{1}{a(n)^2}  \sum_{i=1}^{a(n)} \EE\left[ \left( D^{(2,a),i,n}_{m-1}\right)^2\right]
\leq \kappa_3  \frac{k}{a(n)}.
$$
Applying the Cauchy-Schwarz inequality, we have, for some $\kappa_4 \in (0,\infty)$, and all $n\in \NN$, $0\le k \le n$,
\begin{equation}\label{eq:d2abound}
\gamma_{k} \EE\left| \sum_{m=1}^k D^{(2,a),n}_{m-1}\right| 
\leq \kappa_3^{1/2} \gamma_{k} \left( \frac{k}{a(n)}\right)^{1/2} = \left( \kappa_4 \frac{\gamma_{k}}{a(n)}\right)^{1/2}.
\end{equation}
We now consider $\{D_k^{(2,b),n}\}$. Letting
$
\kappa_5 \doteq \sup_{\theta \in \mathcal{P}(\Delta^o), z\in \Delta^o}|U_{\theta}(z)| < \infty,
$
\begin{equation*}
\begin{split}
\left| \sum_{m=1}^k D^{(2,b),n}_{m-1} \right|  &= \left| \sum_{m=1}^k \frac{1}{a(n)} \sum_{i=1}^{a(n)} D^{(2,b),i,n}_{m-1}\right|\\
&= \left| \frac{1}{a(n)} \sum_{i=1}^{a(n)} \sum_{m=1}^k  \left( U_{{\theta}^{n}_{m-1}}(X^{i,n}_{m-1}) - U_{\theta^n_{m-1}}(X^{i,n}_{m}) \right) \right|\\
&= \left| \frac{1}{a(n)} \sum_{i=1}^{a(n)} \left( U_{\theta^n_0}(X^{i,n}_0) - U_{\theta^n_{k-1}}(X^{i,n}_{k}) + \sum_{m=1}^k \left( U_{\theta_m^{n}} (X^{i,n}_m) - U_{\theta_{m-1}^{n}}(X^{i,n}_m) \right) \right)\right|\\
&\leq 2\kappa_5 + \frac{1}{a(n)}\sum_{i=1}^{a(n)} \left| \sum_{m=1}^{k} \left(U_{\theta_m^{n}} (X^{i,n}_m) - U_{\theta_{m-1}^{n}}(X^{i,n}_m)\right) \right|
\end{split}
\end{equation*}
Since the maps $\theta \mapsto K[\theta]$ and $\theta \mapsto Q[\theta]$ are bounded Lipschitz maps, there is a 
$\kappa_6 \in (0,\infty)$ such that for all $x \in \Delta^o$ and $\theta,\theta' \in \mathcal{P}(\Delta^o)$, 
$
|U_{\theta}(x) - U_{\theta'}(x)| \leq \kappa_6 \|\theta - \theta'\|.
$
Observe that
\begin{equation}\label{eq:thetadiffa}
\| \theta^n_m - \theta^n_{m-1}\| = \gamma_{m} \left\| a(n)^{-1}\sum\limits_{i=1}^{a(n)} \delta_{X^{i,n}_{m}} - \theta^n_{m-1}\right\| \leq 2 \gamma_m,
\end{equation}
so
for some $\kappa_7\in (0,\infty)$,
\begin{equation}\label{eq:d2bbound}
\begin{split}
\left| \sum_{m=1}^k D^{(2,b),n}_{m-1}\right| &\leq \left(2\kappa_5 + \frac{\kappa_6}{a(n)}\sum_{i=1}^{a(n)} \sum_{m=1}^k \big\| \theta^n_{m} - \theta^n_{m-1} \big\| \right)
\leq\kappa_7\left( 1 + \sum_{m=1}^k \gamma_{m}\right).
\end{split}
\end{equation} 
Combining (\ref{eq:d2abound}) and (\ref{eq:d2bbound}) we see that for some $\kappa_8\in (0,\infty)$,
$$
\gamma_{k} \EE \left\| \sum_{m=1}^k D^{(2),n}_{m-1}\right\| \leq \kappa_8\left[ \left( \frac{\gamma_{k}}{a(n)}\right)^{1/2} +  \gamma_{k} \sum_{m=1}^{k} \gamma_{m}\right].
$$
The result follows.
\end{proof}

\subsection{Bounds on the moments of the error terms} \label{sec:moments} 
The following result gives a useful  moment bound for $\{e_{k}^n\}$.

\begin{proposition} \label{prop:clt2 B and C} 
\ There exists $C \in (0,\infty)$ such that for all $n \in \NN$ and $0\le k \leq n-1$, $$\EE \| e^{n}_{k+1}\|^4 \leq  \frac{C}{a(n)^2}.$$
\end{proposition}

\begin{proof}
Recall that 
$$
e_{k+1}^n =  \frac{1}{\gamma_{k+1}} \delta^{1,n}_{k+1} =  \frac{1}{a(n)} \sum_{i=1}^{a(n)} \xi_i
$$
where  for each $1 \leq i \leq a(n)$ and $x\in \Delta^o$, 
$
 \xi_i(x) \doteq  Q[\theta_k^{n}]_{X^{i,n}_{k+1},x} - \left(K[\theta^{n}_k]Q[\theta^n_k]\right)_{X^{i,n}_k,x}.
$
The result is now immediate on observing that
if $1 \leq i_1, i_2, i_3, i_4 \leq a(n)$ and $i_4 \notin \{i_1,i_2,i_3\}$, then we have that 
$
\EE\big[ \xi_{i_1}(x) \xi_{i_2}(x) \xi_{i_3}(y)\xi_{i_4}(y)\big]
= 0,
$
and, for $x,y \in \Delta^o$,  $1 \leq i_1, i_2, i_3, i_4 \leq a(n)$, $0\le k \le n-1$ and $n \in \NN$,
$$
\big| \xi_{i_1}(x) \xi_{i_2}(x) \xi_{i_3}(y)\xi_{i_4}(y) \big| \leq 
\bigg( \sup_{\theta \in \mathcal{P}(\Delta^o)} \sup_{\substack{x_1,y_1\\x_2,y_2} \in \Delta^o} \Big|  Q[\theta]_{x_1,y_1} - (K[\theta]Q[\theta])_{x_2,y_2} \Big| \bigg)^4 < \infty.
$$
\end{proof}

\subsection{Analysis of the remainder terms}\label{sec:r errors} 
In this section we provide bounds to control the remainder terms $r_{k+1}^n$.
\begin{proposition}\label{prop:clt3}
We can write $r^n_{k+1} = r^{n,a}_{k+1} + r^{n,b}_{k+1}$, such that for some $C\in (0,\infty)$, and all $n\in \NN$ and $0\le k \le n-1$, 
$$
\text{(a)}\quad  \EE \left\| \frac{1}{\gamma_{k+1}} r^{n,a}_{k+1}\right\| \leq C 
\qquad\qquad\qquad
\text{(b)}\quad \left\| \;\sum_{i = k}^{n-1} r^{n,b}_{i+1}\, \right\| \leq C.
$$ 
 \end{proposition}

\begin{proof}
Recall that
$$
r^{n}_{k+1} \doteq  \frac{1}{\gamma_{k+1}} \left(\delta^{2,n}_{k+1} + \delta_{k+1}^{3,n} + \delta_{k+1}^{4,n}\right) = \frac{1}{a(n)} \sum_{j=1}^{a(n)} \left( \left(K[\theta^{n}_{k}]Q[\theta^{n}_k]\right)_{X_k^{j,n}, \cdot} - \left(K[\theta^{n}_k]Q[\theta^{n}_k]\right)_{X_{k+1}^{j,n},\cdot}\right).
$$

Rewrite this as $r^{n}_{k+1} = r^{n,a}_{k+1} + r^{n,b}_{k+1}$, where
\begin{equation}\label{eq:radef}
r^{n,a}_{k+1} \doteq \frac{1}{a(n)} \sum_{j=1}^{a(n)} \left( \left(K[\theta^{n}_{k+1}]Q[\theta^{n}_{k+1}]\right)_{X_{k+1}^{j,n}, \cdot} - \left(K[\theta^{n}_{k}]Q[\theta^{n}_{k}]\right)_{X_{k+1}^{j,n},\cdot}\right)
\end{equation}
and
\begin{equation}\label{eq:rbdef}
r^{n,b}_{k+1} \doteq \frac{1}{a(n)} \sum_{j=1}^{a(n)} \left( \left(K[\theta^{n}_{k}]Q[\theta^{n}_{k}]\right)_{X_k^{j,n}, \cdot} - \left(K[\theta^{n}_{k+1}]Q[\theta^{n}_{k+1}]\right)_{X_{k+1}^{j,n},\cdot}\right).
\end{equation}

\noindent\emph{Proof of  Claim (a):}
Since $\theta \mapsto K[\theta]$, $\theta \mapsto Q[\theta]$ are bounded Lipschitz maps,
 there is a $\kappa_1 \in (0,\infty)$ such that for all $\theta, \theta' \in \mathcal{P}(\Delta^o)$
$
\|K[\theta]Q[\theta] - K[\theta']Q[\theta']\| \leq \kappa_1 \| \theta - \theta'\|.
$
From this and \eqref{eq:thetadiffa} it follows that
\begin{equation*}
\begin{split}
 \| r^{n,a}_{k+1}\| 
&\leq \frac{\kappa_1}{a(n)} \sum_{j=1}^{a(n)} \| \theta^n_{k+1} - \theta^n_k\|
= \kappa_1 \| \theta^n_{k+1} - \theta^n_k\| \le 2\kappa_1 \gamma_{k+1},
\end{split}
\end{equation*}
which shows that $\EE \| \gamma_{k+1}^{-1} r^{n,a}_{k+1}\| \leq 2\kappa_1$.\\

\noindent\emph{Proof of  Claim (b):}
From the definition of $r^{n,b}_{i}$,
Note that
\begin{equation*}
\begin{split}
\sum_{i=k}^{n-1}r^{n,b}_{i+1} 
&= \frac{1}{a(n)} \sum_{j=1}^{a(n)} \left(  
 \left(K[\theta^{n}_k]Q[\theta^{n}_k]\right)_{X_k^{j,n},\cdot} - \left(K[\theta^{n}_{n}]Q[\theta^{n}_{n}]\right)_{X_{n}^{j,n},\cdot}\right)
\end{split}
\end{equation*}
from which it follows that 
$$
\left\| \sum_{i=k}^{n-1} r^{n,b}_{i+1} \right\| \leq \kappa_2 \doteq 2\sup_{\theta \in \mathcal{P}(\Delta^o)} \sup_{x\in \Delta^o} \big\| (K[\theta]Q[\theta])_{x,\cdot} \big\|.
$$
\end{proof}

\section{Central Limit Theorem for Algorithm I}\label{sec:clt}

The goal of this section is to prove Theorem \ref{thm:clt1}.  To do this we first study the linearized evolution \eqref{eq:mu}. Then, in Section \ref{sec:mucalc1}, we study the asymptotic behavior of the discrepancy  \eqref{eq:rho}. Finally, in Section  \ref{sec:proveclt} we present the proof of Theorem \ref{thm:clt1}.

\subsection{The linearized evolution}\label{sec:mucalc1}
Let, for $1\le k \le m <\infty$,
\begin{equation}\label{eq:psistarmatrixdef}
\psi_*(m,k) \doteq \prod_{j=k}^m (I + \gamma_j \nabla h(\theta_*)), \;\; \psi_*(m,m+1) \doteq I.
\end{equation}
Then by a simple recursion we see that for $0\le m \le n-1$,
\begin{equation}\label{eq:mu2}
\mu^{n}_{m+1} = \sum_{k=1}^{m+1} \gamma_k \psi_*(m+1,k+1) e^{n}_{k}.
\end{equation}
Furthermore, from \cite{for}*{Lemma 5.8}, with $L$ as introduced above \eqref{eq:stepsize}, for each $0<L'<L$ there is a $C(L')\in (0, \infty)$ such that
for all $n\in \NN$ and $1\le k \le n$
\begin{equation}\label{eq:eq104}
\| \psi_*(n,k)\| \leq C(L') \exp \left( -L' \sum_{j=k}^n \gamma_j\right).
\end{equation}
The following proposition provides some  useful bounds on $\{\mu_k^n\}$. 
\begin{proposition}\label{prop:muprop}
The following hold:
\begin{enumerate}[(i)]
\item With probability one, as  $n \to \infty$ we have $\mu_{n}^{n} \to 0$. Furthermore, for each $\po \in (0,1)$, and a.e. $\om$, there is some $\alpha > 0$
and $n_0(\om)\in \NN$ such that if $n\ge n_0(\om)$ and $n^{\po} \leq k \leq n$, then  $\| \mu_{k+1}^{n}\| \leq k^{-\alpha}$.
\item Suppose that $\gamma_* > (2L)^{-1}$. Then there is some $C > 0$ such that for all $n \in \NN$ and $ 0\le k \leq n-1$,  $$\EE \|\mu^{n}_{k+1}\|^2 \leq \frac{C \gamma_{k+1}}{a(n)}.$$
\end{enumerate}
\end{proposition}

\begin{proof}
The proof of (i) is similar to the proof of Theorem \ref{thm:mainrate} and is omitted for brevity. 

  Next, using (\ref{eq:mu2}), \eqref{eq:eq104}, and Proposition \ref{prop:clt2 B and C}, we see that for each $L' \in (0,L)$, there is a $\kappa_1(L')\in (0,\infty)$, such that, for $0\le m \le n-1$,
\begin{equation*}
\begin{split}
\EE\left[ \| \mu^{n}_{m+1}\|^2\right] &\leq \sum_{k=1}^{m+1} \gamma_k^2 \| \psi_*(m+1,k+1)\|^2 \EE\left[ \| {e}^{n}_k\|^2\right] \leq \frac{\kappa_1(L')}{a(n)} \sum_{k=1}^{m+1} \gamma_k^2  \exp\left( - 2L' \sum_{j=k+1}^{m+1}\gamma_j\right).\\
\end{split}
\end{equation*}
Choosing $L'\in (0,L)$ such that $L'\gamma_* >1/2$, and using the form of $\gamma_k$, we can find a $\kappa_2 \in (0,\infty)$ such that for all $0\le m \le n-1$,
$
\EE  \| \mu_{m+1}^{n}\|^2 \leq \frac{\kappa_2 \gamma_{m+1}}{a(n)}.
$
The result follows.
\end{proof}

\subsection{Analysis of the discrepancy}\label{sec:rhocalc1}

The following result is used to study the asymptotic behavior of $\{\rho^n_n\}$.
Following \cite{for}, consider for $\theta \in \clp(\Delta^o)$ the collection of $d\times d$ matrices $\{\mathbf{R}^{(n)}_i(\theta)\}_{i=1}^d$ defined as
\begin{equation}\label{eq:eq523}
	\mathbf{R}^{(n)}_i(\theta)[k,l] \doteq \int_0^1 \frac{1}{2} (1-t)^2 \frac{\partial^2 h_i}{\partial \theta_k \partial \theta_l}(\theta + t (\theta-\theta_*)) dt, \qquad 1\le k,\, l \le d.
\end{equation}
For $0\le j \le n$, we denote the random matrix $\mathbf{R}^{(n)}_i(\theta^n_j)$ as $R^{(n,j)}_i$.
Then, using Taylor's expansion, for $1\le i \le d$ and $0\le j \le n$,
$$
h_i(\theta^{n}_{j}) = \nabla h_i(\theta_*) \cdot (\theta^{n}_{j} - \theta_*) + (\theta^{n}_{j} - \theta_*)^T R_i^{(n, j)}(\theta^{n}_{j}-\theta_*).
$$
For brevity we write the above display in a vector form as
\begin{equation}\label{eq:tenseq}
h(\theta^{n}_{j}) = \nabla h(\theta_*) (\theta^{n}_{j} - \theta_*) + (\theta^{n}_{j} - \theta_*)^T R_{\bullet}^{(n, j)}(\theta^{n}_{j}-\theta_*).
\end{equation}

\begin{corollary}\label{cor:matrixcor1}
Let, for $1\le k \le n <\infty$,
\begin{equation*}
\begin{split}
\psi(n,k) &\doteq \prod_{j=k}^{n}\left(I+ \gamma_j \left( \nabla h(\theta_*) + 2( \mu_{j-1}^{n})^TR^{(n,j-1)}_{\bullet} + (\rho^{n}_{j-1})^T R^{(n, j-1)}_{\bullet}\right)\right).
\end{split}
\end{equation*}
Then for each $\po \in (0,1)$,  $L' \in (0,L)$, and a.e. $\om$, there is a   $C= C(\po,L',\om)\in (0,\infty)$ such that if  $n^{\po} \leq k \leq n$, then
$$
\| \psi(n,k)\| \leq C \exp \left( -L' \sum_{j=k}^n \gamma_j\right).
$$
\end{corollary}
\begin{proof}
Let $A \doteq \nabla h(\theta_*)$ and
\begin{align*}
A^n_j &\doteq \nabla h(\theta_*) + 2( \mu_{j-1}^{n})^TR^{(n,j-1)}_{\bullet} + (\rho^{n}_{j-1})^T R^{(n, j-1)}_{\bullet}\\
&=  \nabla h(\theta_*) +  ( \mu_{j-1}^{n})^TR^{(n,j-1)}_{\bullet} + ( \theta^{n}_{j-1} - \theta_*)^T R^{(n, j-1)}_{\bullet}, \;\; 
\end{align*}
so that with
$
\kappa_1 \doteq \sup_{n\in\NN} \sup_{\theta \in \clp(\Delta^o)}\max_{1\le i\le d}\|R_{i}^{(n)}(\theta)\|,
$
 we have
\begin{equation*}
\begin{split}
\|A^n_j - A\| &=  \left\|  ( \mu_{j-1}^{n})^TR^{(n,j-1)}_{\bullet} + ( \theta^{n}_{j-1} - \theta_*)^T R^{(n, j-1)}_{\bullet} \right\|
\leq \kappa_1 \left( \| \mu^{n}_{j-1}\| + \| \theta^{n}_{j-1} - \theta_*\|\right).
\end{split}
\end{equation*}
Fix $\po \in (0,1)$ and $L'\in (0,L)$. Applying Proposition \ref{prop:muprop}(i) and Theorem \ref{thm:mainrate}, choose $\alpha > 0$ and,
for a.e. $\om$, $n_1 \in \NN$ 
 such that if $n\ge n_1$, and $n^{\po} \leq j \leq n$, then $\| \mu^{n}_{j-1}\| \leq j^{-\alpha}$ and $\| \theta^{n}_{j-1} - \theta_*\| \leq j^{-\alpha}.$
Thus, for a.e. $\om$ there is an $n_0\in \NN$  such that for all $n\ge n_0$ and $n^{\po} \leq j \leq n$, 
$
\|A^n_j - A\| \leq j^{-\alpha/2}.
$
The result now follows from Lemma \ref{lem:matrixlem1}.
\end{proof}

Recall that $\sigma_n = \sqrt{a(n)/\gamma_n}$. The next result will be used to show that $\sigma_n \rho^{n}_n \stackrel{\PP}{\to} 0$ as $n \to \infty$.

\begin{lemma}\label{lem:rholem1}
Suppose that $\gamma_*> L^{-1}$ and $a(n)/n\to 0$ as $n\to \infty$. Then, for some $\kappa , \po \in (0,1)$, we have, as $n\to \infty$,
$$
 \sigma_n^{1+\kappa}\left[\rho^{n}_n  - \sum_{k=n^{\po}}^n \gamma_k \psi(n,k+1) r^{n}_k\right] \stackrel{\PP}{\to} 0
$$
and
\begin{equation}\label{eq:eq449r}
\sigma_n \left[\sum_{k=n^{\po}}^n \gamma_k \psi(n,k+1) r^{n}_k\right] \stackrel{\PP}{\to} 0.
\end{equation}
\end{lemma}

\begin{proof}
Fix $L' \in ( \gamma_*^{-1},L)$ and let $\po \in (0 , 1 - (L'\gamma_*)^{-1})$. 
 Using \eqref{eq:rep1}, \eqref{eq:mu}, \eqref{eq:rho}, \eqref{eq:tenseq} and a recursive argument,  we can write
$$\label{eq:rhon1}
\rho_{n}^{n} \doteq \psi(n,n^{\po})\rho^{n}_{n^{\po} -1} + \sum_{k=n^{\po}}^{n} \gamma_k \psi(n,k+1)\left[r^{n}_{k} + (\mu^{n}_{k-1})^TR_{\bullet}^{(n, k-1)}\mu^{n}_{k-1}\right],
$$
or equivalently,
\begin{equation}\label{eq:rhon12}
\rho_{n}^{n} -   \sum_{k=n^{\po}}^{n} \gamma_k \psi(n,k+1)r^{n}_{k}  = \psi(n,n^{\po})\rho^{n}_{n^{\po}-1} + \sum_{k=n^{\po}}^{n} \gamma_k \psi(n,k+1) (\mu^{n}_{k-1})^TR_{\bullet}^{(n, k-1)}\mu^{n}_{k-1}.
\end{equation}
Fix $\kappa \in \left( 0,  \frac{(1- \po)L'\gamma_* -1}{2} \wedge \frac{1}{4}\right)$. We begin by showing that
\begin{equation}\label{eq:eq505}
	 \sigma_n^{1+ \kappa} \psi(n, n^{\po}) \rho^{n}_{n^{\po}-1} \stackrel{\PP}{\to} 0.
\end{equation}
	  Since $\| \rho^{n}_{n^{\po}-1}\|$ is a bounded sequence, it is enough to show that $\sigma^{1+\kappa}_n \psi(n,n^{\po})$
	 converges to $0$ in probability. From Corollary \ref{cor:matrixcor1}, for a.e. $\om$,  there is a $\kappa_1(\om) \in (0,\infty)$ such that
\begin{equation}\label{eq:eq633}
\begin{aligned}
\sigma_n^{1 + \kappa} \|\psi(n,n^{\po})\| 
&\leq \kappa_1  (a(n)n)^{\frac{1}{2}(1+\kappa)}\exp\left( - L' \sum_{j = n^{\po}}^{n} \gamma_j\right) 
= \kappa_1 \left( \frac{a(n)}{n}\right)^{\frac{1}{2}(1 + \kappa)} n^{1 +  \kappa} \exp\left( - L' \sum_{j=n^{\po}}^{n} \gamma_j\right).
\end{aligned}
\end{equation}
From the definition of $\gamma_k$ we see that for all $1\le k \le n$
\begin{equation}\label{eq:eq628}
	\exp\left( - L' \sum_{j = k}^{n} \gamma_j\right) \le \left(\frac{N_*+k}{N_*+n}\right)^{L'\gamma_*}.
\end{equation}
From our choice of $\kappa$,
$
(\po - 1)L'\gamma_* + 1 + \kappa < 0,
$
and so we have,  on applying \eqref{eq:eq628} with $k= n^{\po}$, that the expression in \eqref{eq:eq633} converges to $0$. This completes the proof of \eqref{eq:eq505}. We now show that 
\begin{equation}\label{eq:rhomu1}
\sigma_n^{1 + \kappa} \sum_{k=n^{\po}}^{n} \gamma_k \psi(n,k+1) (\mu^{n}_{k-1})^T R^{(n,k-1)}_{\bullet}\mu_{k-1}^{n} \stackrel{\PP}{\to} 0.
\end{equation}
 Using the uniform-boundedness of $\{R_{\bullet}^{(n,k-1)}\}$ and Corollary \ref{cor:matrixcor1}, for a.e. $\om$, we can find a  $\kappa_2(\om) \in (0,\infty)$ such that
\begin{equation}\label{eq:rhomu2}
\begin{split}
\left\|  \sum_{k=n^{\po}}^{n} \gamma_k \psi(n,k+1)(\mu^{n}_{k-1})^TR_{\bullet}^{(n, k-1)}\mu^{n}_{k-1}\right\| &\leq \sum_{k=n^{\po}}^n  \gamma_k \| \psi(n,k+1)\| \| \mu_{k-1}^{n}\|^2 \| R^{(n,k-1)}_{\bullet}\|\\
&\leq \kappa_2 \sum_{k = n^{\po}}^{n} \gamma_k \exp\left( -L' \sum_{j=k}^{n} \gamma_j\right) \| \mu^{n}_{k-1}\|^2.
\end{split}
\end{equation}
From Proposition \ref{prop:muprop}, there is a  $\kappa_3 \in (0,\infty)$ such that for all $n \in \NN$ and $k \leq n$, $\EE\|\mu^n_{k-1}\|^2 \leq \kappa_3 \gamma_k / a(n)$. It follows that, for some $\kappa_4 \in (0,\infty)$, with $\tilde{\kappa} \doteq \frac{1}{2}(1+\kappa) < 1$,
\begin{multline}\label{eq:rhomu3}
 \sigma_n^{1+ \kappa}\EE \left( \sum_{k = n^{\po}}^{n} \gamma_k 
 \exp\left( -L' \sum_{j=k}^{n} \gamma_j \right) \|\mu^{n}_{k-1}\|^2 \right) \\
 \leq  \kappa_3 \sigma_n^{1 + \kappa} \sum_{k = n^{\po}}^{n} \gamma_k \exp\left( -L' \sum_{j=k}^{n} \gamma_j\right) \frac{\gamma_k}{a(n)}
 \le \kappa_4 n^{\tilde{\kappa}} a(n)^{\tilde{\kappa} - 1} \sum_{k=n^{\po}}^n \gamma_k^2 \exp\left( - L' \sum_{j=k}^n \gamma_j\right).\\
\end{multline}
Using \eqref{eq:eq628} once more, we can find some $\kappa_5 \in (0,\infty)$ such that the last expression is bounded above by
\begin{equation}\label{eq:eq350}
a(n)^{\tilde{\kappa}-1} (n+N_*)^{\tilde{\kappa} - L'\gamma_*} \sum\limits_{k=n^p}^n (k + N_*)^{L'\gamma_* - 2},
\end{equation}
which tends to $0$ as $n \to \infty$, since $\tilde{\kappa} < 1$. Combining this convergence with (\ref{eq:rhomu2}) and (\ref{eq:rhomu3}), we have  (\ref{eq:rhomu1}),
which together with \eqref{eq:eq505} proves the first statement in the lemma.

 We now prove the second statement. Let $r^{n,a}_i, r^{n,b}_i$ be as in (\ref{eq:radef}) and (\ref{eq:rbdef}), respectively,  so that ${r}^{n}_i = r^{n,a}_i + r^{n,b}_i$. Using Corollary \ref{cor:matrixcor1}, we can find, for a.e. $\om$,  some $\kappa_6 \equiv \kappa_6(\om) \in (0,\infty)$ such that 
\begin{equation}\label{eq:rhor11}
\begin{split}
\left\| \sum_{k=n^{\po}}^{n} \gamma_k \psi(n,k+1) r^{n,a}_k\right\| &\leq \kappa_6\sum_{k=n^{\po}}^n  \gamma_k^2 \exp\left( - L' \sum_{j=k}^{n} \gamma_j \right) \Big\| \frac{1}{\gamma_k} r^{n,a}_k\Big\|.\\
\end{split}
\end{equation}
Using Proposition \ref{prop:clt3}(a), we can find some $\kappa_7\in (0,\infty)$ such that for all $n \in \NN$ and $1 \le k \leq n$, $\EE \| r^{n,a}_k/\gamma_k \| \leq \kappa_7$.
Then, for some $\kappa_8\in (0,\infty)$,
\begin{equation}\label{eq:rhor12}
\begin{split}
\sigma_n \sum_{k=n^{\po}}^{n} \gamma_k^2 \exp\left( - L' \sum_{j=k}^n \gamma_j\right)  \EE \Big\| \frac{1}{\gamma_k}r^{n,a}_k \Big\| &\leq 
\kappa_7 \sigma_n \sum_{k=n^{\po}}^{n} \gamma_k^2 \exp\left( - L' \sum_{j=k}^{n} \gamma_j\right)\\
&\le \kappa_8 \sqrt{\frac{a(n)}{n}} n   \sum_{k=n^{\po}}^{n} \gamma_k^2 \exp\left( - L' \sum_{j=k}^{n} \gamma_j\right).\\
\end{split}
\end{equation}
As in \eqref{eq:eq350}, the last term can be bounded above by
$$\kappa_8 \sqrt{\frac{a(n)}{n}} n  \frac{\gamma_*^2}{(n+N_*)^{L'\gamma_*}}\sum_{k=n^{\po}}^n (k+N_*)^{L'\gamma_*-2}$$
which, since $a(n)=o(n)$, converges to $0$ as $n\to \infty$.
 Combining (\ref{eq:rhor11}) and (\ref{eq:rhor12}), we have that, as $n\to \infty$, 
\begin{equation}\label{eq:rhor1bound}
\sigma_n \left\| \sum_{k=n^{\po}}^n \gamma_k \psi(n,k+1) r^{n,a}_k\right\| \stackrel{\PP}{\to} 0.
\end{equation}
Now, consider the term
$
\sigma_n \left\|\sum_{k=1}^{n} \gamma_k \psi(n,k+1) r^{n,b}_k\right\|.
$
Define for $n \in \NN$ and $1\le k\le n$ 
$$
\Xi^{n}_k \doteq \sum_{i=1}^{k} r^{n,b}_i, \;\; H^{n}_k \doteq \nabla h(\theta_*) + 2 (\mu^{n}_k)^T R^{(n,k)}_{\bullet} + (\rho^{n}_k)^T R^{(n,k)}_{\bullet},
$$
and apply the summation by parts formula to see that
\begin{equation}\label{eq:rhor21a}
\sum_{k=n^{\po}}^{n} \gamma_k \psi(n,k+1) r^{n,b}_k = \gamma_n \Xi^{n}_n  - \gamma_{n^{\po}} \psi(n,n^{\po}+1) \Xi^{n}_{n^{\po}-1} - \sum_{k=n^{\po}}^{n-1} \Xi^{n}_k\left( \gamma_{k+1} \psi(n,k+2) - \gamma_k \psi(n,k+1)\right).\\
\end{equation}
Thus
\begin{equation}\label{eq:rhor21}
\begin{split}
&\sigma_n \sum_{k=n^{\po}}^{n} \gamma_k \psi(n,k+1) r^{n,b}_k\\
&= \sigma_n (\gamma_n \Xi^{n}_{n} -  \gamma_{n^{\po}} \psi(n,n^{\po}+1)\Xi^{n}_{n^{\po}-1}) - \sigma_n \sum_{k = n^{\po}}^{n-1} \Xi^{n}_k\left( \gamma_{k+1} \psi(n,k+2) - \gamma_k \psi(n,k+1)\right)\\
&= \sigma_n (\gamma_n \Xi^{n}_{n}  -   \gamma_{n^{\po}} \psi(n,n^{\po}+1)\Xi^{n}_{n^{\po}-1})\\
&\quad\quad- \sigma_n \sum_{k= n^{\po}}^{n-1} \Xi^{n}_k\left( \gamma_{k+1} \psi(n,k+2) - \gamma_k \psi(n,k+2)\left( I + \gamma_{k+1} H^{n}_k\right)\right)\\
&= \sigma_n (\gamma_n \Xi^{n}_{n}  -   \gamma_{n^{\po}} \psi(n,n^{\po}+1)\Xi^{n}_{n^{\po}-1})  + \sigma_n \sum_{k=n^{\po}}^{n-1} \gamma_k \gamma_{k+1} \Xi^{n}_k \psi(n,k+2) \left( \gamma_*^{-1}I + H^{n}_k\right).\\
\end{split}
\end{equation}
Applying Corollary \ref{cor:matrixcor1} and Proposition \ref{prop:clt3}, for a.e. $\om$, we can find a $\kappa_9 \equiv \kappa_9(\om) \in (0,\infty)$ such that
\begin{equation}\label{eq:rhor22}
\begin{split}
\sigma_n \gamma_{n^{\po}}\| \psi(n,n^{\po}+1) \Xi^{n}_{n^{\po}-1}\| &\leq \kappa_9 \sigma_n \gamma_{n^{\po}} \exp\left(- L' \sum_{j = n^{\po}+1}^{n} \gamma_j\right).\\
\end{split}
\end{equation}
Using \eqref{eq:eq628}, the expression in the previous display can be bounded by
$$\kappa_{10}\sqrt{na(n)} \frac{\gamma_*}{n^{\po}+N_*-1} \left(\frac{N_*+n^{\po}+1}{N_*+n}\right)^{L'\gamma_*},$$
which tends to $0$ as $n \to \infty$. 
Also,  using Proposition \ref{prop:clt3} we see that for some $\kappa_{11}\in (0,\infty)$
\begin{equation}\label{eq:rhor23}
\sigma_n \gamma_n\| \Xi^{n}_n \| \leq \kappa_{11} \left( \frac{a(n)}{n}\right)^{1/2},
\end{equation}
which, since $a(n)=o(n)$, also goes to $0$ as $n \to \infty$. Finally, note that
$$
A \doteq \sup_{n \in \NN}\sup_{n^{\po} \leq k \leq n} \left\| \left( \gamma_*^{-1}I + H^{n}_k\right) \right\| < \infty \text{ a.s.},
$$
which, together with Proposition \ref{prop:clt3}(b), ensures that for a.e. $\om$, there is a $\kappa_{12} \equiv \kappa_{12}(\om) \in (0,\infty)$ such that 
\begin{equation}\label{eq:rhor24}
\begin{split}
\left\|  \sigma_n \sum_{k=n^{\po}}^{n-1} \gamma_k \gamma_{k+1} \Xi^{n}_k \psi(n,k+2) \left( \gamma_*^{-1}I + H^{n}_k\right) \right\| 
&\leq  \kappa_{12} \sigma_n \sum_{k=n^{\po}}^{n-1} \gamma_k^{2} \exp\left( -L' \sum_{j=k}^{n} \gamma_j\right)\\
\end{split}
\end{equation}
which, as for \eqref{eq:rhor12}, goes to 0 as $n \to \infty$. 
Upon combining (\ref{eq:rhor21}), (\ref{eq:rhor22}), (\ref{eq:rhor23}), and (\ref{eq:rhor24}), we  see that, as $n\to \infty$,
$$
\sigma_n \left\| \sum_{k=n^{\po}}^n \gamma_k \psi(n,k+1) r^{n,b}_k \right\| \stackrel{\PP}{\to} 0.
$$
This, along with (\ref{eq:rhor1bound}), shows \eqref{eq:eq449r} and completes the 
proof of the lemma.
\end{proof}

\subsection{Proof of Theorem \ref{thm:clt1}}\label{sec:proveclt}
In order to prove Theorem \ref{thm:clt1}, it will be convenient to consider the array $\{Z_{n,k}, n\in \NN, 1\le k \le n\}$ defined for $n\in \NN$ and $1\le k \le n$  as
\begin{equation}\label{eq:xarray}
Z_{n,k} \doteq \sigma_{n} \gamma_k \psi_*(n,k+1)e_k^{n}. 
\end{equation}
From \eqref{eq:mu2} we see that
\begin{equation}\label{eq:eq925}
\sum_{k=1}^{n} Z_{n, k} = \sigma_{n} \mu_{n}^{n}.
\end{equation}
The next lemma is used to verify that a conditional Lindeberg condition holds for  $\{Z_{n,k}\}$. 

\begin{lemma}\label{lem:lindeberg}
Suppose that $\gamma_* > (2L)^{-1}$.
Then, as $n \to \infty$, we have
$
\sum_{k=1}^{n} \EE \|Z_{n,k}\|^4 \to 0.
$
\end{lemma}

\begin{proof}
From Proposition \ref{prop:clt2 B and C},  there is $\kappa_1 \in (0,\infty)$ such that for all $n \in \NN$ and all $1\le k \leq n$,
$
\EE\left[ \| e^{n}_k\|^4 \right] \leq \frac{\kappa_1}{a(n)^2}.
$
Now, fix $L' \in (0,L)$ such that $L'\gamma_*>1/2$, 
and recall from \eqref{eq:eq104} that for some $\kappa_2 \in (0,\infty)$, and for all $n \in \NN$ and $1\le k \leq n$,
$$
\| \psi_*(n,k+1)\| \leq \kappa_2 \exp\left( -L' \sum_{j=k}^{n}\gamma_j\right).
$$
Thus, for some $\kappa_3 \in (0,\infty)$ we  have that
\begin{multline}
\sum_{k=1}^{n} \EE  \| Z_{n,k}\|^{4} \le \sum_{k=1}^{n} \sigma_{n}^4 \gamma_k^4 \| \psi_*(n,k+1)\|^4 \EE \| e_k^{n}\|^4
\leq \kappa_3 \sum_{k=1}^{n} \sigma_{n}^4  \gamma_k^4 \exp\left( - 4L' \sum_{j=k}^{n} \gamma_j\right) \frac{1}{a(n)^2}\\
\le \kappa_3  \frac{n^2\gamma_*^4}{(N_*+n)^{4L'\gamma_*}}  \sum_{k=1}^{n} (N_*+k)^{4(L'\gamma_*-1)},
\end{multline}
which tends to $0$ as $n \to \infty$. The result follows.
\end{proof}
The next lemma is used in the proof of Theorem \ref{thm:clt1} to establish the form of the limiting covariance matrix $V$. 
Recall the matrix $U_*$ introduced in \eqref{eq:ustar}.
\begin{lemma}\label{lem:limitvar}
Suppose that $\gamma_* > (2L)^{-1}$. Define
$$
V^{(1)}_n \doteq \sigma_n^2 \sum_{k=1}^{n} \gamma_k^2 \psi_*(n,k+1)  \frac{U_*}{a(n)}\psi_*(n,k+1)^T, 
$$
and
$$
V^{(2)}_n \doteq \sigma_n^2 \sum_{k=1}^{n} \gamma_k^2 \psi_*(n,k+1) \left( \EE[ e^n_k (e^n_k)^T | \mathcal{F}^n_{k-1}] - \frac{U_*}{a(n)}\right)\psi_*(n,k+1)^T.
$$
Then $V_n^{(1)} \stackrel{\PP}{\to} V $ and $V_n^{(2)} \stackrel{\PP}{\to} 0$, where $V$ is the matrix given as the unique solution of the Lyapunov equation
\begin{equation}\label{eq:eq918}
U_* + \nabla h(\theta_*)V + V \nabla h(\theta_*)^T + \gamma_*^{-1}V = 0.
\end{equation}
\end{lemma}

\begin{proof}
We begin by noting that the  Lyapunov equation in \eqref{eq:eq918} has a unique solution.
 Indeed, note that $U_*$ is nonnegative definite and $ \nabla h(\theta_*) + (2 \gamma_*)^{-1}I$ is Hurwitz, as 
$
-L + (2\gamma_*)^{-1}  < 0.
$
The unique solvability of \eqref{eq:eq918} is now an immediate consequence of \cite{horjoh}*{Theorem 2.2.3}. Next, noting that 
$$
V_n^{(1)} = \gamma_n^{-1} \sum_{k=1}^{n} \gamma_k^2 \psi_*(n,k+1) U_*\psi_*(n,k+1)^T,
$$
 the proof that $V_n^{(1)} \stackrel{\PP}{\to} V$ as $n \to \infty$ follows from \cite{for} (see Section 5.4, {\em Limiting Variance}, therein).
  Now, recall that with the matrices $\{D^{(1),n}_k\} $ and $\{D^{(2),n}_k\}$ introduced in Proposition \ref{prop:clt2 A}, we can write
$$
 \EE[ e^n_k (e^n_k)^T | \mathcal{F}^n_{k-1}] - \frac{U_*}{a(n)} = \frac{ D^{(1),n}_k}{a(n)} + \frac{D^{(2),n}_k}{a(n)}.
$$
Thus, $V^{(2)}_n = V^{(2,a)}_n  + V^{(2,b)}_n$, where
$$
V^{(2,a)}_n \doteq \sigma_n^{2}  \sum_{k=1}^{n} \gamma_k^2 \psi_*(n,k+1) \frac{D^{(1),n}_{k-1}}{a(n)}\psi_*(n,k+1)^T =  \frac{1}{\gamma_n} \sum_{k=1}^{n} \gamma_k^2 \psi_*(n,k+1)   D^{(1),n}_{k-1} \psi_*(n, k+1)^T,
$$
and
$$
V^{(2,b)}_n \doteq \sigma_n^{2} \sum_{k=1}^n \gamma_k^2 \psi_*(n,k+1) \frac{D^{(2),n}_{k-1} }{a(n)}\psi_*(n,k+1)^T =  \frac{1}{\gamma_n} \sum_{k=1}^{n} \gamma_k^2 \psi_*(n,k+1)   D^{(2),n}_{k-1} \psi_*(n, k+1)^T.
$$
Using  part (i) of Proposition \ref{prop:clt2 A}, we can find some $\kappa_1 \in (0,\infty)$ such that for all $n \in\NN$ and $1\le k \leq n$, $\| D^{(1),n}_{k-1}\| \leq \kappa_1 \| \theta^n_{k-1} - \theta_*\|$. Also, for each $\po \in (0,1)$, from Theorem \ref{thm:mainrate}, we can find  $\alpha > 0$ such that for a.e. $\om$, there is an $n_0(\om)\in \NN$ such that for all $n \ge n_0(\om)$ and $n^{\po} \leq k \leq n$, $\|\theta^n_{k-1}(\om) - \theta_*\| \leq k^{-\alpha}$. 
Fix $L' \in (0,L)$ such that $L'\gamma_*>1/2$. 
Then, for $n \ge n_0(\om)$ and  some $\kappa_2 \in (0,\infty)$, we have
\begin{equation}\label{eq:v2rate}
\begin{split}
\| V_n^{(2,a)}\| &\leq  \frac{1}{\gamma_n} \sum_{k=1}^{n} \gamma_k^2   \| \psi_*(n, k+1)\|^2 \|D^{(1),n}_{k}\| \\
&\leq \kappa_2 \frac{1}{\gamma_n} \sum_{k=1}^{n} \gamma_k^2 \exp\left( - 2 L' \sum_{i=k}^{n} \gamma_i\right)  \| \theta^n_k - \theta_* \|\\
&= \kappa_2 \frac{1}{\gamma_n} \left( \sum_{k=1}^{n^{\po}-1}\gamma_k^2 \exp\left( - 2 L' \sum_{i=k}^{n} \gamma_i\right)  \| \theta^n_k - \theta_* \|  + \sum_{k=n^{\po}}^n \gamma_k^2 \exp\left( - 2 L' \sum_{i=k}^{n} \gamma_i\right)  \| \theta^n_{k-1} - \theta_* \|  \right)\\
&\leq 2\kappa_2 \frac{1}{\gamma_n} \left( \sum_{k=1}^{n^{\po}-1} \gamma_k^{2} \exp\left( -2 L' \sum_{i=k}^n \gamma_i\right) + \sum_{k=n^{\po}}^n \gamma_k^2 \exp\left( - 2L' \sum_{i=k}^n \gamma_i\right) k^{-\alpha}\right).\\
\end{split}
\end{equation}
We can find $\kappa_3 \in (0,\infty)$ such that first term on the last line is bounded above by
$$
\kappa_3 (n+N_*)^{-(2L'\gamma_*-1)}(n^{\po}+N_*)^{2L'\gamma_*-1},
$$
and such that the second term is bounded above by
$$
\kappa_3 \frac{(N_*+n)}{(N_*+n)^{2L'\gamma_*}}\sum_{k=n^{\po}}^n k^{-\alpha} (k+N_*)^{2L'\gamma_* -2}.$$
Since $L'\gamma_*>1/2$, both of these terms converge to $0$ as $n\to \infty$ and
so we have that $\| V_n^{(2,a)}\| \to 0$ almost surely, as $n \to \infty$. Now, let
$
\Xi^{n}_k \doteq \sum_{j=1}^k D^{(2),n}_{j-1}
$.
Using  summation by parts we have
\begin{equation}\label{eq:v2b1}
\begin{split}
V^{(2,b)}_{n} &= \gamma_{n} \Xi^{n}_{n}  + \frac{1}{\gamma_{n}} \sum_{k=1}^{n-1}\bigg( \gamma_k^2  \psi_*(n,k+1) \Xi^{n}_k \psi_*(n,k+1)^T
- \gamma_{k+1}^2 \psi_*(n,k+2) \Xi^{n}_{k} \psi_*(n,k+2)^T\bigg).
\end{split}
\end{equation}
From Proposition \ref{prop:clt2 A}(ii) we have that for some  $\kappa_4, \beta \in (0,\infty)$ and all  $n \in \NN$,  $1\le k \leq n$, 
$
\gamma_k \EE\| \Xi^{n}_k \| \leq \kappa_4 k^{-\beta},
$
and so 
\begin{equation}\label{eq:eq1114}
	\gamma_{n} \Xi^{n}_{n} \stackrel{\PP}{\to} 0 \mbox{ as } n\to \infty.
\end{equation} 
Let 
$\tilde{V}^{(2,b)}_{n} = V^{(2,b)}_{n} - \gamma_{n} \Xi^{n}_{n}$.
 Following \cite{dey}, if we let $A_1 \doteq \gamma_{k} \psi_*(n,k+1)$, $A_2 \doteq \gamma_{k+1} \psi_*(n,k+2)$, and $B \doteq \Xi^n_k$, then using the inequality
$$
\| A_1 B A_1^T - A_2 B A_2^T\| = \| (A_1 - A_2)B A_1^T - A_2 B (A_2^T - A_1^T)\| \leq  \| A_1 - A_2 \| \|B\| ( \|A_1\| + \|A_2\|),
$$
we see that
\begin{multline}\label{eq:v2bound1}
\| \tilde{V}^{(2,b)}_{n}\|\\
\leq \frac{1}{\gamma_{n}} \sum_{k=1}^{n-1}\bigg(  \|  \gamma_{k} \psi_*(n,k+1) - \gamma_{k+1} \psi_*(n,k+2)\| \cdot \| \Xi^n_k\| 
\cdot \big( \|  \gamma_{k} \psi_*(n,k+1)\| + \|\gamma_{k+1} \psi_*(n,k+2)\|\big) \bigg).
\end{multline}
Furthermore,  using the fact that $\gamma_{k+1} - \gamma_{k+2} \le \gamma_{k+1}^2 /\gamma_*$, we can find some $\kappa_5 \in (0,\infty)$ such that
\begin{multline}\label{eq:v2bound2a}
\|  \gamma_{k} \psi_*(n,k+1) - \gamma_{k+1} \psi_*(n,k+2)\|  
 = \| \gamma_{k}( I + \gamma_{k+1}\nabla h(\theta_*)) - \gamma_{k+1}I\| \| \psi_*(n,k+2)\|  \\
 \leq\left(  \gamma_{k} - \gamma_{k+1} + \| \nabla h(\theta_*)\| \gamma_{k}\gamma_{k+1}\right)\|\psi_*(n,k+2)\|  
 \leq \kappa_5 \gamma_{k+1}^2 \| \psi_*(n,k+2)\|.
\end{multline}
Additionally, there is some $\kappa_6\in (0,\infty)$ such that for all $n \in \NN$ and $1\le k \leq n -1$, 
\begin{equation}\label{eq:v2bound3a}
\| \gamma_{k} \psi_*(n,k+1)\| + \| \gamma_{k+1} \psi_*(n,k+2)\| \leq \kappa_6 \gamma_{k+1} \| \psi_*(n,k+2)\|.
\end{equation}
Using (\ref{eq:v2bound1}), (\ref{eq:v2bound2a}) and (\ref{eq:v2bound3a}) we see  that
\begin{align*}
\| \tilde{V}^{(2,b)}_{n}\|  
&\leq  \kappa_5 \kappa_6 \frac{1}{\gamma_{n}} \sum_{k=0}^{n-1} \gamma_{k+1}^2 \|\psi_*(n,k+2)\|^2 \gamma_{k+1}\|\Xi^n_k\|. 
\end{align*}
Thus, from Proposition \ref{prop:clt2 A}(ii),  for some $\kappa_7, \kappa_8\in (0,\infty)$, 
\begin{equation*}
\begin{split}
\EE \| \tilde{V}^{(2,b)}_{n}\| \leq \kappa_7 \frac{1}{\gamma_{n}} \sum_{k=0}^{n-1} \gamma_{k+1}^{2}\exp\left( - 2L' \sum_{j=k+2}^{n} \gamma_j\right) k^{-\beta}
\leq \frac{\kappa_7\gamma_*}{(n+N_*)^{2L'\gamma_*-1}}  \sum_{k=0}^{n-1} (k+N_*)^{2(L'\gamma_*-1)} k^{-\beta},
\end{split}
\end{equation*}
which goes to $0$ as $n \to \infty$. Combining the above convergence with \eqref{eq:eq1114},
it follows from \eqref{eq:v2b1} that $V^{(2,b)}_{n} \stackrel{\PP}{\to} 0$ as $n \to \infty$. The result follows.
\end{proof}

We now complete the proof of Theorem \ref{thm:clt1}.

\noindent\emph{Proof of Theorem \ref{thm:clt1}.} From \eqref{eq:rho} we see that
$
\sigma_n (\theta_n^n - \theta_*) = \sigma_n \mu_n^n + \sigma_n \rho_n^n. 
$
Also, from Lemma \ref{lem:rholem1}, $\sigma_n \rho_n^n \stackrel{\PP}{\to} 0$ as $n \to \infty$. Thus, it suffices to show that 
 $\sigma_n \mu_n^n \stackrel{\mathcal{L}}{\to} \mathcal{N}(0, V)$ where $V$ is as in the statement of the theorem.
  Recall the martingale difference array $\{Z_{n,k}\}$ introduced in (\ref{eq:xarray}), and note from \eqref{eq:eq925} that
$
\sigma_n \mu^n_n = \sum_{k=1}^{n} Z_{n,k}.
$
In order to complete the proof we apply \cite{halhey}*{Corollary 3.1}. From Lemma \ref{lem:lindeberg} it follows that for each $\epsilon > 0$,
$$
\sum_{k=1}^n \EE \left[ \|Z_{n,k}\|^2 \mathbbm{1}_{\| Z_{n,k} \|\geq \epsilon} \;\middle|\; \mathcal{F}^{n}_{k-1}\right] \stackrel{\PP}{\to} 0,
$$
as $n \to \infty$. Additionally, if we let $\{V^{(1)}_n\}$, $\{V^{(2)}_n\}$ and $V$ be as in Lemma \ref{lem:limitvar}, then we have from this lemma that
$$
\sum_{k=1}^{n} \EE\left[ Z_{n,k} Z_{n,k}^T | \mathcal{F}^n_{k-1}\right] = V^{(1)}_n + V^{(2)}_n \stackrel{\PP}{\to} V,
$$
as $n \to \infty$. 
Therefore, the conditions of \cite{halhey}*{Corollary 3.1} are satisfied, proving that
$
\sum_{k=1}^{n} Z_{n,k} \stackrel{\mathcal{L}}{\to} \mathcal{N}(0,V),
$
as $n \to \infty$. The result follows.\qed

\section{Convergence of Algorithm II}\label{sec:branch}
In this section we prove the a.s. convergence of  Algorithm II introduced in Section \ref{sec:algdescription}.
Namely we provide the proof of Theorem \ref{thm:branchlln1}.
Recall that in this method, we  initialize the scheme with  a single particle and as time progresses, particles are added to the system and the total time occupation measure of all particles is used to update the SA estimate. 
The goal of the section is to prove that $\btheta_n \to \theta_*$ as $n \to \infty$, where
$\btheta_n$ is as introduced in \eqref{eq:branchalg2}. The proof idea is similar to that in \cite{benclo}. We introduce the continuous-time process $\{\hat{\btheta}(t)\}$ given by
$$
\hat{\btheta}(\tau_n + t) \doteq \btheta_n + t \frac{\btheta_{n+1}  - \btheta_n}{\tau_{n+1} - \tau_n}, \;\; t \in [0, \gamma_{n+1}), \; n \in \NN_0,
$$
where the sequence $\{\tau_n\}$ is defined in (\ref{eq:updatetime}).  As before, for each $\nu \in \mathcal{P}(\Delta^o)$, we denote the solution to the ODE (\ref{eq:phiflow}) by $\{\Phi_t(\nu)\}$. We now recall the notion of an asymptotic pseudo-trajectory for a single trajectory \cites{ben,benclo}. Recall the space $\mathcal{C}^0$ from Section \ref{sec:notat}.
\begin{defn}
A trajectory $X \in \mathcal{C}^0$ is an asymptotic pseudo-trajectory of $\Phi$ if for all $T > 0$,
\begin{equation}\label{eq:singleasymptotic}
    \begin{split}
    \limsup_{t\to\infty} \sup_{0\leq u\leq T}\| X(t + u) - \Phi_u(X(t))\| = 0.
    \end{split}
\end{equation}
\end{defn}
In order to prove Theorem \ref{thm:branchlln1} we will show that $\{\hat{\btheta}\}$ is a.s. an asymptotic pseudo-trajectory of $\Phi$. For this, we begin by 
decomposing algorithm's noise in the following manner: for each $n \in \NN_0$, $1 \leq i \leq a(n+1)$, and $x \in \Delta^o$, let

\begin{equation}\label{eq:delta-defs-branch}
\delta^{\ell,i}_{n+1}(x) \doteq \left\lbrace
\begin{aligned}
&\gamma_{n+1}\,Q[\btheta_n]_{\bX_{n+1}^i,x}\;
-\gamma_{n+1}\,\left(K[\btheta_n]Q[\btheta_n]\right)_{\bX_n^i,x}
&\qquad\ell = 1&\;\\
&\gamma_{n+1}\,\left(K[\btheta_n]Q[\btheta_n]\right)_{\bX_n^i,x}
\;-\gamma_{n}\,\left(K[\btheta_n]Q[\btheta_n]\right)_{\bX_n^i,x}
&\qquad\ell= 2&\;\\
&\gamma_{n}\,\left(K[\btheta_n]Q[\btheta_n]\right)_{\bX_n^i,x}
\;-\gamma_{n+1}\left(K[\btheta_{n+1}]Q[\btheta_{n+1}]\right)_{\bX_{n+1}^i,x}
&\qquad\ell = 3&\;\\
&\gamma_{n+1}\,\left(K[\btheta_{n+1}]Q[\btheta_{n+1}]\right)_{\bX_{n+1}^i,x}
\;-\gamma_{n+1}\,\left(K[\btheta_n]Q[\btheta_n]\right)_{\bX_{n+1}^i,x}
&\qquad\ell  = 4&\;
\end{aligned}
\right.
\end{equation}
where $Q(\cdot)$ denotes the solution to the Poisson equation in (\ref{eq:pois}).

For  $1 \leq \ell \leq 4$, let $$\delta^{\ell}_{n+1} \doteq  \frac{1}{a(n+1)}\sum_{i=1}^{a(n+1)}\delta^{\ell,i}_{n+1},$$  and observe that
with $\beps_{n+1}$ as in \eqref{eq:eq143r},
\begin{equation}\label{eq:error-decomp-branch}
\gamma_{n+1} \beps_{n+1} =   \sum_{\ell =1}^4 \delta^{\ell}_{n+1}.
\end{equation}
We will now establish  several bounds on the error terms. 
 The following lemma provides a bound for the martingale noise term, namely the term corresponding to $\ell=1$. Recall the function $m(\cdot)$ defined in \eqref{eq:mtdef}.
\begin{lemma}\label{lem:delta1branch}
For each $T \in (0,\infty)$,
$$
\lim_{t\to\infty} \sup_{0\leq u \leq T}\left\|  \sum_{j=m(t)}^{m(t+u)-1} \delta^1_{j+1}\right\| = 0.
$$
\end{lemma}
\begin{proof}
Note that $\{\delta^1_j\}_{j=1}^{\infty}$ is adapted to $\{\bF_j\}_{j=1}^{\infty}$, 
and $\EE[\delta_{j+1}^1 | \bF_j] = 0$, where $\bF_j$ is as introduced above \eqref{eq:branchalg1}. Thus,  $\{\delta_j^1\}_{j=1}^{\infty}$ is a martingale difference sequence. Furthermore, there is some $\kappa_1 \in (0,\infty)$ such that 
$
\| \delta^1_{j+1}\| \leq \kappa_1 \gamma_{j+1}$ for each $j \in \NN_0$.
The result now follows by standard martingale estimates (see e.g. the proof of  Proposition 4.4 in \cite{ben}).
\end{proof}
The next result provides bounds for the remaining error terms.
\begin{lemma}\label{lem:deltasbranch}
For $\ell=2,3,4$ and $T > 0$, we have that
\begin{equation}\label{eq:eq1030r}
\lim_{t\to\infty} \sup_{0\leq u \leq T} \left\| \sum_{j = m(t)}^{m(t+u)-1} \delta^{\ell}_{j+1}\right\| = 0.
\end{equation}
\end{lemma}
\begin{proof}
The proofs for the cases when $\ell  = 2$ and $\ell = 4$ are similar to the proofs of Lemma \ref{lem:delta2} and Lemma \ref{lem:delta4}, respectively, and are omitted.  We now consider the case when $\ell = 3$. Recall the sequence $\{b(n), n \in \NN_0\}$ defined in (\ref{eq:kappadef}). Note that for an array $\{\alpha^i_j, i,j \in \NN\}$, we have, for $n\in \NN$,
\begin{multline}\label{eq:eq1032}
        \sum_{j=0}^{n-1} \sum_{i=1}^{a(j+1)} \alpha^i_{j+1} =  \sum_{j=0}^{n-1} \sum_{i=1}^{a(n)} \alpha^i_{j+1} \mathbbm{1}_{\{1 \leq i \leq a(j+1)\}}
         =  \sum_{i=1}^{a(n)} \sum_{j=0}^{n-1} \alpha^i_{j+1} \mathbbm{1}_{\{1 \leq i \leq a(j+1)\}}\\
         = \sum_{j=0}^{n-1} \alpha^1_{j+1}  + \sum_{i=2}^{a(n)} \sum_{j=0}^{n-1} \alpha^i_{j+1} \mathbbm{1}_{\{1 \leq i \leq a(j+1)\}}
         = \sum_{j=0}^{n-1} \alpha^1_{j+1}  + \sum_{i=2}^{a(n)} \sum_{j=0}^{n-1} \alpha^i_{j+1} \mathbbm{1}_{\{0\le b(i) \leq j+1 \}}\\
         = \sum_{j=0}^{n-1} \alpha^1_{j+1}  + \sum_{i=2}^{a(n)} \sum_{j=b(i)-1}^{n-1} \alpha^i_{j+1}.\\
    \end{multline}
Let, for $j \in \NN_0$, $\beta^i_j \doteq \gamma_j (K[\theta_j]Q[\theta_j])_{\bX^i_j, \cdot}$ so that,
$$
\delta^{3}_{j+1} = \frac{1}{a(j+1)} \sum_{i=1}^{a(j+1)}\delta^{3,i}_{j+1} = \sum_{i=1}^{a(j+1)} \frac{\beta^i_j - \beta^i_{j+1}}{a(j+1)}.
$$
Then, using \eqref{eq:eq1032},
\begin{equation*}
\begin{split}
 \sum_{j=0}^{n-1} \delta^3_{j+1} &=  \sum_{j=0}^{n-1} \frac{\beta^1_j - \beta^1_{j+1}}{a(j+1)} + \sum_{i=2}^{a(n)} \sum_{j=b(i)-1}^{n-1} \frac{\beta^i_j - \beta^i_{j+1}}{a(j+1)}\\
 &= \sum_{j=0}^{n-1} \frac{\beta^1_j - \beta^1_{j+1}}{a(j+1)}
+ \sum_{i=2}^{a(n)} \left[\frac{\beta^i_{b(i)-1}}{a(b(i))} - \frac{\beta^i_n}{a(n)} + \sum_{j=b(i)}^{n-1} \left(\frac{\beta^i_j}{a(j+1)} - \frac{\beta^i_j}{a(j)} \right)\right].
 \end{split}
\end{equation*}
 For $t, u > 0$, let
 \begin{equation}\label{eq:defnm1}
 n  \doteq n(t,u) \doteq m(t+u), \;\; m\doteq m(t),
 \end{equation}
 where $m(t)$ is given by (\ref{eq:mtdef}). Observe that
 
\begin{equation*}
     \begin{split}
   \sum_{j=m}^{n-1} \delta^3_{j+1}
   &= \sum_{j=0}^{n-1} \delta^3_{j+1} - \sum_{j=0}^{m-1} \delta^3_{j+1}\\
   &= \sum_{j=m}^{n-1} \frac{\beta^1_j - \beta^1_{j+1}}{a(j+1)} + \sum_{i=a(m)+1}^{a(n)} \frac{\beta^i_{b(i)-1}}{a(b(i))} 
     + \sum_{i=2}^{a(m)} \frac{\beta^i_m}{a(m)} - \sum_{i=2}^{a(n)} \frac{\beta^i_n}{a(n)} \\
   &\qquad+ \sum_{i=2}^{a(n)} \sum_{j=b(i)}^{n-1} \left(\frac{\beta^i_j}{a(j+1)} - \frac{\beta^i_j}{a(j)}\right)
      -  \sum_{i=2}^{a(m)} \sum_{j=b(i)}^{m-1} \left(\frac{\beta^i_j}{a(j+1)} - \frac{\beta^i_j}{a(j)}\right).
   \end{split}
\end{equation*}
  The last expression can be rewritten as
  \begin{equation*}
     \begin{split}
   &\sum_{j=m}^{n-1} \frac{\beta^1_j - \beta^1_{j+1}}{a(j+1)} + \sum_{i=a(m)+1}^{a(n)} \frac{\beta^i_{b(i)-1}}{a(b(i))} 
     + \sum_{i=2}^{a(m)} \frac{\beta^i_m}{a(m)} - \sum_{i=2}^{a(n)} \frac{\beta^i_n}{a(n)}  \\
   &\qquad+ \sum_{i=2}^{a(m)} \sum_{j=m}^{n-1} \left(\frac{\beta^i_j}{a(j+1)} - \frac{\beta^i_j}{a(j)}\right)
     + \sum_{i=a(m)+1}^{a(n)}\sum_{j=b(i)}^{n-1} \left(\frac{\beta^i_j}{a(j+1)} - \frac{\beta^i_j}{a(j)}\right).
  \end{split}
\end{equation*}  
 Define 
\begin{align*}
\eta_0 (n,m) \doteq \sum_{j=m}^{n-1} \frac{\beta^1_j - \beta^1_{j+1}}{a(j+1)},&
 \;\; \eta_1(n,m) \doteq  \sum_{i=a(m)+1}^{a(n)} \frac{\beta^i_{b(i)-1}}{a(b(i))}  + \sum_{i=2}^{a(m)} \frac{\beta^i_m}{a(m)}  - \sum_{i=2}^{a(n)} \frac{\beta^i_n}{a(n)},\\
 \eta_2(n,m) \doteq \sum_{i=2}^{a(m)}\sum_{j=m}^{n-1} \left(\frac{\beta^i_j}{a(j+1)} - \frac{\beta^i_j}{a(j)}\right),& \;\;
 \eta_3(n,m) \doteq \sum_{i=a(m)+1}^{a(n)} \sum_{j=b(i)}^{n-1} \left(\frac{\beta^i_j}{a(j+1)} - \frac{\beta^i_j}{a(j)}\right).
\end{align*}
 Then
 \begin{equation}\label{eq:delta3etadecomp}
 \sum_{j=m}^{n-1}\delta^3_{j+1} = \sum_{\ell = 0}^3 \eta_{\ell}(n,m).
 \end{equation}
 We begin by considering $\eta_3(n,m)$. Let
 $
 \kappa_1 \doteq \sup_{\substack{\theta \in \mathcal{P}(\Delta^o)\\ x \in \Delta^o}} \| K[\theta]Q[\theta]_{x,\cdot}\|,
 $
 so that $\|\beta^i_j\| \leq \kappa_1 \gamma_j$. Note that there is some $\kappa_2 \in (0,\infty)$ such that
 $
\frac{\gamma_j}{a(j)} - \frac{\gamma_j}{a(j+1)} \leq \kappa_2 \frac{\gamma_j}{a(j)^{2}}.
 $
 Using the last two estimates, the form of $\gamma_j$, and the fact that $a(j) \sim j^{\zeta}$, we can find some $\kappa_3 \in (0,\infty)$ such that 
 \begin{equation}\label{eq:eq226}
\left\| \frac{\beta^i_j}{a(j+1)} - \frac{\beta^i_j}{a(j)} \right\| \leq \frac{\kappa_3}{j^{1 + 2\zeta}}.
\end{equation}
Note that if $1 \leq i \leq a(n)$, then $b(i) \leq n$, so $b(i)^{-2 \zeta} \geq n^{- 2 \zeta}$. It follows that there is some $\kappa_4 \in (0,\infty)$ such that
 \begin{equation}\label{eq:eta3split}
         \| \eta_3(n,m) \| \leq \kappa_3 \sum_{i=a(m)+1}^{a(n)} \sum_{j=b(i)}^{n} \frac{1}{j^{1 + 2 \zeta}}
         \leq \kappa_4 \sum_{i=a(m)+1}^{a(n)}\left( \frac{1}{(b(i)-1)^{2 \zeta}} - \frac{1}{n^{2 \zeta}}\right)
        \leq \kappa_4 \frac{a(n)}{b(a(m))^{2\zeta}}.
 \end{equation}
 Note that there are some $c_1,c_2 \in (0,\infty)$ and some $t_1 \in (0,\infty)$ such that if $t \geq t_1$, then 
 $
 c_1 m(t)^{\zeta} \leq a(m(t)) \leq c_2 m(t)^{\zeta}.
 $
From the definition of $a(\cdot)$ and $b(\cdot)$, we see that
$
b(a(m)) \sim m.
$
Fix $\epsilon \in (0,1)$. Then, there is a $t_2 \in (t_1,\infty)$ such that if $t \geq t_2$, then $b(a(m(t))) \geq (1-\epsilon) m(t)$. It follows that if $t \geq t_2$, then
\begin{equation}\label{eq:banamdiff}
\begin{split}
\frac{a(n(t,u))}{b(a(m(t))^{2\zeta}} &\leq \frac{c_2 n(t,u)^{\zeta}}{(1 - \epsilon)^{2\zeta} m(t)^{2 \zeta}}.
\end{split}
\end{equation}
Recall that $
 \tau_k = \sum_{j=1}^{k} \gamma_j \sim \gamma_* \log(k).
 $
 From this and the definition of $m(\cdot)$ it follows that, with $\alpha = 1/\gamma_*$,
for some $t_3 \in (t_2,\infty)$ and $c_3, c_4 \in (0,\infty)$, and all $t \geq t_3$, 
$$
c_3 \exp(4\alpha t/5) \leq m(t) \leq c_4\exp(3\alpha t/2).
$$
Combining the previous display and (\ref{eq:banamdiff}), we see that if $t \geq t_3$, then
\begin{equation}\label{eq:eta3to02}
    \frac{a(n(t,u))}{b(a(m(t))^{2\zeta}} \leq \frac{c_2 n(t,u)^{\zeta}}{(1 - \epsilon)^{2 \zeta} m(t)^{2 \zeta}} 
    \leq \frac{c_2 c_4^{\zeta} e^{3\zeta \alpha(t + u)/2}}{c_3^{2 \zeta}(1 - \epsilon)^{2 \zeta} e^{8 \zeta \alpha t/5}}  
    \leq \frac{c_2  c_4^{\zeta} e^{2\zeta \alpha u}}{c_3^{2 \zeta} (1 - \epsilon)^{2 \zeta} e^{\zeta \alpha t/10}}.\\
  \end{equation}
Let
$
\kappa_6 \doteq \kappa_6(T)  \doteq \frac{c_2 c_4^{\zeta} e^{2\zeta \alpha T}}{c_3^{2 \zeta}(1 - \epsilon)^{2 \zeta}}.
$
Then combining (\ref{eq:eta3split}), (\ref{eq:banamdiff}), and (\ref{eq:eta3to02}), we see that if $t \geq t_3$, then
$
\sup_{0\leq u \leq T}\|\eta_3(n(t,u) ,m(t))\| \leq \frac{\kappa_6}{e^{\zeta \alpha t/10}},
$
and so, as $t \to \infty$,
\begin{equation}\label{eq:eta3to0final}
\sup_{0\leq u \leq T}\|\eta_3(n(t,u) ,m(t))\| \to 0.
\end{equation}

We now consider $\eta_2(n,m)$. From \eqref{eq:eq226}, there is some $\kappa_7 \in (0,\infty)$ such that
\begin{equation}\label{eq:eta2new1}
        \| \eta_2(n,m)\| \leq \kappa_3 \sum_{i=2}^{a(m)} \sum_{j=m}^{n-1} \frac{1}{j^{1 + 2\zeta}}
        \leq \kappa_7 \sum_{i=1}^{a(m)} \frac{1}{m^{2 \zeta}} \\
        = \kappa_7\, a(m) \frac{1}{m^{2 \zeta}},
\end{equation}
which shows that as $t \to \infty$,
\begin{equation}\label{eq:eta2to0final}
\sup_{0 \leq u \leq T}\| \eta_2(n(t, u),m(t))\| \to 0.
\end{equation}

We now consider $\eta_1(n,m)$. We can find some $\kappa_8 \in (0,\infty)$ such that
\begin{equation}\label{eq:eta1aa2}
        \| \eta_1(n,m)\| \leq \kappa_8 \left[\sum_{i=a(m)+1}^{a(n)} \frac{\gamma_{b(i)-1}}{a(b(i))} +  \sum_{i=1}^{a(n)} \frac{\gamma_n}{a(n)} + \sum_{i=1}^{a(m)} \frac{\gamma_m}{a(m)} \right]
        = \kappa_8 \left[ \gamma_m + \gamma_n + \sum_{i=a(m)+1}^{a(n)} \frac{\gamma_{b(i)-1}}{a(b(i))} \right].
\end{equation}
Note that $a(b(i)) = i$, so there is some $\kappa_9 \in (0,\infty)$ such that
$$
\sum_{i=a(m)+1}^{a(n)} \frac{\gamma_{b(i)-1}}{a(b(i))} \leq \kappa_9 \sum_{i=a(m)+1}^{a(n)} \frac{1}{i b(i)}.
$$
Additionally, $b(i) \sim i^{1/\zeta}$, so we can find some $t_4 \in (0, \infty)$ such that if $t \geq t_4$ and $i \geq a(m(t))$, then
$
\frac{1}{b(i)} \leq (1 + \epsilon)\frac{1}{i^{1/\zeta}}.
$
It follows that there is some $\kappa_{10} \in (0,\infty)$ such that if $t \geq t_4$, 
\begin{equation}\label{eq:eta1aa1}
\sum_{i=a(m)+1}^{a(n)} \frac{\gamma_{b(i)-1}}{a(b(i))} \leq \kappa_9 (1+ \epsilon) \sum_{i=a(m)}^{a(n)} \frac{1}{i^{1 + 1/\zeta}}
\leq \kappa_{10}  \frac{1}{a(m)^{1/\zeta}}.
\end{equation}
Combining (\ref{eq:eta1aa2}) and (\ref{eq:eta1aa1}), we see that
\begin{equation}\label{eq:eta1to0final}
\sup_{0\leq u \leq T} \| \eta_1(n(t,u), m(t))\| \to 0,
\end{equation}
as $t \to \infty$. Finally,  consider $\eta_0(n,m)$. We have, for some $\kappa_{11}\in(0,\infty)$, that
        $\| \eta_0(n,m)\| \leq \kappa_{11} \sum_{j=m}^{n} \frac{1}{j^{1 + \zeta}}$,
so it follows that as $t \to \infty$.
\begin{equation}\label{eq:eta0to0}
\sup_{0\leq u \leq T} \| \eta_0(n(t,u),m(t))\| \to 0.
\end{equation}

Combining (\ref{eq:delta3etadecomp}), (\ref{eq:eta3to0final}), (\ref{eq:eta2to0final}), (\ref{eq:eta1to0final}), and (\ref{eq:eta0to0}) we see that the convergence in \eqref{eq:eq1030r} holds with $\ell=3$.
The result follows.
\end{proof}

Define the continuous-time process $\{\bar{\epsilon}(t), t \geq 0\}$ by
$$
\bar{\epsilon}(\tau_n + t) \doteq \beps_{n+1}, \; \; t \in [0,\gamma_{k+1}), \; n \in \NN_0,
$$
and define
$$
\Delta(t,T) \doteq \sup_{0\leq u \leq T}\left\| \int_t^{t+u} \bar{\epsilon}(s)ds\right\|, \;\; t, T \geq 0.
$$

We now complete the proof of Theorem \ref{thm:branchlln1}.

\noindent\emph{Proof of Theorem \ref{thm:branchlln1}:}
Fix $T\in (0,\infty)$. Then, for some $\kappa_1 \in (0,\infty)$, and all $t>0$.
\begin{equation}\label{eq:deltatT1}
\begin{split}
\Delta(t,T) &\leq \sup_{0\leq u \leq T} \left\|\int_{\tau_{m(t)}}^{\tau_{m(t+u)}} \bar{\epsilon}(s)ds\right\| +  \sup_{0\leq u \leq T} \left\|\int_{\tau_{m(t)}}^{t} \bar{\epsilon}(s)ds\right\|  + \sup_{0\leq u \leq T} \left\|\int_{\tau_{m(t+u)}}^{t+u} \bar{\epsilon}(s)ds\right\|\\
&\leq  \sum_{k=1}^{4} \sup_{0\leq u\leq T} \left\| \sum_{j=m(t)}^{m(t+u)-1} \delta^k_{j+1}\right\| + 2 \kappa_1 \gamma_{m(t)}.
\end{split}
\end{equation}
From Lemma \ref{lem:delta1branch} and Lemma \ref{lem:deltasbranch} we now have that
$
\lim_{t\to\infty} \Delta(t,T) = 0.
$
From \cite{ben3}*{Proposition 4.1} it  follows that  $\{\hat{\btheta}(t)\}$ is an asymptotic pseudo-trajectory. 
%
%
%
%
%
%
The result now follows exactly as in the proof of \cite{benclo}*{Theorem 1.2}. \qed

\section{Central Limit Theorem for Algorithm II}
\label{sec:cltalg2}

In this section we provide the proof of Theorem \ref{thm:branchclt1}.
In Section \ref{sec:branchcoven}, we characterize the covariance structure of the error sequence $\{\be_n\}$. 
In Section \ref{sec:branchmun} we present some results for the linearized evolution sequence $\{\bmu_n\}$, and in Section \ref{sec:branchrhon} we characterize the asymptotic behavior of  the discrepancy sequence $\{\brho_n\}$. The proof of Theorem \ref{thm:branchclt1} is completed in Section \ref{sec:branchclt}.

We begin by studying the covariance structure of the error terms.

\subsection{Covariance structure of the error terms}\label{sec:branchcoven} 
Recall the collection of matrices $\{F_{\theta}(z) : \theta \in \mathcal{P}(\Delta^o), z \in \Delta^o\}$ defined by \eqref{eq:eq444}
%
and let $U_*$ be the $d \times d$ matrix introduced in \eqref{eq:ustar}.
The following result gives an expression for the conditional covariance matrix of $\{\be_{n+1}\}$ introduced in \eqref{eq:benrn}. The proof is similar to the proof of Proposition \ref{prop:clt2 A}.

\begin{proposition}\label{prop:covarlem}
 For each $n \in \NN_0$ and $x, y \in \Delta^o$,
			\begin{equation}\label{eq:branchcovid}
			\EE\left[ \be_{n+1} (x) \be_{n+1} (y) | \mathcal{F}_n \right] = \frac{1}{a(n+1)} \left( (U_*)_{x,y} + (D^{(1)}_{n})_{x,y} + (D^{(2)}_n)_{x,y}\right),
			\end{equation}
			where $D^{(1)}_{n}$ and $D^{(2)}_n$ are $d\times d$ random matrices satisfying
			the following:
			\begin{enumerate}[(i)]
			\item for some $C_1 \in (0,\infty)$ and all $n \in \NN$, $\| D^{(1)}_n \| \leq C_1 \| \btheta_n - \theta_*\|$.
			\item for some $C_2,\beta \in (0,\infty)$ and for all $n \in \NN$,
			$$
			\gamma_{n+1} \EE \left\| \sum_{m=1}^{n} D^{(2)}_{m-1}\right\| \leq C n^{-\beta}.
			$$
			\end{enumerate}
\end{proposition}

\begin{proof}
By a similar argument as in the proof of Proposition \ref{prop:clt2 A}, we have that
\begin{equation}\label{eq:neweproduct4}
\begin{split}
&\EE[\be_{n+1}(x)\be_{n+1}(y) |\mathcal{F}_n] = a(n+1)^{-2} \sum_{i=1}^{a(n+1)} F_{\btheta_n}(\bX_n^i)_{x,y}.
\end{split}
\end{equation}
We can write
$$
F_{\btheta_n}(\bX_n^{i})_{x,y} = (U_*)_{x,y} + (D^{(1),i}_{n})_{x,y} + (D^{(2),i}_n)_{x,y},
$$
where
$$
D_n^{(1),i} \doteq \sum_{w\in \Delta^o} (F_{\btheta_n}(w)\pi(\btheta_n)_w - F_{\theta_*}(w)(\theta_*)_w)
\mbox{ and }
D_n^{(2),i} \doteq F_{\btheta_n}(\bX_n^{i}) - \sum_{w\in \Delta^o}F_{\btheta_n}(w)\pi(\btheta_n)_w.
$$
The identity in (\ref{eq:branchcovid}) is obtained by defining
$$
D_n^{(1)} \doteq \frac{1}{a(n+1)} \sum_{i=1}^{a(n+1)} D^{(1),i}_n = \sum_{w\in\Delta^o} \left(F_{\btheta_n}(w)\pi(\btheta_n)_w - F_{\theta_*}(w)(\theta_*)_w \right),
$$
 and 
 $$
D_n^{(2)} \doteq  \frac{1}{a(n+1)} \sum_{i=1}^{a(n+1)} D^{(2),i}_n,
$$
and using the identity in (\ref{eq:neweproduct4}).  The proof of   (i) is similar to the proof of part  (i) of Proposition \ref{prop:clt2 A} and is omitted. We now show that   (ii) holds as well.

\emph{Proof of  (ii):} 
 As in the proof of part (ii) of Proposition \ref{prop:clt2 A}, it suffices to show that there is some $C_2, \beta > 0$ such that 
 for each $(u,v) \in \Delta^o \times \Delta^o$ and all $n \in \NN$,
$$
\gamma_{n+1}  \EE \left\| \sum_{m=1}^{n} (D_m^{(2)})_{u,v} \right\| \leq C_2 (n+1)^{-\beta}.
$$ 
Fix $(u,v) \in \Delta^o \times \Delta^o$ and, once more abusing notation, denote $(D_m^{(2)})_{u,v}$ as $D_m^{(2)}$. Using the Poisson equation \eqref{eq:pois} we have, as in the proof of Proposition \ref{prop:clt2 A}, with
$$
D^{(2,a),i}_n \doteq U_{\btheta_n}(\bX_{n+1}^{i}) - (K[\btheta_n]U_{\btheta_n})(\bX_{n}^{i}),\;\;  D^{(2,b),i}_n \doteq U_{\btheta_n}(\bX_n^{i}) - U_{\btheta_n}(\bX_{n+1}^{i}),
$$
that $D_n^{(2),i} = D_n^{(2,a),i} + D_{n}^{(2,b),i}$. Now, let
$$
D^{(2,a)}_n \doteq \frac{1}{a(n+1)}\sum_{i=1}^{a(n+1)}D_n^{(2,a),i},\;\;  D^{(2,b)}_n \doteq \frac{1}{a(n+1)}\sum_{i=1}^{a(n+1)}D_n^{(2,b),i},
$$
so that $D^{(2)}_n = D^{(2,a)}_n + D^{(2,b)}_n$. Note that with $\mathcal{G}_n \doteq \mathcal{F}_{n+1}$, $\{D^{(2,a)}_n\}_{n=1}^{\infty}$ is a $\{\mathcal{G}_n\}_{n=1}^{\infty}$-martingale increment sequence. Consequently, we can apply Burkholder's inequality and use a conditioning argument to show that, for some $\kappa_1 \in (0,\infty)$,
and all $n\in \NN$, 
\begin{equation}\label{eq:branchd2a}
\gamma_{n+1} \EE\left| \sum_{m=1}^{n} D^{(2,a)}_{m-1}\right | \leq \kappa_1 \gamma_{n+1} \left( \sum_{m=1}^{n} \frac{1}{a(m)}\right)^{1/2}.
\end{equation}

We now consider $\{D_n^{(2,b)}\}$. Observe that 
\begin{equation}\label{eq:newd2split1}
\begin{split}
\left| \sum_{m=1}^n D^{(2,b)}_{m-1} \right| 
&\leq \left| \frac{1}{a(1)} \sum_{i=1}^{a(1)}U_{\theta_1}(\bX^i_1) - \frac{1}{a(n)} \sum_{i=1}^{a(n)} U_{\btheta_n}(\bX^i_{n+1})\right|\\
&\qquad + \left| \sum_{m=2}^n \left[ \frac{1}{a(m)} \sum_{i=1}^{a(m)}U_{\theta_m}(\bX^i_m) - \frac{1}{a(m-1)} \sum_{i=1}^{a(m-1)} U_{\theta_{m-1}}(\bX^i_m)\right]\right|,
\end{split}
\end{equation}
and
\begin{equation}\label{eq:newd2split2}
\begin{split}
 &\left| \sum_{m=2}^n \left[ \frac{1}{a(m)} \sum_{i=1}^{a(m)}U_{\theta_m}(\bX^i_m) - \frac{1}{a(m-1)} \sum_{i=1}^{a(m-1)} U_{\theta_{m-1}}(\bX^i_m)\right]\right|\\
 &\leq \left| \sum_{m=2}^{n} \frac{1}{a(m)}\sum_{i=a(m-1)+1}^{a(m)} U_{\theta_m}(\bX^i_m)\right| + \left| \sum_{m=2}^{n} \sum_{i=1}^{a(m-1)} \left[ \frac{1}{a(m)} U_{\theta_m}(\bX^i_m) - \frac{1}{a(m-1)} U_{\theta_{m-1}}(\bX^i_m)\right]\right|.
\end{split}
\end{equation}
Letting
$
\kappa_2  \doteq \sup_{\theta \in \mathcal{P}(\Delta^o), z \in \Delta^o}|U_{\theta}(z)| < \infty,
$
and noting that $0\leq a(m) - a(m-1) \leq 1$ for all $m \in \NN$, we  see that
\begin{equation}\label{eq:newd2split3}
\begin{split}
 \left| \sum_{m=2}^{n} \frac{1}{a(m)}\sum_{i=a(m-1)+1}^{a(m)} U_{\theta_m}(\bX^i_m)\right| 
 \leq  \kappa_{2} \sum_{m=2}^{n} \frac{1}{a(m)}.
\end{split}
\end{equation}
Since the maps $\theta \mapsto K(\theta)$ and $\theta \mapsto Q(\theta)$ are bounded Lipschitz maps, there is a 
$\kappa_3 \in (0, \infty)$ such that for all $x \in \Delta^o$ and $\theta,\theta' \in \mathcal{P}(\Delta^o)$, 
$
|U_{\theta}(x) - U_{\theta'}(x)| \leq \kappa_3 \|\theta - \theta'\|,
$
so there is some $\kappa_{4} \in (0,\infty)$ such that
\begin{multline}\label{eq:newd2split4}
\left| \sum_{m=2}^{n} \sum_{i=1}^{a(m-1)} \left[ \frac{1}{a(m)} U_{\theta_m}(\bX^i_m) - \frac{1}{a(m-1)} U_{\theta_{m-1}}(\bX^i_m)\right]\right|\\
 \leq \sum_{m=2}^{n} \sum_{i=1}^{a(m-1)} \Bigg[\left| \frac{1}{a(m)}U_{\theta_m}(\bX^i_m) - \frac{1}{a(m-1)}U_{\theta_m}(\bX^i_m)\right|
+ \left| \frac{1}{a(m-1)}\left(U_{\theta_m}(\bX^i_m) -  U_{\theta_{m-1}}(\bX^i_m)\right)\right|\Bigg]\\
 \leq \sum_{m=2}^{n}\sum_{i=1}^{a(m-1)} \left[ \kappa_{2} \frac{1}{a(m-1)a(m)} + 2 \kappa_3  \gamma_m \frac{1}{a(m-1)}\right]
 \leq \kappa_{4}  \sum_{m=1}^n \left[\frac{1}{a(m)} + \gamma_{m}\right].
\end{multline}

Additionally,
\begin{equation}\label{eq:newd2split6}
 \left| \frac{1}{a(1)} \sum_{i=1}^{a(1)}U_{\theta_1}(\bX^i_1) - \frac{1}{a(n)} \sum_{i=1}^{a(n)} U_{\btheta_n}(\bX^i_{n+1})\right| \leq 2 \kappa_2.
\end{equation}
Combining  (\ref{eq:newd2split1}),(\ref{eq:newd2split2}),  (\ref{eq:newd2split3}),   (\ref{eq:newd2split4}), and (\ref{eq:newd2split6}) we see that there is some $\kappa_5 \in (0,\infty)$ such that
\begin{equation}\label{eq:newd2bfinal}
\gamma_{n+1} \left\| \sum_{m=1}^{n} D^{(2,b)}_{m-1}\right\| \leq \kappa_{5}\gamma_{n+1} \left(1 + \sum_{m=1}^{n} \left[ \frac{1}{a(m)} + \gamma_m \right]\right).
\end{equation}
The result follows on combining  (\ref{eq:branchd2a}) and (\ref{eq:newd2bfinal}).
\end{proof}

The next result provides a useful bound for the moments of the error sequence $\{\be_n\}$. The proof is similar to the proof of Proposition \ref{prop:clt2 B and C} and is omitted.

\begin{proposition}\label{prop:newemoments}
There is some $C > 0$ such that for all $n \in \NN$,
$
\EE \| \be_n\|^4 \leq C/a(n)^{2}.
$
\end{proposition}

\subsection{The linearized evolution sequence}\label{sec:branchmun}

The goal of this section is to study the linearized evolution sequence given in (\ref{eq:mudef1}). The following lemma says that $\bmu_n$ given by the linearized evolution in \eqref{eq:mudef1} converges a.s. to 0. The proof is similar to the proof of \cite{for}*{Proposition 5.1} and Proposition \ref{prop:muprop}, and is therefore omitted. 

\begin{lemma}\label{lem:newmuto0}
As $n \to \infty$ we have $\bmu_n \to 0$ a.s.
\end{lemma}

The next result is used in the proof of Proposition \ref{prop:rhordiff}. It provides a useful bound on the moments of the linearized evolution sequence. Recall the quantity $\sigma_{n} = \sqrt{a(n)/\gamma_n}$ defined in  
\eqref{eq:signdef}.

\begin{proposition}\label{prop:newmu1}
Suppose that $\gamma_*> L^{-1}$. Then there is some $C \in (0,\infty)$ such that for all $n \in \NN_0$, $\EE\| \bmu_{n+1}\|^2 \leq C/\sigma_{n+1}^2$.
\end{proposition}

\begin{proof}
Recall the collection of matrices $\{\psi_*(n,k), n \in \NN, k \leq n+1\}$ defined in (\ref{eq:psistarmatrixdef}). A simple recursive argument shows that
\begin{equation}\label{eq:branchmurecursive}
\bmu_{n+1} = \sum_{k=1}^{n+1}  \gamma_{k} \psi_*(n+1,k+1) \be_k.
\end{equation}
Proposition \ref{prop:newemoments} ensures that there is some $\kappa_1  \in(0,\infty)$ such that $\EE \|\be_k\|^2 \leq \kappa_1/a(k)$ for all $k \in \NN$. Fix $L' \in (0,L)$ such that $L'\gamma_* > 1$, and use (\ref{eq:eq104}) and (\ref{eq:branchmurecursive}) to find some $\kappa_2(L') \in (0,\infty)$ such that
$$
\EE \| \bmu_{n+1}\|^2 \leq  \kappa_2(L') \sum_{k=1}^{n+1} \gamma_k^2  \exp\left( -2L' \sum_{j=k+1}^{n+1} \gamma_j\right) \frac{1}{a(k)}.
$$
Note that there is some $\kappa_3 \in (0, \infty)$ such that
\begin{equation}\label{eq:bnupperbound}
\sigma_n^{2} \leq \kappa_3 n^{1 + \zeta},\;\; n \in \NN,
\end{equation}
so there is some $\kappa_4 \in (0,\infty)$ such that
$$
\sigma_{n+1}^{2} \EE \| \bmu_{n+1}\|^2 \leq  \kappa_4 n^{1 + \zeta} \sum_{k=1}^{n+1} \gamma_k^{2 + \zeta} \exp\left( - 2L' \sum_{j=k+1}^{n+1}\gamma_j\right).
$$
The right side is bounded since $\gamma_*L'>1$. The result follows.
\end{proof}

\subsection{Analysis of the discrepancy sequence}\label{sec:branchrhon}

The goal of this section is to show that the discrepancy sequence $\{\brho_n\}$ converges to $0$ in probability under the central limit scaling. As in Section \ref{sec:rhocalc1}, for each $n \in\NN$, we let $R^{(n)}_{\bullet}$ denote the tensor $R^{(n)}_i \doteq \mathbf{R}^{(n)}_i(\btheta_n)$, $1\le i\le d$,
where $\mathbf{R}^{(n)}_i(\theta)$, for $\theta \in \clp(\Delta^o)$, is defined as in \eqref{eq:eq523}.
Note that this tensor satisfies
\begin{equation}\label{eq:branchR}
h(\btheta_n) = \nabla h(\theta_*)(\btheta_n - \theta_*)  + (\btheta_n - \theta_*)^T R^{(n)}_{\bullet}(\btheta_n - \theta_*).
\end{equation}
For each $1 \leq k \leq n$, define the matrix $\psi(n,k)$ by
$$
\psi(n,k) \doteq \prod_{j=k}^{n} \left( I + \gamma_j \left( \nabla h(\theta_*) + 2\bmu^T_{j-1} R_{\bullet}^{(j-1)} + \brho^{T}_{j-1}R^{(j-1)}_{\bullet}\right)\right),
$$
and let $\psi(n,n+1) \doteq  I$. The next proposition provides a useful bound on $\psi(n,k)$. 

\begin{proposition}\label{prop:branchpsi}
For each $L' \in (0,L)$, and a.e. $\om$, there is a $C = C(L',\om) \in (0,\infty)$ such that if $1 \leq k \leq n$, then
\begin{equation}\label{eq:branchpsboundi}
\| \psi(n,k)\| \leq C \exp\left( - L' \sum_{j=k}^{n} \gamma_j\right).
\end{equation}
\end{proposition}
\begin{proof}
	Let $A \doteq \nabla h(\theta_*)$ and
\begin{equation*}
	A_n  \doteq \nabla h(\theta_*) + 2 \bmu^T_{j-1} R^{(j-1)}_{\bullet} + \brho^T_{j-1} R^{(j-1)}_{\bullet}
	= \nabla h(\theta_*) + \bmu^T_{j-1} R^{(j-1)}_{\bullet} + (\theta_{j-1} - \theta_*)^T R^{(j-1)}_{\bullet}.
\end{equation*}
  From Theorem \ref{thm:branchlln1} and Lemma \ref{lem:newmuto0}, $\|A_n-A\|\to 0$ a.s. as $n \to \infty$. The result now follows from \cite{for}*{Lemma 5.8}.
\end{proof}

The next result will be used to show that $\{\brho_n\}$ tends to 0 in probability under the central limit scaling. 
\begin{proposition}\label{prop:rhordiff}
Suppose $\gamma_*>L^{-1}$.
For some $\kappa \in (0,1)$, we have, as $n \to \infty$,
$$
\sigma_n^{1 + \kappa} \left( \brho_n - \sum_{k=1}^n \gamma_k \psi(n,k+1)\br_k\right) \stackrel{\PP}{\to} 0.
$$
\end{proposition}
\begin{proof}
Fix $L' \in ( \gamma_*^{-1} ,L)$ and $\kappa \in \left(0, \frac{1}{2} \wedge  \left[2L'\gamma_*(1+\zeta)^{-1} - 1\right] \right)$. Using (\ref{eq:branchTheta1}), (\ref{eq:mudef1}), (\ref{eq:rhobranch1}), (\ref{eq:branchR}), and a recursive argument similar to the one used for obtaining \eqref{eq:rhon12}, we have
\begin{equation}\label{eq:rhosum}
\begin{split}
\brho_n - \sum_{k=1}^{n} \gamma_k \psi(n,k+1)\br_k &= \psi(n,1)\brho_0 + \sum_{k=1}^n \gamma_k \psi(n,k+1) \left( \bmu^T_{k-1} R^{(k-1)}_{\bullet}\bmu_{k-1}\right).
\end{split}
\end{equation}
We begin by showing that as $n \to \infty$,
\begin{equation}\label{eq:rho branch p1}
\sigma_n^{1+ \kappa} \psi(n,1)\brho_0 \stackrel{\PP}{\to} 0.
\end{equation}
Since $\|\brho_0\|$ is bounded, it suffices to show that $\sigma_n^{1 + \kappa} \psi(n,1)$ tends to $0$ in probability. Using Proposition \ref{prop:branchpsi}, for a.e. $\om$, there is some $\kappa_1(\omega) \in (0,\infty)$ such that
\begin{equation}\label{eq:branchpsin1}
\begin{split}
\| \psi(n,1)\| &\leq \kappa_1  \exp\left( -L' \sum_{j=1}^{n} \gamma_j\right)
\end{split}
\end{equation}
Additionally, we can find some $\kappa_2 \in(0,\infty)$ such that 
\begin{equation}\label{eq:boundexpbn}
 \exp\left( -L' \sum_{j=1}^{n} \gamma_j\right) \leq \kappa_2 n^{-L'\gamma_*},
 \qquad
 \sigma_n^{1 + \kappa} \leq \kappa_2 n^{\frac{1}{2}(1 + \kappa)(1 + \zeta)}.
\end{equation}
Combining (\ref{eq:branchpsin1}) and (\ref{eq:boundexpbn}), we obtain, for a.e. $\om$, some $\kappa_3(\omega) \in (0,\infty)$ such that
\begin{equation}\label{eq:branchrhobn1a}
\begin{split}
\sigma_n^{1 + \kappa} \| \psi(n,1)\| &\leq \kappa_3 n^{\frac{1}{2}(1 + \kappa)(1+\zeta)} n^{-L'\gamma_*}.
\end{split}
\end{equation}
Since 
$\kappa \in \left(0, \frac{1}{2} \wedge  \left[2L'\gamma_*(1+\zeta)^{-1} - 1\right] \right)$, we have that $\frac{1}{2}(1 + \kappa)(1 +\zeta) - L'\gamma_* < 0,$
which ensures that the expression in (\ref{eq:branchrhobn1a}) converges to $0$. Therefore, (\ref{eq:rho branch p1}) holds. We now show that as $n \to \infty$,
\begin{equation}\label{eq:rho branch p2}
 \sigma_n^{1+ \kappa}\sum_{k=1}^n \gamma_k \psi(n,k+1) \left( \bmu^T_{k-1} R^{(k-1)}_{\bullet}\bmu_{k-1}\right) \stackrel{\PP}{\to} 0.
\end{equation}
 Note that $ R \doteq \sup_k \|R^{(k-1)}_{\bullet}\| < \infty$ a.s., so using Proposition \ref{prop:branchpsi}, for a.e. $\om$, we can find some $\kappa_4(\omega)\in(0,\infty)$ such that 
\begin{equation}\label{eq:branchrhopart1}
\begin{split}
\sigma_n^{1+ \kappa}\left\| \sum_{k=1}^n \gamma_k \psi(n,k+1) \left(\bmu^T_{k-1} R^{(k-1)}_{\bullet} \bmu_{k-1}\right)\right\|
&\leq \kappa_4  \sigma_n^{1+\kappa}\sum_{k=1}^{n} \gamma_k \exp\left( -L'\sum_{j=k}^n \gamma_j\right)\|\bmu_{k-1}\|^2
\end{split}
\end{equation}
Using Proposition \ref{prop:newmu1}, we can find some $\kappa_5 \in(0,\infty)$ such that for all $k \in \NN$, $\EE\| \bmu_{k-1}\|^2 \leq \kappa_5 \sigma_{k-1}^2$, so
\begin{equation}\label{eq:branchrhopart2}
\begin{split}
  \sigma_n^{1+\kappa}\EE\left[\sum_{k=1}^{n}  \gamma_k \exp\left( -L'\sum_{j=k}^n \gamma_j\right)\|\bmu_{k-1}\|^2 \right] &\leq \kappa_5  \sigma_n^{1 + \kappa} \sum_{k=1}^{n} \gamma_k \exp\left( - L'\sum_{j=k}^n \gamma_j\right) \sigma_{k-1}^2.\\
 \end{split}
 \end{equation}
 From (\ref{eq:boundexpbn}), we can bound the last term in (\ref{eq:branchrhopart2})  above by 
 \begin{equation}\label{eq:branchkappa6}
\kappa_6  n^{\frac{1}{2} (1+ \kappa)(1 + \zeta)}\sum_{k=1}^{n} \gamma_k \exp\left( - L'\sum_{j=k}^{n}\gamma_j\right) (k-1)^{-(1 + \zeta)}
 \end{equation}
 for some $\kappa_6 \in (0,\infty)$. Recalling that $\kappa <1/2$ we see that the expression in (\ref{eq:branchkappa6}) converges to $0$ as $n \to \infty$. Combining this observation with (\ref{eq:branchrhopart1}) and (\ref{eq:branchrhopart2}) we obtain (\ref{eq:rho branch p2}). The result now follows on combining (\ref{eq:rho branch p1}) and (\ref{eq:rho branch p2}).
\end{proof}

The next result will be used to prove Corollary \ref{cor:rhoto0}.
\begin{proposition}\label{prop:r1r2bounds}
Suppose that $\gamma_*> L^{-1}$. Then,
as $n \to \infty$, 
\begin{equation}\label{eq:rhoasy2}
\sigma_n \left[ \sum_{k=1}^{n} \gamma_k \psi(n,k+1) \br_k \right] \stackrel{\PP}{\to} 0.
\end{equation}
\end{proposition}
\begin{proof}
Fix $L' \in (\gamma_*^{-1}, L)$, and define
$$
\br^{(a)}_{n+1} \doteq \frac{1}{a(n+1)} \sum_{i=1}^{a(n+1)} \left( \left(K[\btheta_{n+1}]Q[\btheta_{n+1}]\right)_{\bX^i_{n+1}, \cdot} - \left(K[\btheta_{n}]Q[\btheta_{n}]\right)_{\bX^i_{n+1},\cdot}\right),
$$
and
$$
\br^{(b)}_{n+1} \doteq \frac{1}{a(n+1)} \sum_{i=1}^{a(n+1)} \left( \left(K[\btheta_n]Q[\btheta_n]\right)_{\bX^i_n, \cdot} - \left(K[\btheta_{n+1}]Q[\btheta_{n+1}]\right)_{\bX^i_{n+1},\cdot}\right),
$$
so that $\br_{n+1} = \br^{(a)}_{n+1} + \br^{(b)}_{n+1}$. Using Proposition \ref{prop:branchpsi}, we can find, for a.e. $\om$, some $\kappa_1(\omega)\in(0,\infty)$ such that
\begin{equation}\label{eq:branchr1e1}
\left\| \sum_{k=1}^{n} \gamma_k \psi(n,k+1)\br_k^{(a)}\right\| \leq  \kappa_1 \sum_{k=1}^{n} \gamma_k^2 \exp\left(L' \sum_{j=k}^{n} \gamma_j\right) \Big\| \frac{1}{\gamma_k} \br_k^{(a)} \Big\|.
\end{equation}
Using the fact that $\theta \to K[\theta]$, $\theta \to Q[\theta]$ are bounded Lipschitz maps, we can find some $\kappa_2 \in (0,\infty)$ such that for $k \in \NN$,
$
\EE \Big\| \frac{1}{\gamma_k} \br^{(a)}_{k} \Big\| \leq \kappa_2.
$
From this and  (\ref{eq:bnupperbound}), we can find some $\kappa_3 \in (0,\infty)$ such that
\begin{equation}\label{eq:branchr1e3}
\begin{split}
\sigma_n \sum_{k=1}^{n} \gamma_k^{2} \exp\left( - L'\sum_{j=k}^{n}\gamma_j\right) \EE \Big\| \frac{1}{\gamma_k} \br_k^{(a)} \Big\| 
&\leq \kappa_3 n^{\frac{1}{2}(1 + \zeta) - L'\gamma_*} \sum_{k=1}^{n} k^{L'\gamma_* - 2}.
\end{split}
\end{equation}
Since $\zeta < 1$, the final term in (\ref{eq:branchr1e3}) tends to $0$ as $n \to \infty$. Combining this observation with (\ref{eq:branchr1e1}) and (\ref{eq:branchr1e3}), we have that, as $n \to \infty$,
\begin{equation}\label{eq:branchr1e4}
\sigma_n \left\| \sum_{k=1}^{n} \gamma_k \psi(n,k+1)\br^{(a)}_k \right\|  \stackrel{\PP}{\to} 0.
\end{equation}
Next, for $n \in \NN$ define
$$
\Xi_n \doteq \sum_{k=1}^n \br^{(b)}_{k}, 
\qquad
H_n \doteq \nabla h(\theta_*) + 2\bmu_n^T R^{(n)}_{\bullet} + \brho^T_n R^{(n)}_{\bullet},
$$
and apply the summation by parts formula as in \eqref{eq:rhor21a}-\eqref{eq:rhor21} to obtain
\begin{equation}\label{eq:boundxia}
\begin{split}
\sum_{k=1}^n \gamma_k \psi(n,k+1) \br^{(b)}_k &= \gamma_n \Xi_n -  \sum_{k=1}^{n-1} \Xi_k  ( {\gamma}_{k+1} \psi(n,k+2) - {\gamma}_k \psi(n,k+1))\\
&=  \gamma_n \Xi_n + \sum_{k=1}^{n-1} \gamma_k \gamma_{k+1} \Xi_k \psi(n,k+2)\left( \gamma_*^{-1} I + H_k\right).\\
\end{split}
\end{equation}
Recall the sequence  $\{b(n), n \in \NN_0\}$ defined as $b(n)\doteq \lfloor n^{1/\zeta}\rfloor$. Let $\beta^i_k \doteq \left(K[\theta_{k}]Q[\theta_{k}]\right)_{\bX^i_k,\cdot}$. Using the fact that  the map $\theta \to \| K[\theta]Q[\theta]\|$ is bounded, we can find some $\kappa_5 \in (0,\infty)$ such that
\begin{multline}\label{eq:xinbranch1}
\|\Xi_n\| = \left\|\sum_{k=1}^{n} \frac{1}{a(k)} \sum_{i=1}^{a(k)}\left( \beta^i_{k-1} - \beta^i_k \right)\right\|
= \left\| \sum_{k=1}^{n} \sum_{i=2}^{a(n)} \frac{\beta^i_{k-1} - \beta^i_{k}}{a(k)} \mathbbm{1}_{\{2 \leq i \leq a(k)\}}\right\| + \left\| \sum_{k=1}^{n} \frac{\beta^1_{k-1} - \beta^1_k}{a(k)} \right\|\\
\leq \left\| \sum_{k=1}^{n} \sum_{i=2}^{a(n)} \frac{\beta^i_{k-1} - \beta^i_{k}}{a(k)} \mathbbm{1}_{\{2 \leq i \leq a(k)\}}\right\|  + \left\| \frac{\beta^1_0}{a(1)}  - \frac{\beta^1_n}{a(n)} \right\| + \left\| \sum_{k=1}^{n-1} \left(\frac{\beta^1_k}{a(k+1)} - \frac{\beta^1_k}{a(k)}\right) \right\|\\
\leq \left\| \sum_{k=1}^{n} \sum_{i=2}^{a(n)} \frac{\beta^i_{k-1} - \beta^i_{k}}{a(k)} \mathbbm{1}_{\{2 \leq i \leq a(k)\}}\right\|  + \kappa_5.
\end{multline}
Furthermore, there is a $\kappa_6 \in (0,\infty)$ such that
\begin{multline}\label{eq:xinbranch2}
\left\| \sum_{i=2}^{a(n)}  \sum_{k=1}^{n} \frac{\beta^i_{k-1} - \beta^i_{k}}{a(k)} \mathbbm{1}_{\{2 \leq i \leq a(k)\}}\right\| 
= \left\| \sum_{i=2}^{a(n)}  \sum_{k=b(i)}^{n} \frac{\beta^i_{k-1} - \beta^i_{k}}{a(k)} \right\| \\
= \left\| \sum_{i=2}^{a(n)} \left[ \frac{\beta^i_{b(i)-1}}{a(b(i))} + \sum_{k=b(i)}^{n-1}\left(\frac{\beta^i_k}{a(k+1)} - \frac{\beta^i_k}{a(k)}\right) - \frac{\beta^i_n}{a(n)}\right]\right\|\\
\leq  \kappa_6 \sum_{i=2}^{a(n)}\left[ \frac{1}{a(b(i))} + \sum_{k=b(i)}^{n-1} \left( \frac{1}{a(k)} - \frac{1}{a(k+1)}\right) + \frac{1}{a(n)}\right] \\
= \kappa_6 \sum_{i=2}^{a(n)}\left[ \frac{1}{a(b(i))} + \left(\frac{1}{a(b(i))} - \frac{1}{a(n)}\right) + \frac{1}{a(n)}\right]
= 2 \kappa_6 \sum_{i=1}^{a(n)} \frac{1}{i}.
\end{multline}

Combining (\ref{eq:xinbranch1}) and (\ref{eq:xinbranch2}) we see that there is some $\kappa_7 \in (0,\infty)$ such that for each $n \in \NN$,
\begin{equation}\label{eq:branchpsir2}
\|\Xi_n\| \leq \kappa_7\log n.
\end{equation}
 Using  (\ref{eq:bnupperbound}) and (\ref{eq:branchpsir2}), we can find some $\kappa_8 \in (0,\infty)$ such that
\begin{equation}\label{eq:boundxi1}
\sigma_n \gamma_n \|\Xi_n\|  \leq   \kappa_8 (\log n) n^{- \frac{1}{2}(1 - \zeta)},
\end{equation}
which tends to $0$ as $n \to \infty$, since $\zeta < 1$. Note that 
$
A \doteq \sup_{n \in \NN} \left\| \gamma_*^{-1} I + H_n\right\| < \infty \text{ a.s.},
$
which, along with (\ref{eq:branchpsir2}) and Proposition \ref{prop:branchpsi}, ensures that for a.e. $\om$, there is some $\kappa_9(\om)\in(0,\infty)$ such that
\begin{equation}\label{eq:branchr2e4}
\begin{split}
\sigma_n \left\| \sum_{k=1}^{n-1} \gamma_k \gamma_{k+1} \Xi_k \psi(n,k+2)\left( \gamma_*^{-1} I + H_k\right)\right\| \leq \kappa_9 \sigma_n \sum_{k=1}^{n-1} \gamma_k^2 \exp\left( - L' \sum_{j=k}^{n} \gamma_j\right) \log k.
\end{split}
\end{equation}
The last term in (\ref{eq:branchr2e4})  can be bounded above by
$
\kappa_{10} (\log n)n^{-\frac{1}{2}(1 - \zeta)},
$
for some $\kappa_{10} \in(0,\infty)$ and hence the expression in \eqref{eq:branchr2e4} tends to $0$ as $n \to \infty$. 
 Combining this with (\ref{eq:boundxia}) and (\ref{eq:boundxi1}),  we see that as $n \to \infty$,
\begin{equation}\label{eq:branchr2prob}
\sigma_n \left\| \sum_{k=1}^{n} \gamma_k \psi(n,k+1)\br^{(b)}_k \right\|  \stackrel{\PP}{\to} 0.
\end{equation}
The result follows on combining (\ref{eq:branchr1e4}) and (\ref{eq:branchr2prob}).
\end{proof}

The following corollary says that the discrepancy sequence $\{\brho_n\}$ tends to 0 in probability under the central limit scaling.

\begin{corollary}\label{cor:rhoto0}
Suppose that $\gamma_*> L^{-1}$.
Then, as $n \to \infty$, $\sigma_n \brho_n \stackrel{\PP}{\to} 0$. 
\end{corollary}

\begin{proof}
The result is immediate from Proposition  \ref{prop:rhordiff} and Proposition \ref{prop:r1r2bounds}.
\end{proof}

\subsection{Proof of Theorem \ref{thm:branchclt1}}\label{sec:branchclt}

Consider the array $\{Z_{n,k} , n \in \NN, 1 \leq k \leq n \}$ given by
\begin{equation}\label{eq:newxarray}
Z_{n,k} \doteq \sigma_{n} \gamma_k \psi_*(n,k+1)\be_k, \;\; 1 \leq k \leq n, \; n \in \NN.
\end{equation}
Note that
\begin{equation}\label{eq:newmux}
\sigma_{n} \bmu_{n} = \sum_{k=1}^{n} Z_{n, k}.
\end{equation}
We will apply \cite{halhey}*{Corollary 3.1} to complete the proof of Theorem  \ref{thm:branchclt1}. The conditions for this result are verified in Lemma \ref{lem:newlindeberg} and Lemma \ref{lem:newlimitvar} given below.

\begin{lemma}\label{lem:newlindeberg}
Suppose that $\gamma_*> L^{-1}$.
Then, as $n \to \infty$,
$
\sum_{k=1}^{n+1} \EE \| Z_{n+1,k}\|^4 \to 0.
$
\end{lemma}

\begin{proof}
Fix $L' \in (\gamma_*^{-1}, L)$.  Using  (\ref{eq:eq104}), Proposition \ref{prop:newemoments}, and  (\ref{eq:bnupperbound}), we can find some $\kappa_1 \in(0,\infty)$ such that
\begin{multline}
\sum_{k=1}^{n} \EE \| Z_{n,k}\|^4 = \sum_{k=1}^{n} \EE \| \sigma_{n} \gamma_k \psi_*(n,k+1)\be_k\|^4
\leq \sigma_{n}^{4} \sum_{k=1}^{n} \gamma_k^4 \|\psi_*(n,k+1)\|^4 \EE\|\be_k\|^4\\
\leq \kappa_1 n^{2( 1 + \zeta)} \sum_{k=1}^{n} \gamma_k^4  \exp\left( - 4L'\sum_{j=k}^{n} \gamma_j\right) \frac{1}{a(k)^2}.
\end{multline}
Thus, for some $\kappa_2 \in (0,\infty)$ we have that
$$
\sum_{k=1}^{n} \EE\| Z_{n,k}\|^4 \leq \kappa_2 n^{2(1 + \zeta) - 4L'\gamma_*} \sum_{k=1}^n k^{4(L'\gamma_* -1)} k^{-2\zeta},
$$
which tends to $0$ as $n\to \infty$. The result follows.
\end{proof}

The next lemma is used to characterize the limiting covariance matrix in Theorem \ref{thm:branchclt1}. Recall the matrix $U_*$ defined in (\ref{eq:ustar}).

\begin{lemma}\label{lem:newlimitvar}
Suppose that $\gamma_*> L^{-1}$.
Define
$$
V_n^{(1)} \doteq \sigma_n^2 \sum_{k=1}^{n} \gamma_k^2 \psi_*(n,k+1) \frac{U_*}{a(k)}\psi_*(n,k+1)^T,
$$
and
$$
V_n^{(2)} \doteq \sigma_n^2 \sum_{k=1}^{n} \gamma_k^2 \psi_*(n,k+1) \left[ \EE[ \be_k \be_k^T | \mathcal{F}_{k-1}] - \frac{U_*}{a(k)}\right]\psi_*(n,k+1)^T.
$$
As $n \to \infty$, $V^{(2)}_n \stackrel{\PP}{\to} 0$ and $V^{(1)}_n \stackrel{P}{\to} V$, where $V$ is the unique solution of the Lyapunov equation
\begin{equation}\label{eq:lyapunovpf1}
U_* + (1+ \zeta)\gamma_*^{-1} V + \nabla h(\theta_*) V + V\nabla h(\theta_*)^T = 0.
\end{equation}
\end{lemma}

\begin{proof}
As in the proof of Lemma \ref{lem:limitvar}, the Lyapunov equation (\ref{eq:lyapunovpf1}) has a unique solution, as $ U_*$ is nonnegative definite, and the matrix
\begin{equation}\label{eq:branchhurwitz}
\tilde{H} \doteq \nabla h(\theta_*) + (1 + \zeta)(2\gamma_*)^{-1}I,
\end{equation}
 is Hurwitz since
$
- L + (1 + \zeta)(2 \gamma_*)^{-1} < -L + \gamma_*^{-1} < 0.
$
We now consider $V^{(1)}_n$. 
Define $\tsigma_n^2 \doteq n^{-\zeta} \gamma_n$. Since $\tsigma_n/\sigma_n\to 1$ as $n\to \infty$,
it suffices to prove that
\begin{equation}\frac{\tsigma_n^{2}}{\sigma_n^2} V^{(1)}_n \doteq \tilde V^{(1)}_n \stackrel{P}{\to} V.\label{eq:eq906}
	\end{equation}
Observe that
\begin{equation}\label{eq:bnasy3}
 \frac{\tsigma_{n+1}^2}{\tsigma_n^2} - 1 = (1 + \zeta) \frac{\gamma_n}{\gamma_*} + o(\gamma_n), 
 \qquad
\frac{\tsigma_n^2}{a(n)} \gamma_n^2 = \gamma_n + o(\gamma_n),
\end{equation}
and therefore 
\begin{equation}\label{eq:bnasy4}
\frac{\tsigma_{n+1}^2}{\tsigma_n^2} \gamma_{n+1}^2 = \gamma_{n+1}^2 + o(\gamma_n^2),
\qquad
\frac{\tsigma_{n+1}^2}{\tsigma_n^2} \gamma_{n+1} = \gamma_{n+1} + o(\gamma_n).
\end{equation}
From \eqref{eq:psistarmatrixdef} it follows that
$$
\psi_*(n+1,k+1) = (I + \gamma_{n+1} \nabla h(\theta_*)) \psi_*(n, k+1),
$$
so we have that
\begin{equation}\label{eq:vn1expand1}
\begin{split}
\tilde V^{(1)}_{n+1}  &= \frac{\tsigma_{n+1}^2}{a(n+1)}\gamma_{n+1}^2 U_*  + \frac{\tsigma_{n+1}^2}{\tsigma_n^2} \big(I + \gamma_{n+1} \nabla h(\theta_*)\big) \tilde V_n^{(1)} \big(I + \gamma_{n+1} \nabla h(\theta_*)\big)^T\\
&= \frac{\tsigma_{n+1}^2}{a(n+1)}\gamma_{n+1}^2 U_* + \frac{\tsigma_{n+1}^2}{\tsigma_n^2} \Big( \tilde V_n^{(1)} + \gamma_{n+1}\Big[ \nabla h(\theta_*) \tilde V^{(1)}_n + \tilde V_{n}^{(1)} \nabla h(\theta_*)^T\Big]\\
&\hspace{14em} + \gamma_{n+1}^2 \nabla h(\theta_*) \tilde V^{(1)}_{n} \nabla h(\theta_*)^T\Big).\\
\end{split}
\end{equation}
Using the second identity in (\ref{eq:bnasy3}), 
\begin{equation}\label{eq:branchustarb1}
\frac{\tsigma_{n+1}^2}{a(n+1)} \gamma_{n+1}^2 U_{*}  =\gamma_{n+1}U_* + o(\gamma_n).
\end{equation}
Fix $L' \in (\gamma_*^{-1}, L)$. Then, from \eqref{eq:psistarmatrixdef} there is some $\kappa_1\in(0,\infty)$ such that
$$
\|\tilde V_n^{(1)}\| \leq \kappa_1 n^{1 + \zeta} \sum_{k=1}^{n} \gamma_k^{2}  \exp\left( - 2L' \sum_{j=k+1}^{n} \gamma_j\right)k^{-\zeta},
$$
and so we can find some $\kappa_2 \in (0,\infty)$ such that
\begin{equation}\label{eq:vn1bound}
\|\tilde V_n^{(1)}\| \leq \kappa_2, \;\; n \in \NN.
\end{equation}
 Combining the first identity in (\ref{eq:bnasy3}) with (\ref{eq:vn1bound}) we see that
\begin{equation}\label{eq:branchsubv}
\frac{\tsigma_{n+1}^2}{\tsigma_n^2} \tilde V^{(1)}_n - \tilde V^{(1)}_n = \left( \frac{\tsigma_{n+1}^2}{\tsigma_n^2}  - 1\right) \tilde V_{n}^{(1)} = (1+\zeta)\frac{\gamma_n}{\gamma_*} \tilde V_n^{(1)} + o(\gamma_n).
\end{equation}
Additionally, from  \eqref{eq:bnasy4} and  (\ref{eq:vn1bound}) we see that 
\begin{equation}\label{eq:vn1prodterm}
\frac{\tsigma_{n+1}^2}{\tsigma_n^2} \gamma_{n+1}^2  \nabla h(\theta_*) \tilde V^{(1)}_n \nabla h(\theta_*)^T = o(\gamma_n).
\end{equation}
Finally, noting that $\gamma_{n+1} = \gamma_n + o(\gamma_n)$, and combining (\ref{eq:bnasy4}), (\ref{eq:branchustarb1}), (\ref{eq:branchsubv}), and (\ref{eq:vn1prodterm}) with (\ref{eq:vn1expand1}) 
we see that
\begin{equation}\label{eq:eq924}
\begin{split}
\tilde V_{n+1}^{(1)} &= \tilde V_{n}^{(1)}  + \gamma_{n} \left[   U_* + \gamma_*^{-1} (1 + \zeta) \tilde V^{(1)}_n + \nabla h(\theta_*) \tilde V^{(1)}_n + \tilde V^{(1)}_n \nabla h(\theta_*)^T\right] + o(\gamma_n).\\
\end{split}
\end{equation}
Let $V$ denote the aforementioned unique solution to (\ref{eq:lyapunovpf1}). Recall from \eqref{eq:lyapunovpf1} that the matrix $\tilde{H}$ defined in (\ref{eq:branchhurwitz}) satisfies
$U_*+\tilde H V + V\tilde H^T=0$. Thus, \eqref{eq:eq924} can be rewritten as 
\begin{equation}\label{eq:vn1minv}
\begin{split}
\tilde V_{n+1}^{(1)} - V  
&= \tilde V_n^{(1)} - V + \gamma_n \left[ \tilde{H}(\tilde V_n^{(1)} - V) + (\tilde V_n^{(1)} - V) \tilde{H}^T\right] + o(\gamma_n).
\end{split}
\end{equation}
Recalling once more that $\tilde{H}$ is Hurwitz, it follows from (\ref{eq:vn1minv}) and the proof of \cite{for}*{Lemma 5.11} that $\tilde V_n^{(1)} \stackrel{\PP}{\to} V$ as $n \to \infty$. This proves \eqref{eq:eq906} and, as noted previously, shows that $V_n^{(1)} \stackrel{\PP}{\to} V$ as $n \to \infty$. The proof that $V_n^{(2)} \stackrel{\PP}{\to} 0$ as $n \to \infty$ is similar to the analogous result in Lemma \ref{lem:limitvar}, and is omitted.
\end{proof}

\begin{proof}[Proof of Theorem \ref{thm:branchclt1}]
From \eqref{eq:rhobranch1} we see that
$
\sigma_n (\btheta_n - \theta_*) = \sigma_n \bmu_n + \sigma_n \brho_n. 
$
Also, from Corollary \ref{cor:rhoto0}, $\sigma_n \brho_n \stackrel{\PP}{\to} 0$ as $n \to \infty$. Thus, it suffices to show that 
 $\sigma_n \bmu_n \stackrel{\mathcal{L}}{\to} \mathcal{N}(0, V)$ where $V$ is as in the statement of the theorem.
  Recall the martingale difference array $\{Z_{n,k}\}$ introduced in (\ref{eq:newxarray}), and note from \eqref{eq:newmux} that
$
\sigma_n \bmu_n = \sum_{k=1}^{n} Z_{n,k}.
$
In order to complete the proof we apply \cite{halhey}*{Corollary 3.1}. From Lemma \ref{lem:lindeberg} it follows that for each $\epsilon > 0$,
$$
\sum_{k=1}^n \EE \left[ \| Z_{n,k}\|^2 \mathbbm{1}_{\| Z_{n,k} \|\geq \epsilon} | \mathcal{F}_{k-1}\right] \stackrel{\PP}{\to} 0,
$$
as $n \to \infty$. Additionally, if we let $\{V^{(1)}_n\}$, $\{V^{(2)}_n\}$, and $V$ be as in Lemma \ref{lem:newlimitvar}, then we have that
$$
\sum_{k=1}^{n} \EE\left[ Z_{n,k} Z_{n,k}^T | \mathcal{F}_{k-1}\right] = V^{(1)}_n + V^{(2)}_n
$$
and from Lemma \ref{lem:newlimitvar},
$
\sum_{k=1}^{n} \EE\left[ Z_{n,k} Z_{n,k}^T | \mathcal{F}_{k-1}\right]  \stackrel{\PP}{\to} V,
$
as $n \to \infty$. Therefore, the conditions of \cite{halhey}*{Corollary 3.1} are satisfied, proving that
$
\sum_{k=1}^{n} Z_{n,k} \stackrel{\mathcal{L}}{\to} \mathcal{N}(0,V),
$
as $n \to \infty$. The result follows.
\end{proof}


\section{Numerical Results}\label{sec:numeric}

In this section we present results from some   numerical experiments.  We compare five simulation based methods for computing the QSD of a finite state Markov chain. 
The first four methods can be viewed as stochastic approximation algorithms and are described in terms of a sequence of step sizes given as
$$
\gamma_{n+1} \doteq \frac{\gamma_*}{n + N_*}, \;\; n \in\NN_0,
$$
where $\gamma_* \in (0,\infty)$ and $N_* \doteq \lfloor \gamma_*\rfloor + 1$. In order to ensure that the results from the various methods are comparable, we measure the run-time of each method by the total number of particle transitions. 

The first estimation method, which we refer to as the {\em Single scheme}, is the algorithm given  in \cite{benclo}*{Equation (7)}. In order to obtain an estimate for the QSD using this scheme, we run the algorithm for $na(n)$ time steps. Since there is a single particle, and it moves once at each time step, this means that there are a total of $na(n)$ particle movements. The second scheme, which we refer to as the {\em Independent scheme}, is given by evolving $a(n)$ Single schemes independently of one another. Each of these independent schemes runs for $n$ time steps, and the estimate for the QSD is then given by the average of  the $a(n)$ estimates. At each time instant, there are $a(n)$ particle movements, so the total number of particle movements is  $n a(n)$. The third scheme, which we refer to as the {\em Interacting scheme}, is the algorithm defined in (\ref{eq:stochalg}). In the notation of this work, our final estimate for the QSD is then given by $\theta^n_n$. As with the Independent scheme, since $a(n)$ particles move at each time instant, there are $na(n)$ particle movements in total.  The fourth scheme is the {\em Branching scheme}, which is described in \eqref{eq:branchalg1}. Note that for this scheme, by time instant $k$ there are a total of
$
\sum\limits_{i=1}^{k}a(i+1)
$
particle movements. Consequently, we run this scheme for $\xi(n)$ time steps, where
$$
\xi(n) \doteq \inf\left\{ k : \sum\limits_{i=1}^{k}a(i+1) \geq na(n)\right\}.
$$
The final method is the Fleming-Viot approximation. A description of this method and some important results regarding its convergence properties can be found in \cite{grojon}. In order to estimate the QSD using the Fleming-Viot approximation, we consider a collection of $a(n)$ particles that evolve according to the dynamics described in \cite{grojon}. At each time instant a particle is chosen uniformly at random to move, so after $n a(n)$ time steps, there have been $na(n)$ particle movements. The final estimate for the QSD is given by the empirical measure of the $a(n)$ particles at the $n a(n)$-th time instant. Our experimental results suggest that the first four methods all converge rapidly when the dynamics of the underlying Markov chain are  simple. For example, for the Markov chain whose transition matrix is given by 
\[
P = 
\begin{bmatrix}
1 & 0 & 0\\
0.2 & 0.5 & 0.3\\
0.3 & 0.3 & 0.4
\end{bmatrix},
\]
the rates of convergence of the various methods were comparable regardless of the distribution of the initial states of the particles in the systems. However, we find significant differences  when there are several points in the state space at which the Markov chain is expected to spend a relatively long time. With an abuse of terminology, for a Markov chain $\{Y_n\}$ on $\Delta$, we refer to a point $x \in \Delta^o$ as  a fixed point if $\EE_x(Y_1) = x$. We now consider an example of a Markov chain that has several fixed points.

Let $\{Y_n\}$ be the Markov chain on $ \Delta \doteq \{0,1,\dots, 9\}$ with transition probability matrix
\[
P = 
\begin{bmatrix}
1  &0 & 0 &0 & 0& 0& 0&  0&  0&  0\\
 0.2 & 0.1 & 0.7 & 0& 0& 0& 0&  0&  0&   0\\
 0 & 0.1 & 0.8 & 0.1 &0 &0& 0&  0 & 0 &  0\\
  0&  0 & 0.8 & 0.1 &0.1 & 0& 0 & 0 & 0 &  0\\
  0&  0&  0 & 0.8& 0.1& 0.1& 0 & 0&  0 &  0\\
    0 & 0&  0 & 0& 0.01 &0.98 &0.01&  0 & 0 &  0\\
   0&  0&  0 & 0 &0 &0.1& 0.1&  0.8 & 0 & 0\\
 0 & 0 & 0 & 0 &0 &0 &0.1&  0.1&  0.8&   0\\
   0 & 0 & 0&  0 &0 &0 &0 & 0.1 & 0.8  & 0.1\\
 0.2&  0&  0&  0& 0 &0 &0 & 0 & 0.7 &  0.1
\end{bmatrix}.
\]
Note that (in our terminology) $\{Y_n\}$ has three fixed points, namely, $2, 5$, and $8$.
We implemented the various QSD approximation methods discussed above for estimating the QSD of $\{Y_n\}$. 
Applying  \cite{benclo}*{Corollary 2.3} we see that $\gamma_* \approx 4.17$ satisfies $\gamma_* > L^{-1}$, where $L$ is as in Section \ref{sec:algdescription}. 

In our first experiment we take $a(n) = \lfloor n^{0.75} \rfloor$ and $n \doteq 1000$. We repeated the experiment for each scheme $R = 300$ times and averaged the results. For the Independent, Interacting, and Fleming-Viot schemes, the initial states of the $a(1000) = 177$ particles were chosen uniformly at random from $\{4,5,6\}$. This same set of initial states was used in each of the 300 repetitions of the simulation. Since the Single and Branching schemes are initialized with only a single particle, we chose the initial states of the 300 repetitions so that they would be proportionate to the initial states used for the schemes that start with $a(n)$ particles. In Figure \ref{fig:figure1} we plot the total variation distance between the estimate of the QSD given by each scheme and the true QSD as a function of the number of particle transitions. The results are plotted for the first 70,000 particle movements.

\begin{figure}
\includegraphics[width=100mm, scale=0.8]{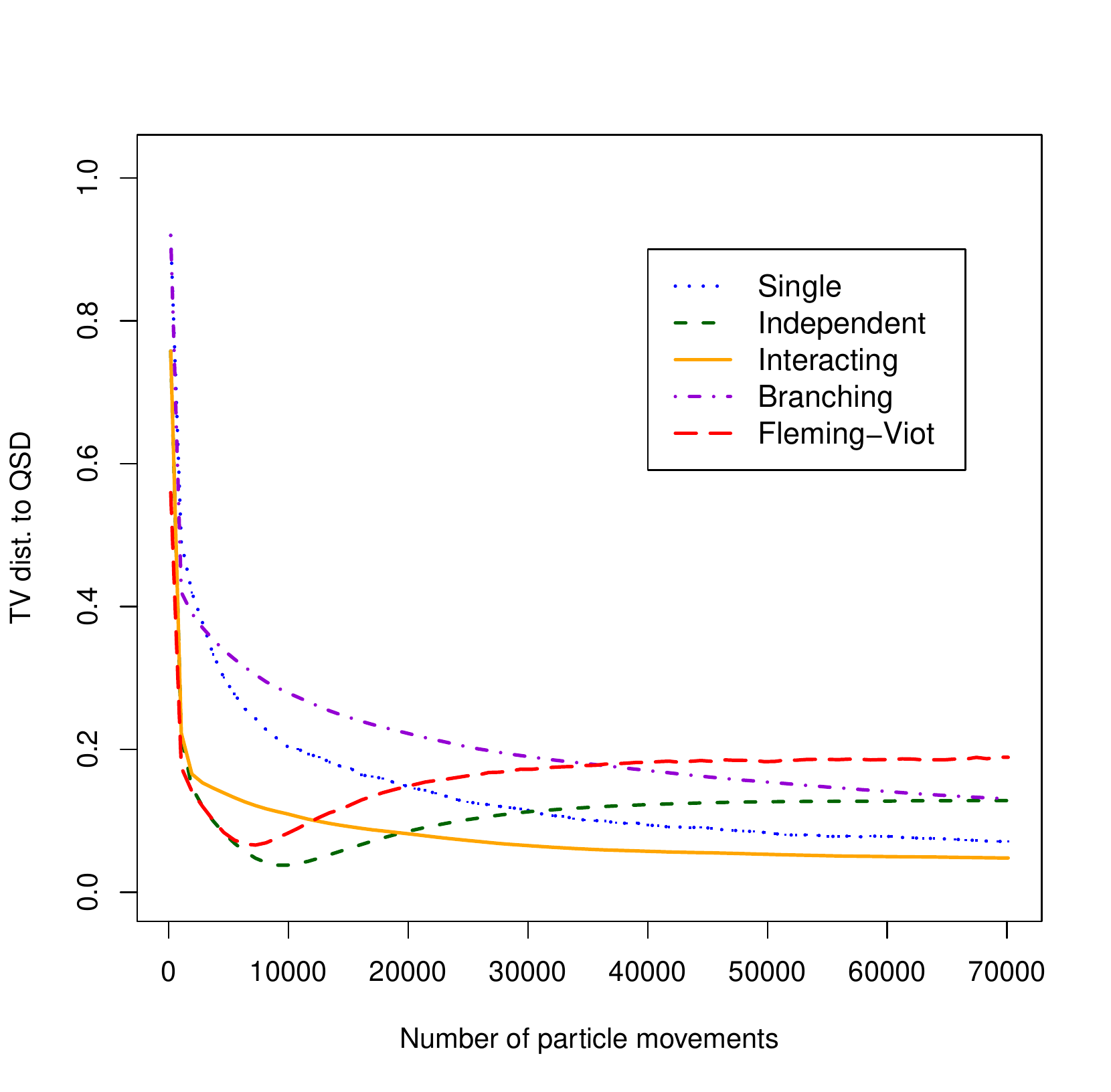}
\caption{Results for the first experiment: $a(n) = \lfloor n^{0.75} \rfloor$ and $n \doteq 1000$. The Interacting scheme converges most rapidly, and  the Single scheme converges more rapidly than the branching scheme.} 
\label{fig:figure1}
\end{figure}

Note that the Interacting scheme converges most quickly to the QSD in this experiment. The Fleming-Viot algorithm appears to have  a significant asymptotic bias, which is a consequence of the fact that the number of particles is not sufficiently large for the time asymptotic behavior of the Fleming-Viot processes to effectively approximate the QSD.
The experimental results when the initial states of the particles were chosen uniformly at random from $\{1,2,\dots,9\}$ were similar.

Our second experiment is on the same underlying Markov chain, but with $a(n) \doteq \lfloor n^{0.5} \rfloor$, and $n = 2000$. In this experiment we started every particle from $5$, which is the fixed point from which the Markov chain is expected to take the longest time to escape. While the Interacting scheme still performed best in this setting, we found that the Branching scheme performed better than the Single scheme. The results for the first 40,000 particle movements are plotted in Figure \ref{fig:figure3}.

\begin{figure}
\includegraphics[width=100mm, scale=0.8]{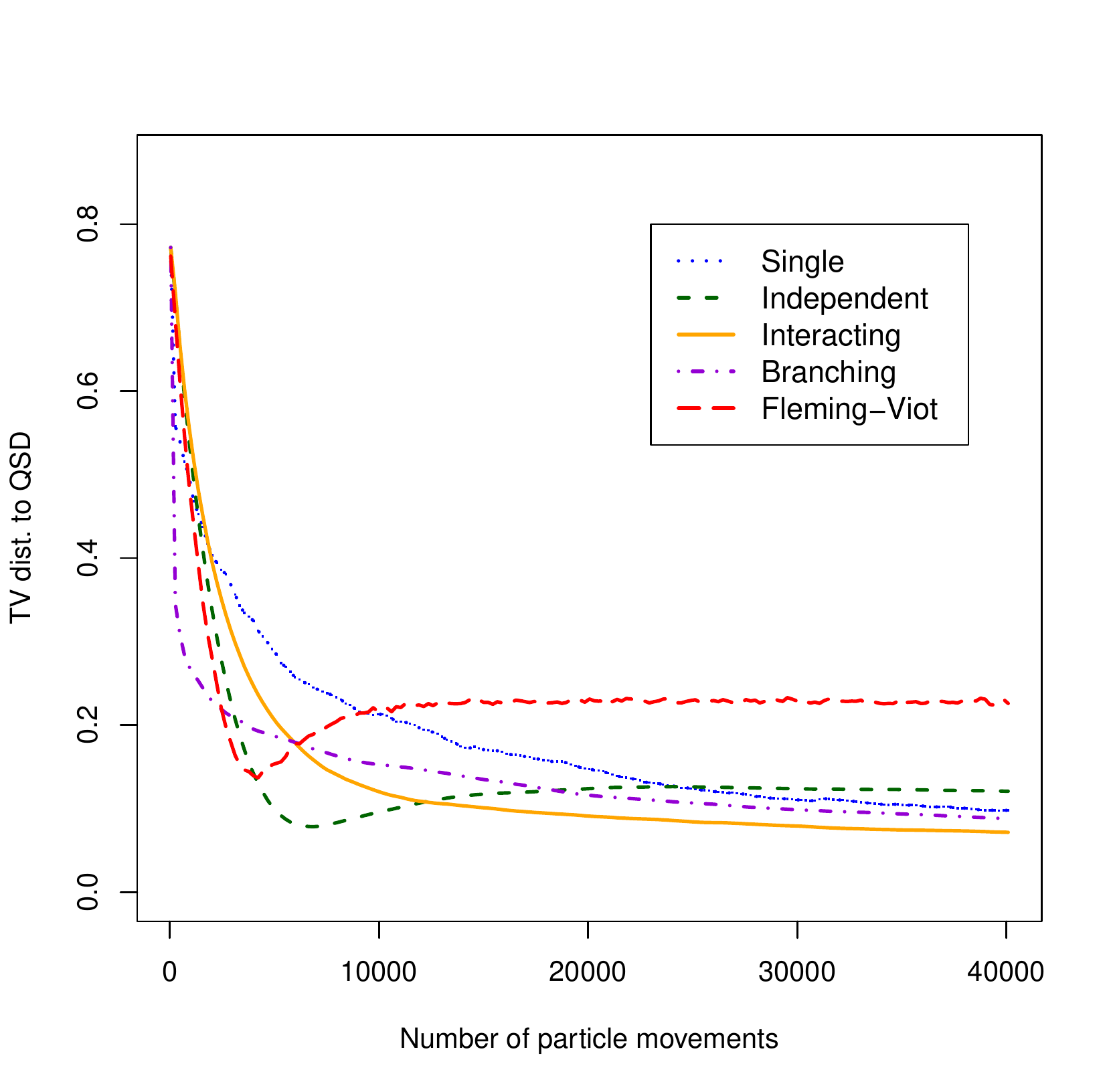}
\caption{Results of the second experiment: $a(n) = \lfloor n^{0.5} \rfloor$ and $n \doteq 2000$. While the Interacting scheme still converges most rapidly, the Branching scheme converges more rapidly than the Single scheme.} 
\label{fig:figure3}
\end{figure}

\appendix
\section{A Matrix Estimate}
\label{sec:auxres}
The following lemma is similar to \cite{for}*{Lemma 5.8}.
\begin{lemma}\label{lem:matrixlem1}
Let $A$ be a $d \times d$ Hurwitz matrix
such that the real part of all of its eigenvalues is bounded above by $-L$ where $L\in (0,\infty)$.
Fix $\po \in (0,1)$. Let $\{A^n_k\}_{n,k=1}^{\infty}$ be an array of matrices such that 
$\sup_{n^\po \le k \le n}\|A^n_k - A\| \to 0$, where $\| \cdot \|$ denotes the Frobenius norm on the space of $d \times d$ matrices.  For each $L' \in (0,L)$, there is a constant $C > 0$ such that if $n^{\po} \leq k \leq n$, then 
$$
\left\| \prod_{j=k}^{n}(I + \gamma_j A^n_j) \right\| \leq C \exp\left(-L' \sum_{j=k}^{n} \gamma_j\right).
$$
\end{lemma}

\begin{proof}
Let $\{\lambda_i\}_{i=1}^d$ denote the eigenvalues of $A$, and use the Jordan decomposition of $A$ to write $A = SJS^{-1}$, where $S$ is invertible and $J$ is a Jordan matrix. Let 
$$
D_t \doteq \text{diag}(t,t^2,\dots,t^d), \;\; \Lambda \doteq \text{diag}(\lambda_1,\lambda_2,\dots,\lambda_d).
$$
Then, following \cite{for}, we have
that
$
A = (SD_t) (\Lambda + R_t)(SD_t)^{-1},
$
where $\lim_{t\rightarrow0}\|R_t\| = 0$. Write
$$
(SD_t)^{-1} (I + \gamma_j A^n_j) (SD_t)= I + \gamma_j \Lambda + \gamma_j R_t + \gamma_j (SD_t)^{-1} (A_j^n - A)(SD_t).
$$
For $n^{\po} \leq k \leq n$, we have
\begin{equation}\label{eq:matrixeq1}
\begin{split}
\| (SD_t)^{-1}(I + \gamma_k A_k^n)(SD_t)\| &\leq \|I + \gamma_k \Lambda\| + \gamma_k \|R_t\| + \gamma_k \|A^n_k - A\|\|SD_t\|\|(SD_t)^{-1}\|.
\end{split}
\end{equation}
Fix $0 < L' < L'' < L$, and note that there is some $t_0 > 0$ such that if $0\le t \leq t_0$, then $\|R_t\| \leq (L'' - L')/2$.  Also, there is some $n_0$ such that if $n \geq n_0$, $n^{\po} \le k \le n$, then 
\begin{equation}\label{eq:matrixeq2}
\| I + \gamma_k \Lambda\| \leq 1 - \gamma_k L'', \;\; \|A^n_k - A\|\|SD_t\| \|(SD_t)^{-1}\| \leq (L''-L')/2.
\end{equation}
Combining (\ref{eq:matrixeq1}) and (\ref{eq:matrixeq2}), we  see that if $n^{\po} \leq k \leq n$, $t \leq t_0$, and $n \geq n_0$,
$$
\| (SD_t)^{-1} (I + \gamma_k A^n_k)(SD_t)\| \leq  1 - \gamma_k L'.
$$
It follows that
\begin{equation*}
\begin{split}
\left\| \prod_{j=k}^n (I + \gamma_j A^n_j)\right\| &= \left\| (SD_t) \left[\prod_{j=k}^n (SD_t)^{-1} (I + \gamma_j A^n_j)(SD_t) \right] (SD_t)^{-1}\right\|\\
&\leq \| SD_t\| \prod_{j=k}^n\| (SD_t)^{-1} (I + \gamma_j A^n_j)(SD_t)\| \|(SD_t)^{-1}\|\\
&\leq \|SD_t\| \|(SD_t)^{-1}\| \exp\left( -L' \sum_{j=k}^n  \gamma_j\right).\qedhere
\end{split}
\end{equation*}
\end{proof}

\section*{Acknowledgements}
Research  of AB is supported in part by the National Science Foundation (DMS-1814894 and DMS-1853968).


\bibliographystyle{amsplain}
\bibliography{main}

\end{document}